\documentclass[sn-mathphys,Numbered]{sn-jnl}

\usepackage{amsmath}
\usepackage{amssymb}
\usepackage{comment}
\usepackage{amsfonts}
\usepackage{epsfig}
\usepackage{graphicx}
\usepackage{hyperref}
\usepackage{xcolor}
\usepackage{todonotes,color} 
\usepackage{enumitem}
\usepackage{stmaryrd}
\usepackage{amsthm}
\usepackage{mathabx} 
\usepackage{multicol}
\usepackage{savesym}
\savesymbol{widering}
\usepackage[T1]{fontenc}
\usepackage[type1]{libertine}

\usepackage[libertine]{newtxmath}
\usepackage{bm}

\newtheorem{theorem}{Theorem}[section]
\newtheorem{proposition}[theorem]{Proposition}
\newtheorem{corollary}[theorem]{Corollary}

\newtheorem{lemma}[theorem]{Lemma}

\newtheorem{remark}[theorem]{Remark}

\def\N{\mathbb{N}}
\def\R{\mathbb{R}}
\def\C{\mathbb{C}}
\def\T{\mathbb{T}}
\def\Z{\mathbb{Z}}
\def\OO{\mathcal{O}}

\def\th{\theta}

\def\dd{\,d}

\def\fN{f^{N}}
\def\gN{g^{M}}

\def\hP{h^{P}}

\def\fNg{f^{\geq N+1}}
\def\gNg{g^{\geq M+1}}

\def\hPg{h^{\geq P+1}}

\def\fNm{\overline{f}^{N}}
\def\gNm{\overline{g}^{M}}

\def\hPm{\overline{h}^{P}}

\def\fNt{\widetilde{f}^{N}}

\def\hPm{\overline{h}^{P}}

\def\Cv{a_{V}}
\def\CIP{a_V^{\pa}}

\def\pa{\mathbf{p}}
\def\Qa{\mathbf{Q}}
\def\TT{\mathbf{w}}

\def\apa{a_{\pa}}
\def\bpa{b_{\pa}}

\def\Apa{A_{\pa}}

\def\CIP{a_V^{\pa}}

\def\CC{\mathbf{C}}

\def\u{u}
\newcommand{\m}{\mathfrak{m}}
\def\Mg{\mathcal{M}}
\def\M{M}
\def\v{\Theta}
\def\vf{\mathup{\theta}}
\newcommand{\cbf}{\mathbf{c}}

\renewcommand{\H}{\mathcal{H}}
\renewcommand{\L}{\mathcal{L}}
\newcommand{\NN}{\mathcal{N}}
\def\Cc{\mathcal{C}}
\def\Spec{{\rm Spec}\,}
\newcommand{\Id}{\mathrm{Id}}
\newcommand{\dist}{\mathrm{dist}}
\newcommand{\wh}{\widehat}
\newcommand{\wt}{\widetilde}

\def\Re{{\rm Re}\,}
\def\Im{{\rm Im}\,}
\def\Lip{{\rm Lip}\,}

\def\xu{u}
\def\K{\mathcal{K}}
\def\RR{\mathcal{R}}

\def\MM{\mathcal{M}}
\def\KK{\mathcal{K}}

\def\BB{\mathcal{B}}
\def\XX{\mathcal{X}}

\def\A{\mathbf{A}}
\def\B{\mathbf{B}}

\newcommand{\TTT}{\mathcal{T}}
\newcommand{\U}{\mathcal{U}}

\newcommand{\LL}{\mathcal{L}}

\newcommand{\F}{\mathcal{F}}
\newcommand{\eps}{\varepsilon}
\newcommand{\Kc}{K}

\newcommand{\Vext}{\mathbb{A}}
\newcommand{\Bext}{\mathbb{B}}
\numberwithin{equation}{section}

\DeclareRobustCommand{\mathup}[1]{\begingroup\changegreek\mathrm{#1}\endgroup}

\makeatletter
\def\changegreek{\@for\next:={%
  alpha,beta,gamma,delta,epsilon,zeta,eta,theta,kappa,lambda,mu,nu,xi,pi,rho,sigma,%
  tau,upsilon,phi,chi,psi,omega,varepsilon,vartheta,varpi,varrho,varsigma,varphi}%
  \do{\expandafter\let\csname\next\expandafter\endcsname\csname\next up\endcsname}}
\def\changegreekbf{\@for\next:={%
  alpha,beta,gamma,delta,epsilon,zeta,eta,theta,kappa,lambda,mu,nu,xi,pi,rho,sigma,%
  tau,upsilon,phi,chi,psi,omega,varepsilon,vartheta,varpi,varrho,varsigma,varphi}%
  \do{\expandafter\def\csname\next\expandafter\endcsname\expandafter{%
    \expandafter\bm\expandafter{\csname\next up\endcsname}}}}

\makeatother

\begin{document}
\title[Invariant manifolds of degenerate tori and  parabolic orbits to infinity]{
Invariant manifolds of degenerate tori and double parabolic orbits to infinity in the $(n+2)$-body problem}

\author[1,3]{\fnm{Inmaculada} \sur{Baldom\'a}}\email{immaculada.baldoma@upc.edu}

\author[2,3]{\fnm{Ernest} \sur{Fontich}}\email{fontich@ub.edu}

\author[1,3]{\fnm{Pau} \sur{Mart\'in}}\email{p.martin@upc.edu}

\affil[1]{\orgdiv{Departament de Matem\`atiques}, \orgname{Universitat Polit\`ecnica de Catalunya (UPC)}, \orgaddress{\street{Pau Gargallo, 14}, \postcode{08028} \city{Barcelona},   \country{Spain}}}

\affil[2]{\orgdiv{Departament de Matem\`atiques i Inform\`atica}, \orgname{Universitat de Barcelona (UB)}, \orgaddress{\street{Gran Via, 585}, 
\postcode{08007 Barcelona}, 
\country{Spain}}}

\affil[3]{
\orgname{Centre de Recerca Matem\`atica (CRM)}, \orgaddress{\street{Carrer de l'Albareda},  \postcode{08193} \city{Bellaterra},  \country{Spain}}}

\abstract{There are many interesting dynamical systems in which degenerate invariant tori appear. We give conditions under which these degenerate tori have stable and unstable invariant manifolds, with stable and unstable directions having arbitrary finite dimension. The setting in which the dimension is larger than one was not previously considered and is technically more involved because in such case the invariant manifolds do not have, in general,  polynomial approximations. As an example, we apply our theorem to prove that there are motions in the $(n+2)$-body problem in which the distances among the first $n$ bodies remain bounded for all time, while the relative distances between the first $n$-bodies and the last two and the distances between the last bodies tend to infinity, when time goes to infinity. Moreover, we prove that the final motion of the first $n$ bodies corresponds to a KAM torus of the $n$-body problem.} 

 



\maketitle

\tableofcontents

\section{Introduction}

\subsection{Parabolic invariant tori with stable and unstable invariant manifolds}

Consider, as a motivating example, the analytic local system of ordinary differential equations
\begin{equation}
\label{eq:Problema_de_Simo}
\left\{
\begin{aligned}
	\dot x & = f(x,y) \left(A_s x + X(x,y,\theta)\right), \\
		\dot y & = f(x,y) \left(A_u y + Y(x,y,\theta)\right), \\
		\dot \theta & = \omega + \Theta(x,y,\theta),
\end{aligned}
\right.
\end{equation}
where $(x,y) \in B \subset \R^n\times \R^m$, $B$ is a ball around the origin, $\theta \in \T^d = (\R/2\pi\Z)^d$, the matrices $A_s$ and $A_u$ satisfy $\mathrm{Spec}\, A_s$, $\mathrm{Spec}\, (-A_u) \subset \{z\in \C\mid\; \mathrm{Im}\, z <0\}$, $\omega \in \R^d$ is a Diophantine frequency vector, $X$, $Y$ are of order greater or equal than 2 with respect $(x,y)$, and $\Theta$ of order greater or equal than 1. Assume that $f$ has order $N$ in $(x,y)$, with $N\ge 0$. Under these hypotheses, the set $\TTT = \{x=0, \; y=0\}$ is an invariant torus of the system and the flow on $\TTT$ is a rigid rotation with frequency vector $\omega$.

If $f \equiv 1$, it is well known that $\TTT$ is an invariant hyperbolic torus with stable and unstable invariant manifolds, which are analytic graphs over $(x,\theta)$ and $(y,\theta)$, respectively.

Assume that $N\ge 1$. Then, the set $\TTT$, although still invariant, is no longer hyperbolic but degenerate. We will say that $\TTT$ is a \emph{parabolic} torus, as opposed to hyperbolic and elliptic. 
In this case, it is a non-trivial matter to establish the local behaviour of the system around $\TTT$. For instance, if $d = 0$, that is, if~\eqref{eq:Problema_de_Simo} does not depend on the angles $\theta$, the system, provided $f(x,y)\ne 0$, is equivalent to a system with a hyperbolic fixed point (by means of the rescaling of time $ds/dt = f(x,y)$) and, hence, it possesses \emph{formal} stable and unstable invariant manifolds, $y = \gamma^s(x)$ and $x = \gamma^u(y)$, in the sense  that $\gamma^{s,u}$ are formal series which are  invariant by~\eqref{eq:Problema_de_Simo}. However, if $d\ge 1$, $n \ge 2$, and $f$ is not a function depending only on $x$, it is not difficult to see that, in general, there is no formal stable manifold because, if one tries to find $y = \gamma^s(x,\theta)$ as a  series in $x$ with coefficients depending on $\theta$  invariant by~\eqref{eq:Problema_de_Simo}, formal obstructions appear. On the contrary, it is not difficult to see that, if $n=1$ or $f$  only depends on $x$, there is always a series representing the stable manifold, regardless of the dimension of the angles. 

Of course, the existence of a formal stable invariant manifold of $\TTT$ does not imply the existence of a true invariant one nor the formal obstructions  necessarily prevent the existence of a true invariant manifold. These questions, that is, if $\TTT$ in~\eqref{eq:Problema_de_Simo} possesses stable or unstable invariant manifolds and, in the case it does, what kind of regularity these manifolds have, were posed by Sim\'o in his 10th problem \cite{Simo2018}, were he remarked the formal obstructions that appear in the case $d\ge 1$ and $n \ge 2$. 

In the present work  we will consider a more general situation, namely, vector fields of the form
\begin{equation}
\label{camp_vectorial_intro}
X(x,y,\theta) = 
\begin{pmatrix}   f^N(x,y,\th,\lambda)  + \OO(\|(x,y)\|^{N+1})   \\
 g^{M} (x,y,\th,\lambda) + \OO(\|(x,y)\|^{M+1}) \\
\omega  +  h^{\geq P}  (x,y,\th,\lambda ) 
\end{pmatrix},
\end{equation}
where $f^N$, $g^M$, and $h^P$ are functions of orders $N$, $M$, and $P$ in $(x,y)$, respectively. Here, the set $\TTT$ is also invariant by the flow of $X$. We will provide a set of assumptions under which $\TTT$ has a stable invariant manifold. 
For the unstable manifold one simply has to consider the reversed time vector field.
Observe that equation~\eqref{eq:Problema_de_Simo} is a particular case of this type of vector fields.

It is important to remark that equation~\eqref{eq:Problema_de_Simo}, although degenerate, appears in many interesting problems. The fact that in many cases $\TTT$ possesses stable and unstable invariant manifolds, has important consequences in the global dynamics of the corresponding systems. 
Actually, we will deal, more generally, with a quasiperiodic non-autonomous version of 
\eqref{camp_vectorial_intro}.

One of the first important examples is the Sitnikov problem~\cite{Sitnikov60,Moser01}, a particular instance of the restricted 3-body problem. In some special coordinates, the Sitnikov problem can be written in the form~\eqref{eq:Problema_de_Simo} with $n=1$, $d=1$, and $f(x,y) = (x+y)^3$. 
McGehee~\cite{McGehee73} proved an existence result of analytic (out of the fixed point) stable manifolds for two dimensional maps  
which implies the existence of an analytic stable manifold for $\TTT$. 
A generalization of this statement for $C^k$ maps providing one dimensional stable manifolds in arbitrary dimension was carried out in~\cite{BFdLM2007}, using the parametrization method. Besides the Sitnikov problem, the restricted planar 3-body problem, either circular or elliptic~\cite{SimoL80,DelshamsKRS12,GuardiaMS14,GuardiaMSS17}, or the planar 3-body problem~\cite{GuardiaMPS22} can be written in the form~\eqref{eq:Problema_de_Simo} with $n=1$, $d=1$, and $f(x) = (x+y)^3$, with important dynamical consequences. Indeed, in all these works, devoted to show the existence of either chaotic and oscillatory motions or diffusion phenomena, one of the key ingredients of the proof is the existence of invariant manifolds of certain parabolic fixed points or periodic orbits at infinity and their analytic dependence with respect to several parameters. See also \cite{Don33}  for a different approach to parabolic tori in celestial mechanics. Parabolic points with invariant manifolds can also be found in problems in economics (see~\cite{Lee2021,Fefferman2021}). In this last case, $n=1$, $d=0$, and $f(x) = x$.

The approaches in \cite{McGehee73,BFdLM2007} required $n=1$ and $d=1$, that is, they only work if the stable invariant manifold for the stroboscopic return map is one dimensional. The generalization for $d\ge 2$ but keeping $n=1$ was carried out in~\cite{BFM20}, with implications in the general $n$-body problem, which, in certain parts of the phase space, can be written in the form~\eqref{eq:Problema_de_Simo} with $n=1$, $d=2n+2$, and $f(x) = x^3$. In this case, $\TTT$ in~\eqref{eq:Problema_de_Simo} admits a formal stable invariant manifold as a power series in $x$ with coefficients depending on $\theta$, which is used as a seed in the parametrization method.

Studying parabolic fixed points with stable invariant manifolds of dimension larger than one with the parametrization method is more involved. The reason is that, unlike the previous cases, if the dimension of the invariant manifolds is larger than one, in general they do not admit a Taylor expansion at the fixed point. To overcome this difficulty, it was shown in~\cite{BFM2020a,BFM2020b} that, for vector fields of the form~\eqref{camp_vectorial_intro} with $d=0$, under suitable hypotheses, they admit expansions as sums of homogeneous functions of increasing order. Having in mind some applications to celestial mechanics (see section~\ref{introducio:mecanica_celest}), in the present work we extend the results in~\cite{BFM2020a,BFM2020b} to parabolic tori. 

\subsection{Degenerate tori and homogeneous functions}

The purpose of the present paper is twofold. On the one hand, we present a general theorem which, under suitable conditions, provides the existence of invariant manifolds of the invariant torus $\TTT$ for vector fields of the form~\eqref{camp_vectorial_intro} (and for maps with equivalent conditions).  On the other, we show the existence of new type of orbits in the $N$-body problem, defined for all time either in the future or in the past, with a prescribed final behaviour. We call these orbits \emph{double parabolic orbits to infinity}. See Section~\ref{introducio:mecanica_celest} for an accurate description of these motions.

The conditions we impose on the vector field~\eqref{camp_vectorial_intro} are placed in Section~\ref{sec:setupmap} (they are completely analogous for maps and for flows). Of course, since the linearization of the vector field at $\TTT$ vanishes identically, they have to involve several terms of the jet of the vector field at the torus. In fact, they only involve the \emph{first} non-vanishing terms of the jet of the $(x,y)$-components of the vector field at the torus, plus a very mild condition on the angular directions. In particular, they imply the existence of a weak contraction in the $x$-direction and a weak expansion in the $y$-direction, but some other requirements are also needed.

We apply the parametrization method~\cite{CabreFL03a,CabreFL05,HaroCFLM16} to find the invariant manifolds of $\TTT$ in~\eqref{camp_vectorial_intro}. 
The main differences among the results in the present paper and those in~\cite{BFM2020a,BFM2020b} are the following. 

First, instead of considering parabolic fixed points, here we consider parabolic tori. This is a non-trivial extension that widens the field of application of the results. We are interested in particular in the case where the dynamics on the manifold synchronizes with the one on~$\TTT$. This fact, that always happens in the hyperbolic case, may not occur in the parabolic one. Our theorem is also valid even when this synchronization does not take place, and we give conditions under which it happens. In this sense, we improve the results in~\cite{BFM20}, where only the cases where the synchronization occurs where considered. One of the consequences of synchronization is that then the invariant manifolds are foliated by the stable leaves of the points in the torus and this foliation is regular in the base.

Second, we do not require the vector field to be defined in a whole neighborhood of the torus, not even at a formal level. We only require some kind of regularity in sectorial domains with the torus at their vertex, expressed in terms of homogeneous functions. We do require the leading terms to be defined and regular around the torus, although we believe that this requirement may be relaxed and we impose it for convenience, since it holds in the examples we consider.

Third, we consider only the analytic case. The only reason is to simplify the proof. We believe that the arguments in~\cite{BFM2020a,BFM2020b} to deal with the $C^k$ case can be adapted here, but they are rather cumbersome and the applications we consider are analytic.

The existence of the manifolds is formulated as an \emph{a posteriori} result, that is, in Theorem~\ref{thm:posterioriresult}, for maps, or Theorem~\ref{thm:posterioriresultflow}, for flows, we show that, if the invariance equation~\eqref{invarianceequationmap}, in the case of maps, or~\eqref{inveqflows}, in the case of flows, admits an approximate solution as sum of homogeneous functions of increasing order up to some specified order, then it has a true analytic solution. Separately, Theorem~\ref{thm:approximationmaps} (Theorem~\ref{th:existenceflow}, in the case of flows) provides such approximation. We emphasize that, in general, there is no polynomial approximate solution of the invariance equations~\eqref{invarianceequationmap} or~\eqref{inveqflows} since formal obstructions appear. Obtaining this approximate solution is a non-trivial task. Finally, Theorems~\ref{th:existencemap} and~\ref{th:existenceflow} simply join the \emph{a posteriori} and the approximation results into an existence result, to ease their application in practice.

Theorems~\ref{thm:posterioriresult} and~\ref{thm:approximationmaps} apply in the case the involved maps have the form
\[
F_\lambda \begin{pmatrix} x \\ y \\ \th
\end{pmatrix}
=
\begin{pmatrix} x  +   f^{\geq N} (x,y,\th,\lambda)   \\
y  + g^{\geq M} (x,y,\th,\lambda) \\
\th +\omega  +  h^{\geq P}  (x,y,\th,\lambda ) 
\end{pmatrix}.
\]
We add Corollary~\ref{cor:rootsunity}, which applies to maps of the form
\[
G_\lambda(x,y,\th)= \left (\begin{array}{c} \A x + f^{\geq N}(x,y,\th,\lambda) \\ \B y + g^{\geq M}(x,y,\th,\lambda) \\ \th + \omega + h^{\geq P}(x,y,\th,\lambda) \end{array}\right ) ,\qquad  \Spec\A,\;\Spec \B \subset 
\bigcup_{k\in \Z} \{z\in \C \mid\, z^k=1\}.
\]
These kind of maps appear in~\cite{Lee2021,Fefferman2021}, where a certain economic model based on critical values is considered.  

\subsection{Double parabolic orbits to infinity in the $(n+2)$-body problem}
\label{introducio:mecanica_celest}

We present an application of Theorem~\ref{th:existenceflow} to celestial mechanics, more concretely, to obtain new types of solutions of the full planar $N$-body problem. In the present paper, by direct application of Theorem~\ref{th:existenceflow}, we show that the set of \emph{double parabolic orbits to infinity} contains manifolds of certain dimension. As far as we know, these solutions have not been previously found. They are defined either for all future or all past time, avoiding collision and non-collision singularities. Further analysis, completely beyond the scope of the present paper, could lead to the existence of solutions that combine both of them, from the past to the future. The existence of solutions of the $n$-body problem combining prescribed final motions in the past and the future is an important question that has been addressed with different techniques in different instances of the problem (see, amongst others,~\cite{Chazy22,Sitnikov60,Moser01,LlibreS80,GuardiaMS14,BarCT21,CapGMSZ22,GuardiaMPS22}).

In a precise way, here \emph{double parabolic orbits to infinity} means the following. Consider the planar $(n+2)$-body problem, with $n\ge 1$. Denote by $Q_0$ the cluster of the first $n$ masses and by $q_0$ the position of their center of mass in some inertial system of reference. Let $q_n$ and $q_{n+1}$ be the positions of the last two bodies. Let $p_0$, $p_n$ and $p_{n+1}$ be their corresponding momenta. Denote by $d_k$  the distance between $q_0$ and $q_k$, $k=n,n+1$, and by $d_0$, the distance between $p_n$ and $p_{n+1}$. Assume, for the moment, that these three distances are infinite, while their momenta $p_0 = p_n = p_{n+1} = 0$. We prove that, in some coordinates, the vector field describing the $(n+2)$-body problem is regular around this configuration. We remark that, in this configuration, the relative positions of $q_0$, $q_1$, and $q_2$ are not free. They are described in this section, below. When the three clusters are at infinity with zero momenta, the motion of the bodies in $Q_0$ is described by an $n$-body problem. It is well known that KAM tori exist in the $n$-body problem~\cite{Arnold63b,Fejoz04,ChierchiaP11}. We choose \emph{any} of those KAM tori. In these regularized variables, the configuration in which the chosen KAM tori and the other two masses are at infinity is a regular invariant torus with dynamics conjugated to a Diophantine rotation. The vector field has the form~\eqref{camp_vectorial_intro}. Our aim is to find invariant manifolds of solutions that tend either in the past or in the future to this invariant torus. 

It is well known, however, that any solution of the $(n+2)$-body problem in which the three clusters arrive to infinity with parabolic velocity must tend to a central configuration of the $3$-body problem for $q_0$, $q_1$ and $q_2$~\cite{MarS76} (see also~\cite{SaaH81,MadV09,BosDFT21}), that is, either the relative positions of the three clusters tend to an equilateral triangle or to a collinear configuration, which only depends on the masses of the bodies. See Figure~\ref{fig:configuracions_centrals}. This is not the case when the limit velocities are hyperbolic~\cite{MadV20}.

\begin{figure}[h]
\begin{center}
\includegraphics[width=.4\textwidth,bb=0in -1.5in 5in -1in]{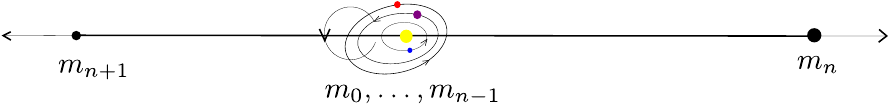}
\hspace{1.2cm}
\includegraphics[width=.4\textwidth]{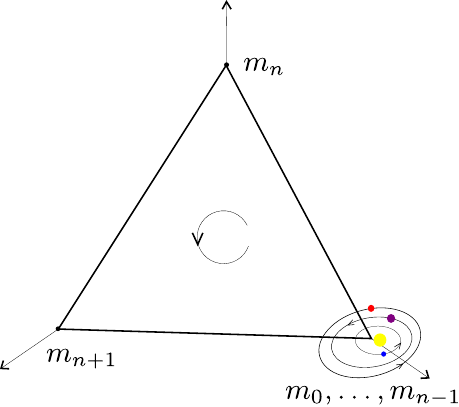}
\end{center}
\caption{Tending to collinear and equilateral  configurations.}
\label{fig:configuracions_centrals}
\end{figure}

Let $m_0, \dots, m_{n+1}$ be the (non-zero) masses  of the planar $(n+2)$-body problem. 
Let $m_j$, $0\le j\le n-1$, 
be fixed and assume that $ m_n,m_{n+1}$ are small enough.

We recall that the planar $(n+2)$-body problem admits a Hamiltonian formulation (see~\eqref{def:Hamiltoniancartesian} for the Hamiltonian formulation and, in general, Section~\ref{sec:coordenadesdeJacobi} for the actual description of the problem and the coordinates we use). It has three classical first integrals, besides the energy, namely, two corresponding to the total linear momentum and one to the total angular momentum. 
Fix any fixed value of the total linear momentum (that can be assumed to be 0), any value of the total angular momentum, and reduce the problem by these integrals. The reduced problem has $2n+1$ degrees of freedom. In the reduced system, we consider three clusters of masses: the first one, containing masses $m_0$ to $m_{n-1}$, and the second and third ones, containing  the masses $m_n$ and $m_{n+1}$, respectively.

Consider the following ``central configurations'' of the planar $3$-body problem:
\begin{enumerate}
    \item[(E)] an equilateral triangle, with a cluster in each vertex,
    \item[(C)] a collinear configuration, where the first, more massive, cluster lies between the two lighter ones.
\end{enumerate}
In the case of the first cluster, which involves several bodies, to be on a vertex means that the center of mass of the cluster lies on the vertex. In the case (E), there is only one of such configurations, modulo permutation of the vertices. The case (C), modulo permutation of the lighter bodies, there is also a single one.

Both in the cases (E) and (C), when the mutual distances of the clusters are infinite and the momenta of each cluster are $0$, the motion of the bodies in the first cluster is described by a $n$-body problem after reduction of the total linear momentum. Let $\TTT$ be a KAM torus of this $n$-body problem, with Diophantine frequency $\omega$. It has dimension $2(n-1)$. Observe that $\TTT$ does not depend on the masses $m_n$, $m_{n+1}$. We call $\TTT_E$ and $\TTT_C$ the invariant torus of the $(n+2)$-body problem where the first cluster evolves in $\TTT$, while the three clusters are in either (E) or (C) configuration, at infinity with zero momentum.

\begin{theorem}
\label{thm:solucionsdoblementparaboliqueainfinitv1}
If $m_n$ and $m_{n+1}$ are small enough but both different from $0$, with the smallness condition only depending on $M_0 = \sum_{k=0}^{n-1} m_k$, the following holds.
\begin{itemize}
    \item $\TTT_E$ possesses $3+2(n-1)$ dimensional stable and unstable manifolds, $W_E^{u,s}$, that can be parametrized by some variables $(u,\varphi) \in V\times \T^{2(n-1)} \subset \R^3 \times \T^{2(n-1)}$, $V$ being some sectorial domain in $\R^3$ with the origin in its vertex,  and such that the $\varphi$-dynamics is given by $\dot \varphi = \omega$.
    \item $\TTT_C$ possesses $2+2(n-1)$ dimensional stable and unstable manifolds, $W_C^{u,s}$, that can be parametrized by some variables $(u,\varphi) \in V\times \T^{2(n-1)} \subset \R^2 \times \T^{2(n-1)}$  and such that the $\varphi$-dynamics is given by $\dot \varphi = \omega$.
\end{itemize}
\end{theorem}

Theorem~\ref{thm:movimentsdoblementparabolicsv2} is a rewording of Theorem~\ref{thm:solucionsdoblementparaboliqueainfinitv1}, expressed in appropriate coordinates, after the explicit reduction by the total linear momentum and the total angular momentum of the system is done. This reduction is performed in Section~\ref{sec:coordenadesdeJacobi}. Later on,
in Section~\ref{sec:tordiofantic}, we introduce the quasiperiodic solutions, which correspond to trajectories on invariant tori of the $n$-body problem.

Theorem~\ref{thm:solucionsdoblementparaboliqueainfinitv1} assumes that the masses of the last two bodies are small, 
but different from~$0$. It provides the existence of an invariant manifold of solutions tending to parabolic motions in a collinear  configuration where the cluster of more massive bodies is between the last two, and to an equilateral configuration, respectively. There is still another possible final configuration, the remaining collinear case,
in which the cluster of more massive bodies moves to infinity in one direction while the last ones go to infinity in the other one. Our current proof does not cover this case, although we believe it could be extended, with additional effort, to include it.

We assume that the masses of the last two bodies are small. In doing so, roughly speaking, the problem becomes 
perturbative, since the interaction between the large cluster with each of the small masses is $\OO(m_n,m_{n+1})$ while  the interaction between the last masses themselves is  $\OO(m_nm_{n+1})$. 
However, the coupling between the small masses is crucial and the existence of the manifolds strongly depends on the non-vanishing
of a coefficient of the perturbation. If the masses are small, this non-degeneracy can be easily checked. The sign of the coefficient is different for $\TTT_E$ and $\TTT_C$, the two configurations we consider, which is the reason why the corresponding invariant manifolds have different dimension. If the masses are taken larger, bifurcations may occur (as happens, for instance, for the Lagrange points $L_4$ and $L_5$ of the restricted $3$-body problem). We have not pursued in this direction, but we believe that Theorems~\ref{thm:posterioriresultflow}
and~\ref{thm:approximationflows} can be applied even if the masses $m_n$ and $m_{n+1}$ are not small. This seems feasible because there are only three clusters and the number of central configurations in the $3$-body problem is well established. One could also consider the problem of more than two masses going to infinity in a parabolic fashion.

It is also worth to remark that, since the existence of $W^{u,s}_E$ and $W^{u,s}_C$ is a consequence of Theorems~\ref{thm:posterioriresultflow}
and~\ref{thm:approximationflows}, parametrizations of them can be approximated by sums of analytic homogeneous functions of increasing order. In some instances
of the $3$-body problem (see~\cite{BFM2020a}), these homogeneous functions are indeed homogeneous polynomials. Then, the question of the Gevrey regularity of these expansions makes sense. This was studied in a lower dimensional problem in~\cite{BFM17}. We conjecture that the invariant manifolds in the present setting also admit polynomial approximations which are Gevrey of a certain class.

Finally we remark that, in the case of the planar $3$-body problem, that is, $n=1$ in our setting, $\TTT$ is a single parabolic point and the configurations $\TTT_E$ and $\TTT_C$ are the well known central configurations of the problem. After the reductions, the planar $3$-body problem is a $3$-degrees of freedom Hamiltonian. Then, our theorem implies that $\TTT_E$ possesses $3$-dimensional stable and unstable manifolds, which both lie in the same $5$-dimensional energy level. These manifolds intersect at least along a homoclinic orbit provided by the homographic solution given by the central configuration.

\subsection{Structure of the paper}

In Section~\ref{sec:teoremes_principals} we introduce the notations and definitions we will use along the paper, as well as the statements of the general theorems. We provide different statements for maps and flows to ease their application, although the claims for flows are deduced from the ones for maps.

Section~\ref{sec:resultats_a_posteriori} is devoted to the proof of the \emph{a posteriori} claims, that is, assuming that a suitable approximate solution of some invariance equation is known, we prove the existence of a true solution. The statements are proven through a fixed point scheme.

Section~\ref{sec:approximation} contains the construction of the approximate solutions of the corresponding invariance equation. As we have already mentioned, these solutions are not polynomial but sums of homogeneous functions of increasing order in certain variables. Notwithstanding, the solutions are given through explicit formulas.

Section~\ref{sec:problemadenmes2cossos} contains the proof of the existence of double parabolic motions to infinity in the $(n+2)$-body problem. It is done by finding suitable coordinates, which include a normal form procedure and blown-up, in which the general theorem applies.

\section{Invariant manifolds of normally parabolic invariant tori}
\label{sec:teoremes_principals}

The first goal of this section is to introduce the main notation and conventions we use along the work. This is done in Section~\ref{sec:notacio}. In Section~\ref{sect:Diophan} we enunciate the small divisors lemma we extensively use along the paper. 

The remaining sections are devoted to state the main results of this work. Section~\ref{sec:casdemaps} deals with the case of the existence of local stable manifolds associated to invariant normally parabolic tori
for analytic maps and Section~\ref{sec:casdefluxos} is devoted to the case of analytic vector fields depending quasiperiodically on time also having  an invariant normally parabolic tori. 

In both settings we present four types of results: the so-called \textit{a posteriori result} (Theorems~\ref{thm:posterioriresult} and~\ref{thm:posterioriresultflow}), an approximation result (Theorems~\ref{thm:approximationmaps} and~\ref{thm:approximationflows}), 
an existence result of local stable manifolds, which is a direct consequence of the previous ones (Theorems~\ref{th:existencemap} and~\ref{th:existenceflow}) 
and finally a conjugation result, Corollaries~\ref{cor:conjugationmaps} and~\ref{cor:conjugationflows}.

\subsection{Notation and a small divisors lemma} \label{sec:notation}

\subsubsection{Notation}\label{sec:notacio}
In this section we introduce the notations and conventions we will use without explicit mention along the paper. Most of them are widely used in the literature and were already used in the previous works~\cite{BFM17,BFM2020a,BFM2020b}. However, for the convenience of the reader, we reproduce them here.  

The general notation about the sets we will use is:
\begin{itemize}
\item 
We denote $B_r$ the open ball of a Banach space $E$ of radius $r$ centered at the origin. We will write $B_r\subset E$ to indicate that $B_r$ is a ball in the space $E$. Given a set $U\subset E$, we denote $\overline{U}$ its closure. 
\item 
When we write $\R^n\times \R^m$, and we have norms in $\R^n$ and $\R^m$, we consider the product norm in it, namely $\|(x,y)\| = \max\{\|x\|, \|y\|\}$. This determines the operator norms for linear maps in these spaces. All these norms will be denoted by $\| \cdot \|$.  
\item 
    Real and complex $d$-torus: we represent the real torus by $\mathbb{T}^d=\big (\mathbb{R}\setminus \mathbb{Z}\big )^d$. Given $\sigma>0$, a complex extension is
    $$
    \mathbb{T}^d_\sigma = \left \{ \theta =(\theta_1,\cdots, \theta_d) \in  (\mathbb{C}\setminus \mathbb{Z}\big )^d\mid\; |\Im \theta_j |< \sigma, \forall j\right \}.
    $$
    \item Given an open set $U\subset \R^k$, we denote by $U_\mathbb{C}$ an open complex extension of it. 
    \item Given a function $f:U\subset \R^k \to \R^l$ and $x\in U$, $Df(x)$ denotes its derivative (or differential)  and, for a function $f(x,y)$,  $f:U\subset \R^k \times \R^{k'}\to \R^l$  $\partial_x f(x,y)$ 
    or $D_x f(x,y)$ denote its partial derivative with respect to the variable $x\in \R^k$, etc. 
\end{itemize}
With respect to  averages, we introduce the following notation:
\begin{itemize}
    \item For a function $f:U\times \T^d  \subset \R^k \times \T^d  \to \R^l$, we denote by $\overline{f}$ its average with respect to $\th  \in \T^d$ and  $\widetilde{f} = f - \overline{f}$ its oscillatory (mean free) part. In Section \ref{sec:problemadenmes2cossos} we will also use the notation $[f]= \overline{f} $.
    \item We say that a function $f(x,\theta, t)$, $f: U\times \T^d\times \R \to \R^l$ is quasiperiodic with respect to $t\in \R$ if there exists a function     $\widehat{f} :U\times \T^d\times \T^{d'} \to \R^l$, for some $d'$ and a vector $\nu \in \R^{d'}$, such that 
    \begin{equation}\label{sec:notquasi}
    f(z,\th,t)= \widehat{f}(z,\th, \nu t).
    \end{equation}
    We say that $\nu$ is the time frequency of $f$. 
    \item If $f$ is a quasiperiodic function, and $\widehat{f}$ satisfies~\eqref{sec:notquasi}, the average of $f$, denoted by $\overline{f}$, is the average of $\widehat{f}(z,\th,\th')$ with respect to $(\th,\th')\in \T^d \times \T^{d'}$. In the same way, the oscillatory part is $\widetilde{f} = f - \overline{f}$. 
    \item We say that a quasiperiodic function $f$ is analytic if $\widehat{f}$ is.
\item We will use the analogous definitions if the functions depend on parameters, considering the corresponding functions defined on $U\times \T^d \times \Lambda $
or $U\times \T^d \times\T^{d'} \times \Lambda $, with $\Lambda \subset \R^p$. 
\item Also, we will use the analogous definitions for the complex extensions of the involved functions. 
\end{itemize}
Next, we enumerate some general conventions we will use:
\begin{itemize}  
\item We will denote $\Mg>0$ a generic constant, that can take different values at different places.  
    \item We will omit the dependence of the functions on some of the variables whenever there is no danger of confusion, mainly the dependence on parameters. 
    \item Given $f:U\times \T^d \times \Lambda \subset \R^k \times \T^d \times \R^p \to \R^l$  we will denote by $f^{(k)}$ its $k$-Fourier coefficient, namely 
    $$
    f(z,\th,\lambda)=\sum_{k\in \Z^d } f^{(k)}(z,\lambda) e^{2\pi ik \cdot \th}, \qquad k \cdot \th =k_1 \th_1 + \cdots + k_d \th_d.
    $$
   
    \item Given $f(z,w)$, $f:U \times W \to \R^l$, $0\in \overline U$, where $W$ is some set, we will write $f(z,w)=\mathcal{O}(\|z\|^k)$,  $f=\mathcal{O}(\|z\|^k)$ or simply $f=\mathcal{O}_k$ if and only if there exists a constant $\Mg$ such that $\|f(z,w)\|\leq \Mg \|z\|^k$ for all $w\in W$ and $z\in U\cap B_1$.
    
    \item  For functions $f(z,\theta, \lambda)$, 
    $z\in \R^k$, $\theta \in \T ^d$, $\lambda \in \R^p$, we use the convention that the superscript in the function, $f^{l}$, indicates that $f^l$ is a homogeneous function of degree $l$ with respect to $z$. We will write 
    $f^{\ge l}$ if $f^{\ge l}=\mathcal{O}_l$.
    \item If $(x,y,z) \in \R^k\times \R^l \times \R^m$ and $f$ is a function taking values on 
    $\R^k\times \R^l \times \R^m$, we will denote by $f_x,f_y,f_z$ the corresponding projections over the subspaces generated by the variables $x,y,z$, respectively. We will also use the notation $f_{x,y}=(f_x,f_y)$ and the analogous notation for other combinations of the variables.  
    \item When $\lambda$ is a parameter and the  composition $f(z,\lambda)= h(g(z,\lambda),\lambda)$
   makes sense,  we will write $f=h \circ g$. When dealing with time dependent vector fields, for notational purposes, the time $t$ will be considered as a parameter. 
 \item We will denote $\Phi_Z(t;t_0,z,\lambda)$ the solution  of the differential equation $\dot{z}=Z(z,t,\lambda)$. 
\end{itemize}

\subsubsection{Diophantine vectors and small divisors lemmas}
\label{sect:Diophan}
We recall the definition of Diophantine vector and the so-called small divisors equation in both the map and the differential equation contexts.

In the map setting, 
$\omega \in \R^d$ is Diophantine if there exist $c> 0$ and $\tau \geq d$ such that for all $k \in \Z^{d}\backslash \{0\}$ and $l \in \Z$
    $$|\omega \cdot k - l |\geq c |k |^{-\tau},$$
    where $|k|= |k_1|+ \cdots + |k_d|$ and $\omega \cdot k$ denotes the Euclidean scalar product. 

In the differential equations setting, $\omega \in \R^d$ is Diophantine if there exist $c> 0$ and $\tau \geq d+1$ such that for all $k \in \Z^{d}\backslash \{0\}$ 
    $$|\omega \cdot k |\geq c |k |^{-\tau}.$$
Given $U\subset \R^n$, $\Lambda \subset \R^p$ and $h: U \times \T^d \times \Lambda \to \R^m $,
the small divisors equation for maps is 
    \begin{equation}\label{not:smalldivisors}
\varphi(u,\th + \omega ,\lambda )-\varphi(u,\th,\lambda)=h(u,\th,\lambda) 
\end{equation}
and the corresponding small divisors equation for differential equations is
\begin{equation}\label{not:smalldivisorsflow}
\partial_\th \psi(u,\th,\lambda) \cdot \omega = h(u,\th,\lambda).
\end{equation}

The following version of the small divisors lemma,  depending on $u\in \C^n$ and on $\lambda\in \C^p$ can be readily adapted from the one in~\cite{Russmann75}.
  
\begin{theorem}\label{thm:smalldivisors}
Take $U\subset \C^n$, $0\in \overline{U}$, $\Lambda \subset \C^p$ and $\sigma>0$. 
Let $h:U \times \T^d_\sigma \times \Lambda \to \C^k$ be real analytic with zero average and let $\omega \in \R^d$  be a Diophantine vector.

Then, there exist unique solutions $\varphi,\psi:U \times \T^d_\sigma \times \Lambda\to \C^k$ of \eqref{not:smalldivisors} and \eqref{not:smalldivisorsflow}, respectively, real
analytic, with zero average, such that, for $(u,\lambda)\in U\times \Lambda$,
$$
\sup_{\th \in \T^d_{\sigma-\delta}} \|\varphi(u,\th,\lambda)\| ,\; \sup_{\th \in \T^d_{\sigma-\delta}} \|\psi(u,\th,\lambda)\|\leq \frac{\Mg}{\delta^\tau} \sup_{\th \in \T^d_\sigma} \|h(u,\th,\lambda)\|, \qquad 0<\delta <\sigma.
$$
Moreover, if $h$ is a homogeneous function of degree $k$ with respect to $u$, then $\varphi,\psi$ also  are homogeneous functions  of degree $k$ with respect to $u$. If $h=\OO(\|u\|^r)$, $r\ge 1$ then also $\varphi,\psi=\OO(\|u\|^r)$.
\end{theorem}

We will denote by $\mathcal{D}[h]$ the unique solution  with zero average of either~\eqref{not:smalldivisors} or~\eqref{not:smalldivisorsflow}. 
We note that, since
\[
\partial _u\varphi(u,\th + \omega ,\lambda )-\partial _u\varphi(u,\th,\lambda)=\partial _uh(u,\th,\lambda),  
\]
if $\partial_u h =\OO(\|u\|^{r-1})$ then $
\partial_u \varphi=\OO(\|u\|^{r-1})$.
\subsection{Results for maps} \label{sec:casdemaps}
This section is  devoted to state the claims concerning with the existence of invariant manifolds of tori  for families of maps with an invariant torus whose transversal dynamics is tangent to the identity. 
In Section~\ref{sec:setupmap} we describe the maps under consideration and the general  conditions we need to guarantee the existence of these invariant manifolds. 
Afterwards, in Section~\ref{sec:resultsmaps} we state the main results.
In the statements of the results, some extra conditions will be introduced.
\subsubsection{Set up and hypotheses}\label{sec:setupmap}

Let $\mathcal{U}\subset \R^n \times \R^m $ be an open set such that $0\in \overline{\mathcal{U}}$ and $\Lambda $ be an open subset of $\R^p$. We consider  families of maps 
$F_\lambda :\mathcal{U}\times \T^d \to \R^n \times \R^m\times \T^d$, $\lambda \in \Lambda$,
of the form
\begin{equation}\label{system}
F_\lambda \begin{pmatrix} x \\ y \\ \th
\end{pmatrix}
=
\begin{pmatrix} x  +   f^{\geq N} (x,y,\th,\lambda)   \\
y  + g^{\geq M} (x,y,\th,\lambda) \\
\th +\omega  +  h^{\geq P}  (x,y,\th,\lambda ) 
\end{pmatrix}\qquad 
\end{equation}
with $\omega \in \R^d$ and $f^{\geq N }=\mathcal{O}(\|(x,y)\|^N)$, $g^{\geq M}=\mathcal{O}(\|(x,y)\|^M)$ and $h^{\geq P}=\mathcal{O}(\|(x,y)\|^P)$ for $N, M\geq 2$ and $P\geq 1$.  

For such maps, the torus
$$
\mathcal{T}=\{ (0,0,\th)\in \mathbb{R}^n \times \mathbb{R}^m\times \mathbb{T}^d \}
$$
is invariant and \emph{normally parabolic}, that is, the dynamics in the transversal directions to the torus is parabolic.  

We are interested in describing the stable and unstable sets of a torus related to a given open set $A\subset \mathbb{R}^n \times \mathbb{R}^m \times \T^d$ such that $\mathcal{T}\in \overline{A}$. Hence, we introduce the stable set
\begin{equation}\label{def:localstablemanifold}
W^{\textrm{s}}_A = \{(x,y,\theta ) \in A \mid \, F^k_{\lambda}(x,y,\theta) \in A, \forall k\ge 0, \, \lim _{k\to \infty} 
(F^k_\lambda)_{x,y}(x,y,\theta) = (0,0) \}
\end{equation}
and the unstable one:  
$$
W^{\textrm{u}}_A = \{(x,y,\theta ) \in A \mid \, F^{-k}_{\lambda}(x,y,\theta) \in A, \forall k\ge 0, \, \lim _{k\to \infty} 
(F^{-k}_\lambda)_{x,y}(x,y,\theta) = (0,0) \}.
$$
Their local versions are defined changing $A$ by $A_\rho = \{ \xi \in A \mid \mathrm{dist}  ( \xi, \mathcal{T} ) <\rho \}$. We will see that these sets are manifolds.

More concretely, we look for invariant manifolds tangent to the $x$-subspace. Therefore, we consider sets $V\subset \R^n$, $0\in \overline{V}$ and their local versions $V_\rho = V\cap B_\rho$, where $ B_\rho$ is the ball of radius $\rho$ in $\R^n\times \R^m$. Moreover, for $\beta>0$, we define the sets   
\begin{equation}\label{def:Vrhobeta}
V_{\rho,\beta}= \{(x,y) \in V_\rho \times \R^m\mid\; \|y\|\leq \beta   \|x\|\}, \qquad 
\Vext_{\rho,\beta} = V_{\rho,\beta} \times \mathbb{T}^d.
\end{equation}
The set $\Vext_{\rho,\beta}$ will play the role of the set $A$ in  \eqref{def:localstablemanifold}.
In this paper, we concentrate on the study of the stable manifold
associated to a set of the form $\Vext_{\rho,\beta}=V_{\rho,\beta} \times \mathbb{T}^d$. The unstable one can be obtained considering $F^{-1}_\lambda$. 

We will provide conditions for the existence of the invariant manifolds using the parametrization method, see \cite{CabreFL03a,CabreFL03b,CabreFL05,HaroCFLM16}
for a general presentation of the method and \cite{BFdLM2007,BFM2020a,BFM2020b,BFM20} for the specific application of the method to parabolic objects. Summarizing, this method consists in looking for functions 
$K(u,\th,\lambda)$ and $R(u,\th,\lambda)$ satisfying the invariance condition
\begin{equation}\label{invarianceequationmap}
F_\lambda( K(u,\th,\lambda)) = K(R(u,\th,\lambda),\lambda),
\end{equation}
with $K(0,\th,\lambda)=0$, $R(0,\th,\lambda)=0$ together with extra conditions to have the manifold  tangent at $\TTT$ to be a suitable subspace. 

We assume the following general conditions on $F_\lambda$ and the domain $\mathcal{U}$:
\begin{enumerate}[label=(\roman*)] 
\item $\mathcal{U}$ is an open set that contains a set of the form
$V_{\rho_0, \beta_0}  \subset \R^n \times \R^m$ for some positive $\rho_0$ and $\beta_0$ (see \eqref{def:Vrhobeta}), where $V$ is a cone-like domain, namely $0\in \partial V$ and for all $x\in V$ and $s> 0$, $sx\in V$. 
We remark that the origin does not necessarily belong to $\mathcal{U}$.
\item $f^{\geq N},g^{\geq M}$ and $h^{\geq P}$ can be expressed as sums of analytic functions, homogeneous with respect to $(x,y)\in \U $ of integer positive degree up to some order $q-1\ge N$. More precisely, there exists $q\in\N$, $q > N$ and 
\begin{equation}\label{sumhomogeneousfunctions}
\begin{aligned}
    f^{\geq N}(x,y,\th,\lambda) & = \sum_{j=N }^{q-1} f^{j} (x,y,\th,\lambda) + f^{\geq q}(x,y,\th,\lambda), \\
    g^{\geq M}(x,y,\th,\lambda) & = \sum_{j=M }^{q-1} g^{j} (x,y,\th,\lambda) + g^{\geq q}(x,y,\th,\lambda),  \\
    h^{\geq P}(x,y,\th,\lambda) & = \sum_{j=P }^{q-1} h^{j} (x,y,\th,\lambda) + h^{\geq q}(x,y,\th,\lambda) ,
\end{aligned}
\end{equation}  
where $f^j,g^j,h^j$ are analytic functions, homogeneous  of degree $j$ in  $(x,y)\in \U$, and the remainders $f^{\geq q},g^{\geq q}, h^{\geq q}$ are analytic and of order $\OO(\|(x,y)\|^{q})$. 
Moreover, we ask that $\partial^j _{x,y}f^{\geq q},\partial^j _{x,y}g^{\geq q}, \partial^j _{x,y}h^{\geq q}=\OO_{q-j}$ for $j=1,2$.

Note that for homogeneous functions in $(x,y)$ this property is automatically satisfied and when we take derivatives with respect to $\theta$ we do not lose order.   
Note that the functions $f^j,g^j,h^j$ can be extended by homogeneity to the set $\U^e \times \T^d\times \Lambda$ where
$\U^e= \{ (x,y) \in \R^n\times \R^m \mid \exists t\in (0,1] \text{ such that } t(x,y) \in \U \}$.
\end{enumerate}

Next, we assume three conditions, (iii), (iv) and (v) below, on $\overline{f}^N$ and $g^M$. First, given $\rho>0$, we define the constant
\begin{equation}\label{defa}
a_f:=-\sup_{x\in V_\rho,\,   \lambda\in \Lambda} \frac{\|x+ \fNm(x,0,\lambda)\| - \|x\|}{\|x\|^N}.
\end{equation}
\begin{enumerate}[label=(\roman*), resume]
\item Let $\rho_0 $ be the radius introduced in (i). The constant $a_f$ with $\rho= \rho_0$ satisfies the \textit{weak contraction} condition 
$$a_f>0.$$ 
\end{enumerate}
Note that this implies 
$$
\|x+ \fNm(x,0,\lambda)\| \le  \|x\| - a_f \|x\|^N, \qquad x\in V_{\rho_0}, 
\quad \lambda \in \Lambda.
$$

\begin{enumerate}[label=(\roman*), resume]
\item 

We assume
$$
g^M(x,0,\th,\lambda)=0.
$$     
Moreover, we ask $\overline{f}^N(x,0,\lambda)$ and $\partial_y \overline{g}^M(x,0,\lambda)$ to be defined and analytic in
$\mathcal{U^*}\times \Lambda$, where $\mathcal{U^*}$ in an open set  of $\R^n$ containing $0$.
Note that, by the homogeneity property, the domain of $\overline{f}^N(x,0,\lambda)$ and $\partial_y \overline{g}^M(x,0,\lambda)$ with respect to $x$ can be extended to $\R^n$. 
\item   
\label{hipotesisv}
We assume that there exists a positive constant $\Cv>0
$ such that
$$
\dist(x+\fNm(x,0,\lambda), V_{\rho_0}^c)\geq \Cv\|x\|^N, \qquad x\in V_{\rho_0},  \quad \lambda\in \Lambda, 
$$
where $V_{\rho_0}^c$ is the complementary set of $V_{\rho_0} \subset \R^n$. As a consequence $V_{\rho_0}$ is an invariant set for the map $x\mapsto x+ \overline{f}^N(x,0,\lambda)$.
\end{enumerate}

\begin{remark} It is important to emphasize that, if  $\mathcal{U}$ is an open set that contains the origin, then condition  (ii) is automatically satisfied; 
the expansions in \eqref{sumhomogeneousfunctions} are the standard Taylor expansions with respect to $(x,y)\in \mathcal{U}$. 

For the sake of completeness and applicability we have preferred to allow the more general situation when the origin is not contained in the regularity domain of $F_\lambda$. In this context we work with decompositions as sums of homogeneous functions instead of the classical Taylor expansion.   
\end{remark}

\begin{remark}
The hypotheses are chosen to obtain local invariant manifolds tangent to the subspace $\{y=0\}$. When the invariant manifold we are looking for is not going to be tangent to $\{y=0\}$ but of the form $y=Lx + \OO(\|x\|^2)$ we can  perform the linear change $u=x, \, v=y-Lx$ and look for the invariant manifold tangent to ${y=0}$. 
\end{remark}

\begin{remark}\label{prop:formsystem}
Let ${\F}_\lambda:\mathcal{U}\times \T^d  \to \R^{n+m} \times \T^d$, $\lambda \in \Lambda$,  be a map satisfying (i)-(iii) 
with $N,\, M\geq 2$ and $P\geq 1$
having an invariant manifold associated to the origin tangent to $\{y=0\}$. That is, assume that the manifold can be represented as the graph,
$y= \KK (x,\th,\lambda) $, with $\KK$ analytic, $\mathcal{C}^1$ at $0$, $\KK(0,\th,\lambda) =0$ and $\partial_x \KK(0,\th,\lambda) =0$.  Then, after a close to the identity change of variables,  $\mathcal{F}_\lambda $ has to satisfy that $M\leq N$
and $g^{M}(x,0,\th,\lambda)=0$.
We prove this remark in Appendix \ref{app:proofprop}.
\end{remark}

\begin{remark} 
  We notice that we are not assuming any condition on $\hP$. Therefore, we can always assume that $P\leq N$ since the case $\hP=0$ is allowed. 
\end{remark}

To finish this section, given $\rho>0$ we define the auxiliary constants related to $\overline{f}^N$
\begin{equation}\label{defbCf}
\begin{aligned}
b_f&=\inf_{\lambda\in \Lambda} \sup_{x\in V_\rho}\frac{\|\fNm(x,0,\lambda)\|} {\|x\|^N}, \qquad 
A_f= -\sup_{x\in V_\rho,\,\lambda\in \Lambda} \frac{\|\Id+ D_x \fNm(x,0,\lambda)\| - 1}{\|x\|^{N-1}},\\ 
D_f&=-\sup_{x\in V_\rho,\,\lambda\in \Lambda} \frac{\|\Id- D_x \fNm(x,0,\lambda)\| - 1}{\|x\|^{N-1}}.
\end{aligned}
\end{equation}
and $B_g$ related to $\overline{g}^M$:
\begin{equation}\label{defABg}
B_g=-\sup_{x\in V_\rho,\,\lambda\in \Lambda} \frac{\|\Id- D_y \gNm(x,0,\lambda)\| - 1}{\|x\|^{M-1}}.
\end{equation}

We notice that the constant $b_f$ is independent on $\rho$ since $\overline{f}^N$ is a homogeneous function of degree $N$. 

\begin{remark}\label{remark:rhosmall} Notice that, if $\rho_1 \leq \rho_2$ then the corresponding constants $a_f^{1,2}$, $b_f^{1,2}$, $A_f^{1,2}$, $D_f^{1,2}$, and $B_g^{1,2}$, associated to $\rho_1$ and $\rho_2$, respectively, defined in~\eqref{defa}, \eqref{defbCf} and \eqref{defABg} satisfy $a_f^1\geq a_f^2$, $b_f^1=b_f^2$, $A_f^1 \geq A_f^2$, $D_f^1 \ge D_f^2$ and $B_g^1\geq B_g^2$.
See Lemma 3.7 in \cite{BFM2020b}. We also have $a_f \le b_f$.

This remark will allow us to take $\rho$ as small as we need. We will use this fact throughout the paper without mention it.
\end{remark}

\subsubsection{Main results}\label{sec:resultsmaps}
Let 
\begin{equation}\label{defEast}
E^*> \left \{\begin{array}{ll} 
\max\{ -B_g, -D_f,0\}, \qquad &\text{if}\quad  M=N, \\
\max\{ -B_g, 0\}, &\text{if}\quad  M<N.
\end{array}\right.
\end{equation} 
Denoting $[\cdot]$ the integer part of a real number, we introduce the required minimum order $q$:
\begin{equation}\label{def:regularitatminima}
q^* = \left [\max\left\{ 2N-P, 2N-M+1, N-1 +\frac{N-1}{N-5/3} \frac{E^*}{a_f}\right\}\right]
\end{equation} 
and the index
\begin{equation}\label{defju}
j^*_u= \left \{\begin{array}{ll} \left [ -\frac{D_f}{a_f}\right ], \qquad &\text{if}\quad  D_f <0, \\
1, &\text{if}\quad  D_f \ge 0.
\end{array}\right.
\end{equation}
The first result we state is an \textit{a posteriori} result. Roughly speaking, it says that, if we know a good enough \textit{approximate solution} of the invariance equation~\eqref{invarianceequationmap}, then there is a \textit{true solution} of this equation close to it. 

\begin{theorem}[\textit{A posteriori} result]\label{thm:posterioriresult}
Let $F_\lambda$ 
be  of the form~\eqref{system} satisfying conditions $(i)-(v)$ with $q\ge q^*$. Assume $\omega$ is Diophantine and 
  $A_f>b_f \max\{1,N-P\}$.

Moreover, assume there exist analytic maps $K^{\leq}:V_{\rho_0} \times \T^d \times \Lambda \to \R^n\times \R^m \times \T^d$ and $R:V_{\rho_0} \times \T^d \times \Lambda \to V_{\rho_0} \times \T^d$, being sums of homogeneous functions with respect to $u$,  of the form
$$
K^{\leq}_{x,y}(u, \v,\lambda) - (u,0)=\OO(\|u\|^2),
\qquad K^{\leq}_\th(u,\v,\lambda)- \v = \OO(\|u\|),
$$
$$
R_u(u,\v,\lambda)-(u+ \overline{f}^N(u,0,\lambda))=\OO(\|u\|^{N+1}), \qquad 
R_\v(u,\v,\lambda)-\v-\omega=\OO(\|u\|)
$$
such that 
\begin{equation}\label{condEth}
E^\leq:= F_\lambda\circ K^{\leq} - K^{\leq } \circ R =
\OO(\|u\|^{q}). 
\end{equation}
 
Then, there exist $0<\rho\leq \rho_0$ and a unique
analytic function
$$
\Delta :V_{\rho} \times \T^d  \times \Lambda \to \R^{n+m}\times \T^d 
$$
satisfying $\Delta_{x,y}=\mathcal{O}(\|\u\|^{q-N+1})$,
$\Delta_\theta = \mathcal{O}(\|\u\|^{q-2N+P+1})$ and
\begin{equation}\label{inveqtheoremmaps}
F_\lambda\circ (K^\leq +\Delta) - (K^{\leq }+ \Delta)\circ R=0.
\end{equation}
Moreover, 
the map $\Delta$ is real analytic in a complex extension of $ V_{\rho}\times \T^d \times \Lambda$.

Let $K=K^\leq + \Delta$. For $\rho, \beta$ small enough, $K(V_{\rho}\times \T^d,\lambda)\subset W^{\mathrm{s}}_{\Vext_{\rho,\beta} }$, with $\Vext_{\rho,\beta}$ defined in~\eqref{def:Vrhobeta},  and, when the constant $B_g>0$, for some slightly smaller cone set $\widehat V$, 
\begin{equation}
\label{darrerstatementdethm28}
K(\widehat V_{\rho}\times \T^d,\lambda) =W^{\mathrm{s}}_{\widehat \Vext_{\rho,\beta} }, \qquad \widehat \Vext_{\rho,\beta} = \widehat{V}_{\rho,\beta} \times \mathbb{T}^d .
\end{equation}
\end{theorem}

Theorem~\ref{thm:posterioriresult} is proven in Section~\ref{sec:resultats_a_posteriori}. The next result  gives conditions that guarantee the existence of approximations that fit the hypotheses of Theorem
\ref{thm:posterioriresult}. Later on, in Section~\ref{sec:approximation}, we provide a concrete algorithm to compute the approximations as sums of homogeneous functions of the variable $u$, depending on the angles and parameters.

\begin{theorem}[Construction of the approximations]\label{thm:approximationmaps}
 Assume that the map $F_\lambda$ is of the form~\eqref{system} satisfying conditions $(i)-(v)$ and $q\ge q^*$. 
Furthermore, assume $\omega$ is Diophantine, $A_f>b_f$ and 
\[
\begin{aligned}
&D_y\overline{g}^M(x,0,\lambda) \quad \text{is invertible for all $(x,\lambda) \in V_{\rho_0}\times \Lambda$},\quad   &\text{if}  \quad   M<N,
\\
&2+\frac{B_g}{a_f}>0, \qquad & \text{if}
\quad
M=N.
\end{aligned}
\]
Then, there exists $0<\rho \leq \rho_0$ such that for any $ j\leq q-N$, there exist analytic maps $K^{(j)}:V_\rho \times \T^d \times \Lambda \to \R^{n+m}\times \T^d$, and $R^{(j)}:V_\rho \times \T^d \times \Lambda \to V_\rho \times \T^d$, such that 
\begin{equation}\label{eq:invthapproximation}
E^{(j)}:=F_\lambda \circ K^{(j)} - K^{(j)}\circ R^{(j)} =\OO(\|u\|^{j+N}).
\end{equation}
Moreover, $K^{(j)}$ and $R^{(j)}$ can be represented as sums of analytic homogeneous functions, of the form 
\begin{align*}
    K_x^{(j)}(u,\v,\lambda)&=u + \sum_{l=2}^j  \overline{K}_x^l(u,\lambda) + \sum_{l=1}^j \widetilde{K}_x^{l+N-1}(u,\v,\lambda), \\ 
    K_y^{(j)}(u,\v,\lambda)&=\sum_{l=2}^{j+N-M}  \overline{K}_y^l(u,\lambda) + \sum_{l=2}^{j+N-M} \widetilde{K}_y^{l+M-1}(u,\v,\lambda), \\ 
    K_\th^{(j)}(u,\v,\lambda)&=\v + \sum_{l=1}^{j+N-P}  \overline{K}_\th^{l }(u,\lambda) + \sum_{l=1}^{j+N-P } \widetilde{K}_\th^{l+P-1}(u,\v,\lambda)
\end{align*}  
and 
$$
R_u^{(j)}(u,\v,\lambda)  = u + \overline{f}^N(u,0,\lambda) + \sum_{l=2}^{j^*_u } R_u^{l+N-1}(u,\lambda),
\qquad 
R_\v^{(j)}(u,\v,\lambda)  = \v + \omega +  \sum_{l=2}^{j} R_\v^{l+P-2}(u,\lambda)
$$  
for  $j>j^*_u$.
Furthermore, if $P=N$, we obtain $R_\v^{(j)}(u,\v,\lambda)  = \v + \omega$.
\end{theorem}

\subsection{Consequences of Theorems~\ref{thm:posterioriresult} and~\ref{thm:approximationmaps}  for maps}

Combining Theorems~\ref{thm:posterioriresult} and~\ref{thm:approximationmaps}, we have the following claim.

\begin{theorem}[Existence of the stable manifold] \label{th:existencemap}
Let $F_\lambda$ be a map of the form~\eqref{system} satisfying conditions $(i)-(v)$ with $q\ge q^*$, where $q^*$ was introduced in~\eqref{def:regularitatminima}. 
Assume that $\omega$ is Diophantine, 
  $A_f>b_f \max\{1,N-P\}$ and   
\[
\begin{aligned}
&D_y\overline{g}^M(x,0,\lambda) \quad \text{is invertible for all} \quad (x,\lambda) \in V_{\rho_0}\times \Lambda,\quad   &\text{if}  \quad   M<N,
\\
&2+\frac{B_g}{a_f}>0, \qquad & \text{if}
\quad
M=N.
\end{aligned}
\]
Then, there exists $0<\rho \leq \rho_0$ such that the invariance equation
$$
F_\lambda \circ K = K \circ R
$$
has analytic solutions $K:V_{\rho} \times \T^d \times \Lambda \to \mathcal{U} \times \T^d$, $R:V_{\rho} \times \T^d  \times \Lambda \to V_{\rho} \times \T^d $ satisfying that, for $ \beta>0$ small enough and $\lambda\in \Lambda$
\begin{equation}\label{KcontainWsmap}
K(V_\rho \times \T^d,\lambda)  \subset   W^{\mathrm{s}}_{\Vext_{\rho,\beta}},  
\end{equation}
where $\Vext_{\rho,\beta}$ defined in~\eqref{def:Vrhobeta} and $W^{\mathrm{s}}_{\Vext_{\rho,\beta}}$ is the stable set of $F_\lambda$ (see~\eqref{def:localstablemanifold}).   

If we further assume that, if $M=N$, $B_g>0$
then,  
for some slightly smaller cone set $\widehat V$,
\begin{equation}\label{KequalWsmap}
K(\wh V_{\rho}\times \T^d,\lambda)  =  W^{\mathrm{s}}_{\wh \Vext_{\rho,\beta}} \qquad \mathrm{and} \qquad 
W^{\mathrm{s}}_{\wh \Vext_{\rho,\beta}}=  \bigcap_{k\geq 0} F_\lambda^{-k} (\wh V_{\rho,\beta} \times \T^d).
\end{equation}
\end{theorem}


\subsubsection{A conjugation result for attracting parabolic tori}

A direct consequence of the previous results is that if the transversal dynamics to the torus is parabolic and (weak) attracting for $x$ belonging to a cone set $V_\rho$, then it is conjugate to a map that can be expressed as a finite sum of homogeneous functions in $x\in V_\rho$, depending trivially on the angles.  

\begin{corollary}\label{cor:conjugationmaps}
Let $F_\lambda$ be a family of maps of the form~\eqref{system}   independent of the $y$-variable, namely
$$
F_\lambda(x,\th)= \big (x+f^{\geq N}(x,\th,\lambda),\th + \omega + h^{\geq P}(x,\th, \lambda)\big ).
$$
Assume that $f^{\geq N}, h^{\geq P}$ satisfy the corresponding conditions in (i)-(v) for some $q_1\ge q^*$,
$\omega$ is Diophantine and  $A_f>b_f\max\{1,N-P\}$. Then, the map $F_\lambda$ is conjugate to a map $R$ of the form
$$
R (u,\th,\lambda)=\left  (u+\overline{f}^N(u,\lambda)+ \sum_{l=2}^{j_u^*} R_x^{l+N-1}(u,\lambda ), \th+\omega+ \sum_{l=2}^{q_1-N} R_\th^{l+P-2}(u,\lambda) \right),  
$$
with $(u, \theta, \lambda) \in V_{\rho} \times \T\times \Lambda$, for some $0<\rho \leq \rho_0$ and $j_u^*$ is defined in~\eqref{defju}. Let $H$ be the conjugation.  
Then, $H$ and $R$ are real analytic in a complex extension of $V_\rho \times \T^d \times \Lambda$. 
\end{corollary}

\subsubsection{The case when all eigenvalues of  the linearization of the transversal dynamics to the torus are roots of 1}

In this section we explain how to apply the previous results to maps, $G_\lambda$ satisfying that for some $\ell \in \N$, $F_\lambda:=G_\lambda^{\ell}$ has the form~\eqref{system}. Namely we assume that  
\begin{equation}\label{defGlmaps}
G_\lambda(x,y,\th)= \left (\begin{array}{c} \A x + f^{\geq N}(x,y,\th,\lambda) \\ \B y + g^{\geq M}(x,y,\th,\lambda) \\ \th + \omega + h^{\geq P}(x,y,\th,\lambda) \end{array}\right ) ,\qquad  \Spec\A,\;\Spec \B \subset 
\bigcup_{k\in \Z} \{z\in \C \mid\, z^k=1\}.
\end{equation}
We notice that in this case the torus 
$\mathcal{T}=\{ (0,0,\th)\in \mathbb{R}^n \times \mathbb{R}^m\times \mathbb{T}^d \}
$ is also invariant and normally parabolic. We define, $W^{\mathrm{s}}_{A}$, the stable  set of $G_\lambda$ associated to the parabolic torus $\mathcal{T}$ as in~\eqref{def:localstablemanifold}, simply by changing $F_\lambda$ by $G_\lambda$.

We have the following result. 
\begin{corollary}\label{cor:rootsunity}
Let $G_\lambda$ be of the form~\eqref{defGlmaps} and $\ell \in \N$ be the minimum integer such that  $F_\lambda:=G_\lambda^\ell$ satisfies that $DF_\lambda(0)=\Id$. 

Assume that $F_\lambda$ is under the conditions in Theorem~\ref{th:existencemap}. Denote by $V$ a cone, $\rho, \beta>0$ constants and $K,R$ functions satisfying the
conclusions of Theorem~\ref{th:existencemap}, that is  
$K(V_\rho \times \mathbb{T}^d,\lambda) \subset W_{\Vext_{\rho,\beta}}^{\mathrm{s}} (F_\lambda)$ with $\Vext_{\rho,\beta} = V_{\rho,\beta} \times \mathbb{T}^d$ being the set defined in~\eqref{def:Vrhobeta} and $W_{\Vext_{\rho,\beta}}^{\mathrm{s}} (F_\lambda)$ the stable set of $F_\lambda$ associated to $\mathcal{T}$. 

Then,  
$$
\mathcal{W}:=\bigcup_{j=0}^{\ell -1}G^j_\lambda \left ( K(V_{\rho}\times \mathbb{T}^d,\lambda)  \right)\subset W^{\mathrm{s}}_{\Bext_{\rho,\beta}}, \qquad \text{with} \qquad 
\Bext_{\rho,\beta}=\bigcup_{j=0}^{\ell -1} 
G_\lambda^j(\Vext_{\rho,\beta} )
$$
and $W^{\mathrm{s}}_{\Bext_{\rho,\beta}}$ being the stable set defined in~\eqref{def:localstablemanifold} with respect to $G_\lambda$.

Assuming further that $B_g>0$ (the constant defined in~\eqref{defABg} for $F_\lambda= G_\lambda^\ell$), we have that $  W^{\mathrm{s}}_{\widehat{\Bext}_{\rho,\beta}} = \widehat{\mathcal{W}}$, where the notation $\,\wh{\,}\,$ means that the sets are related to a slightly smaller cone $\wh V \subset V$.  
\end{corollary}

Roughly speaking, this result asserts that the stable set of $G_\lambda$ is composed by $\ell$ different branches, each of them being the image by some iterate of $G_\lambda$ of the stable set of $F_\lambda= G_\lambda^\ell$.  
The proof of this claim is postponed to Appendix~\ref{proof_corollary}. 

\begin{remark}
    The maps considered in Corollary~\ref{cor:rootsunity} appear in ~\cite{Lee2021,Fefferman2021} when a certain economic model based on critical values is considered.  
\end{remark}

\subsection{Results for differential equations} 
\label{sec:casdefluxos}
Now we consider parametric families of non autonomous vector fields, depending quasi periodically on time, of the form
\begin{equation}\label{systemflow}
X(x,y,\theta,t,\lambda)=\big (f^{\geq N}(x,y,\theta,t,\lambda), g^{\geq M}(x,y,\theta,t,\lambda),\omega + h^{\geq P}(x,y,\theta,t,\lambda) \big ) 
\end{equation}
with  $(x,y,\theta, t,\lambda)\in \mathcal{U}\times \T^{\mathrm{d}} \times \R \times \Lambda\subset \R^{n+m}\times \T^{\mathrm{d}} \times \R\times \R^p$, $0\in \overline{\mathcal{U}} $, $\omega\in \R^{\mathrm{d}}$ and satisfying $f^{\geq N}=\OO(\|(x,y)\|^N)$, $g^{\geq M}=\OO(\|(x,y)\|^M)$, $h^{\geq P}=\OO(\|(x,y)\|^P)$
for some $2 \le M\le N$ and $1\leq P\leq N$. 
 
As in the case of maps, for any fixed value of the parameter, the torus $\mathcal{T}=\{(0,0,\theta)\mid \, \theta \in \T^{\mathrm{d}}\}$ is invariant by the flow having all transversal directions parabolic. We consider the following local stable manifold, which depends on a set $A \subset \R^{n+m} \times \mathbb{T}^d$, $\mathcal{T}\in \overline A$, which is defined by
\begin{align*}
W^{\textrm{s}}_{A} = \{(x,y,\theta ,t_0) \in A \times \R \mid \, &\Phi_X(t;t_0,x,y,\theta,\lambda) \in A, \; \forall t\ge t_0, \\ & \lim _{t\to \infty} 
(\Phi_X)_{x,y}(t;t_0,x,y,\theta,\lambda) = (0,0) \},
\end{align*}
where, according to the notation in Section~\ref{sec:notation}, $\Phi_X(t;t_0,x,y,\theta,\lambda)$ is the flow of the differential equation associated to \eqref{systemflow}. 
The sets $A$ will be of the form $\Vext_{\rho,\beta}=V_{\rho,\beta} \times \mathbb{T}^d$, 
introduced in~\eqref{def:Vrhobeta} or containing it.

We want to provide conditions that guarantee the existence and regularity of the local stable manifold. We will use the parametrization method. In the case of differential equations consists in solving the invariance equation
$$
\Phi_X(t;s,K(u,\vf,s,\lambda),\lambda)=K(\Psi(t;s,u,\vf,\lambda),t,\lambda)
$$
for $K$ and $\Psi$, where $\Psi$ is the solution of the equation restricted to the stable manifold (which is also unknown).  
The equivalent infinitesimal version of the invariance equation is
\begin{equation}\label{inveqflows}
X (K(u,\vf,t,\lambda),t,\lambda) - \partial_{u,\vf} K(u,\vf,t,\lambda) Y(u,\vf,t,\lambda) - \partial_t K(u,\vf,t,\lambda) =0,
\end{equation}
where $Y$ is the vector field associated to the flow $\Psi$ which describes the dynamics on $W^{\textrm{s}}_{\Vext_{\rho,\beta}}$.

By the definition of quasi periodicity we write $X (x,y,\theta,t,\lambda)=\widehat{X} (x,y,\theta,\nu t,\lambda)$ for some $\widehat{X}: \mathcal{U} \times \T^{d} \times \T^{ d'} \times \Lambda\to  \R^{n+m}\times \R^{d}$ (see~\eqref{sec:notquasi} in Section \ref{sec:notacio}) and some $\nu\in \R^{d'}$ independent on $\lambda$ that we call the time frequency of $X$. We introduce  $\widecheck{X}:\mathcal{U} \times \T^{d+d'} \times \Lambda\to  \R^{n+m}\times \T^d\times \T^{d'}$ by
$$
\widecheck{X}(x,y,\vartheta,\lambda) =\left (\begin{array}{cc} \widehat{X}(x,y,\vartheta,\lambda) \\ \nu \end{array}\right ), \qquad \vartheta=(\theta, \tau) \in \R^{d+d'},
$$
and the extended frequency
$$
\widecheck {\omega}=(\omega,\nu).
$$

The following elementary lemma allows us to relate the results for 
vector fields with the ones for maps.  
\begin{lemma}\label{lem:frommapstoflows} Let $F_\lambda: \mathcal{U} \times \T^d \times \T^{d'} \to \R^{n+m} \times \T^d\times \T^{d'}$ be the time $1$ map of \, $\widecheck{X}$, i.e.  $F_\lambda(x,y,\vartheta)=\Phi_{\widecheck{X}}(1;x,y,\vartheta,
\lambda)$. We have that if $f^{\geq N}, g^{\geq M}, h^{\geq P}$ in \eqref{systemflow}  satisfy hypotheses (i)-(v) and $\widecheck{\omega}$ Diophantine, then the map $F_\lambda$ has the form~\eqref{system} with slightly different $f^{\geq N}, g^{\geq M}, h^{\geq P}$ but with the same constants $a_f,b_f,A_f,D_f, B_g,a_V$. 
\end{lemma}
The proof of this lemma is straightforward from Theorem~\ref{thm:smalldivisors}, performing a finite averaging procedure, Gronwall's lemma and easy estimates, see~\cite{BFM20} for the case $n=1$, $N=M=P$. We skip the details of the proof.  

\begin{theorem}[\textit{A posteriori} result for flows]\label{thm:posterioriresultflow}
Let $X$ be a vector field of the form~\eqref{systemflow} with $f^{\geq N}, g^{\geq M}, h^{\geq P}$ satisfying conditions $(i)-(v)$ for some $q\ge q^*$ with $q^*$ given in~\eqref{def:regularitatminima}. Assume that $\widecheck {\omega}=(\omega,\nu)$ is Diophantine  and
$A_f>b_f \max\{1,N-P\}$. 
Assume further that there exist analytic maps $K^{\leq}:V_{\rho_0} \times \T^{\mathrm{d}} \times \R \times \Lambda \to \mathcal{U}\times \T^{\mathrm{d}}$ and $Y:V_{\rho_0} \times \T^{\mathrm{d}} \times \R \times \Lambda \to \R^n \times \R^d$ quasiperiodic with respect to $t$ with time frequency $\nu$, which are sums of homogeneous functions with respect to $u$, of the form
$$
K^{\leq}_{x,y}(u, \vf,t,\lambda) - (u,0)=\OO(\|u\|^2),
\qquad K^{\leq}_\theta(u,\vf,t,\lambda)- \vf = \OO(\|u\|), 
$$
$$
Y_u(u,\vf,t,\lambda)-  \overline{f}^N(u,0,\lambda) =\OO(\|u\|^{N+1}), \quad 
Y_{\vf}(u,\vf,t,\lambda)-\omega=\OO(\|u\|)
$$
such that  
\begin{equation*}
X (K^\leq(u,\vf,t,\lambda),t,\lambda) - \partial_{u,\vf} K^\leq (u,\vf,t,\lambda) Y(u,\vf,t,\lambda) - \partial_t K^\leq (u,\vf,t,\lambda) =
\mathcal{O}(\|u\|^{q}). 
\end{equation*}
 
Then, writing $\v=(\vf,\tau)$, the parametrization  $\widecheck{K}^{\leq}(u,\v,\lambda)= (\widehat{K}^\leq (u, \v,\lambda ), \tau)$ and $\widecheck{R}(u,\v,\lambda)=\Phi_{\widecheck{Y}}(1;u,\v,\lambda)$, the time $1$-map of 
$\widecheck{Y}(u,\v,\lambda)=(\widehat{Y}(u,\v,\lambda),\nu)$, satisfy all the hypotheses in Theorem~\ref{thm:posterioriresult} for the map $F_\lambda(x,y,\th)=\Phi_{\widehat{X}}(1;x,y,\th,\lambda)$. 

Let $\widecheck{\Delta}:V_{\rho} \times \T^{d+d'} \times \Lambda \to \R^{n+m} \times \T^{d+d'}$ be the analytic function provided by Theorem~\ref{thm:posterioriresult}. Then,  the quasiperiodic function $\Delta(u,\vf,t)= \widecheck{\Delta}_{x,y,\theta} (u, \vf, \nu t)$ satisfies the invariance equation
$$
 X  \circ (K^\leq + \Delta) - \partial_{u,\vf} (K^\leq + \Delta )  Y  - \partial_t (K^\leq + \Delta ) =0.
$$

If $ \rho,\beta$ are small enough, $K:=K^\leq + \Delta$ satisfies that  $K(V_{\rho}\times \T^{\mathrm{d}}\times \R,\lambda) \subset W^{\mathrm{s}}_{\Vext_{\rho,\beta} }$ and, when $B_g>0$, 
for some slightly smaller cone set $\widehat V$,
$$
 K(\wh V_{\rho}\times \T^{\mathrm{d}} \times \R,\lambda) =W^{\mathrm{s}}_{\wh \Vext_{\rho,\beta} }.
$$
\end{theorem}
 
Theorem~\ref{thm:posterioriresultflow} can be proven from Theorem~\ref{thm:posterioriresult} and Lemma~\ref{lem:frommapstoflows} following exactly the same lines as the ones showed in Section~5 in~\cite{BFM20} (see also~\cite{BFM2020a}). The details are left to the reader. 

Concerning the approximate solution, we have the analogous result to Theorem~\ref{thm:approximationmaps}.
Even that, using Lemma~\ref{lem:frommapstoflows} we could compute the approximate solution by means of the approximate solution given by Theorem~\ref{thm:approximationmaps} for the time $1$-map of the vector field $X$, in Section~\ref{sec:cohomologicalflow} we provide an  algorithm to compute $K^{(j)}$ and $Y^{(j)}$ directly from the vector field $X$ itself. 

\begin{theorem}[Approximation result for flows]\label{thm:approximationflows}
Assume that $X$ is an analytic vector field of the form~\eqref{systemflow}, satisfying conditions $(i)-(v)$ with $q\ge q^*$ and $q^*$ defined in~\eqref{def:regularitatminima}, that $\widecheck {\omega}=(\omega,\nu)$ is Diophantine and
\[
\begin{aligned}
&D_y\overline{g}^M(x,0,\lambda) \quad \text{is invertible for all} \quad (x,\lambda) \in V_{\rho_0}\times \Lambda,\quad   &\text{if}  \quad   M<N,
\\
&2+\frac{B_g}{a_f}>0, \qquad & \text{if}
\quad
M=N.
\end{aligned}
\]
Then, there exists $0<\rho \leq \rho_0$ such that for any $j\leq q-N$, there exist an analytic map $K^{(j)}:V_\rho \times \T^{d} \times \R \times \Lambda \to \R^{n+m}\times \T^{d}$ and an analytic vector field $Y^{(j)}:V_\rho \times \T^{d} \times \R\times \Lambda \to \R^n \times \R^{d}$,  depending quasiperiodically on $t$ with time frequency $\nu$, such that 
\begin{equation}\label{def:Ejflow}
\begin{aligned}
E^{(j)}&:=X (K^{(j)}(u,\vf,t,\lambda),t,\lambda) - \partial_{u,\vf} K^{(j)}(u,\vf,t,\lambda)Y(u,\vf,t,\lambda)- \partial_t K^{(j)}(u,\vf,t,\lambda) \\ &= \mathcal{O}(\|u\|^{j+N}). 
\end{aligned}
\end{equation}
In addition, $K^{(j)}$ and $Y^{(j)}$ can be expressed as sum of homogeneous functions of the form
\begin{align*}
    K_x^{(j)}(u,\vf,t,\lambda)&=u + \sum_{l=2}^j  \overline{K}_x^l(u,\lambda) + \sum_{l=1}^j \widetilde{K}_x^{l+N-1}(u,\vf,t,\lambda), \\ 
    K_y^{(j)}(u,\vf,t,\lambda)&=\sum_{l=2}^{j+N-M} \overline{K}_y^l(u,\lambda) + \sum_{l=1}^{j+N-M} \widetilde{K}_y^{l+M-1}(u,\vf,t,\lambda), \\ 
    K_\theta^{(j)}(u,\vf,t,\lambda)&=\vf + \sum_{l=2}^{j+N-P}  \overline{K}_\theta^{l-1}(u,\lambda) + \sum_{l=1}^{j+N-P} \widetilde{K}_\theta^{l+P-2}(u,\vf,t,\lambda) 
\end{align*}
and, for $j > j^*_u$ (see~\eqref{defju} for the precise value of $j^*_u$),
$$
Y_u^{(j)}(u,\vf,t,\lambda)  = \overline{f}^N(u,0,\lambda) + \sum_{l=2}^{j^*_u} Y_u^{l+N-1}(u,\lambda),
\qquad 
Y_{\vf}^{(j)}(u,\vf,t,\lambda)  =  \omega+  \sum_{l=2}^{j} Y_{\vf}^{l+P-2}(u,\lambda).
$$  
Moreover, if $P = N$, we obtain $Y^{(j)}_{\vf}(u,\vf,t,\lambda)  =  \omega$.
\end{theorem}

As a consequence of these results we obtain the existence theorem, Theorem~\ref{th:existenceflow} and a conjugation result, Corollary~\ref{cor:conjugationflows}.
\begin{theorem}[Existence of the stable manifold for flows] \label{th:existenceflow}
Let $X $ be an analytic vector field of the form~\eqref{systemflow} satisfying conditions $(i)-(v)$ with $q\ge q^*$. Assume that $\widecheck {\omega}=(\omega,\nu)$ is Diophantine, $A_f>b_f \max\{1,N-P\}$
and
\[
\begin{aligned}
&D_y\overline{g}^M(x,0,\lambda) \quad \text{is invertible for all} \quad (x,\lambda) \in V_{\rho_0}\times \Lambda,\quad   &\text{if}  \quad   M<N,
\\
&2+\frac{B_g}{a_f}>0, \qquad & \text{if}
\quad
M=N.
\end{aligned}
\]

Then, there exists  $0<\rho\leq \rho_0$ such that the invariance equation~\eqref{inveqflows} 
has analytic solutions $K:V_\rho \times \T^{\mathrm{d}}   \times \R \times \Lambda \to \mathcal{U} \times \T^{\mathrm{d}}$ and $Y:V_\rho  \times \Lambda \to \R^n \times \R^d $. If $ \rho,\beta$ are small enough, $K$ satisfies that  $K(V_{\rho}\times \T^{\mathrm{d}}\times \R,\lambda) \subset W^{\mathrm{s}}_{\Vext_{\rho,\beta} }$. Moreover, if $B_g>0$, for $\lambda\in \Lambda$ and for some slightly smaller cone set $\widehat V$,
$$
K(V_{\rho}\times \T^{\mathrm{d}}\times \R,\lambda) =  W^{\mathrm{s}}_{\wh \Vext_{\rho,\beta}}.
$$
\end{theorem}  
The conjugation result, analogous to Corollary~\ref{cor:conjugationmaps}, is: 
\begin{corollary}\label{cor:conjugationflows}
Let $X $ be an analytic vector field of the form~\eqref{systemflow} 
without the $y$-component and independent of the $y$-variable. Assume that $f^{\geq N}$ and $h^{\geq P}$ satisfy the corresponding conditions in (i)-(v) with $q\ge q^*$. We also assume that $\widecheck {\omega}=(\omega,\nu)$ is Diophantine  and 
$A_f>b_f\max\{1,N-P\}$. Then, there exists $0<\rho< \rho_0$ such that the vector field $X $ 
restricted to $V_{\rho, \beta} \times \T^d \times \R  $ is conjugate to  
$$
Y (u,\lambda)=\left (\overline{f}^N(u,\lambda)+ \sum_{l=2}^{j_x^*} Y_x^{l+N-1}(u,\lambda ), \ \omega+ \sum_{l=2}^{q-N} Y_\theta^{l+P-2}(u,\lambda) \right ).
$$
Moreover, both the conjugation and the vector field $Y$ are real analytic in a complex extension of $V_\rho \times \T^d \times \R\times \Lambda$. 
\end{corollary}

\section{Proof of the \textit{a posteriori} result for maps }
\label{sec:resultats_a_posteriori}

We start by explaining the strategy we use to prove Theorem~\ref{thm:posterioriresult}. First, in Section~\ref{sec:preliminariespost} we prove that, using  Theorem~\ref{thm:smalldivisors} in an appropriate way, we can remove the dependence on the angle  up to order $q-1$ in all the functions involved in equation~\eqref{condEth}. Second, in Section~\ref{sec:fixedpointpost} we provide the operator we will deal with  to prove the result, solving a related fixed point equation. This is done using the Fourier expansion (with respect to $\th$) of the  involved functions. In Section~\ref{sec:fixedpointana}, extending the technology developed in~\cite{BFM20,BFM2020a}, we prove that the above mentioned fixed point operator is a contraction. 
Finally, in Section \ref{sec:proofthexistence} we prove 
\eqref{darrerstatementdethm28}.
 
Along this section we will omit the dependence on the parameter $\lambda$ in the notation. 

We assume that the family of maps $F$ satisfy conditions (i)-(v) with  $q\ge q^*$, where $q^*$ is defined in~\eqref{def:regularitatminima}.


\subsection{Preliminaries }\label{sec:preliminariespost}

The purpose of this section is to rewrite Theorem~\ref{thm:posterioriresult} in a more suitable form to apply functional analysis techniques. 
Actually, Theorem~\ref{thm:posterioriresult} will be a consequence of 
Proposition \ref {prop:existencepartialmap} below which will be proved in Sections \ref{sec:fixedpointpost}-\ref{sec:fixedpointana}.

In this section, to be able to apply Theorem~\ref{thm:smalldivisors},
taking into account that $F$ is analytic, we will consider its extension to a complex domain $\U_\C \times \T_\sigma \times \Lambda_\C$.

\begin{proposition}\label{prop:preliminariesaproximationmap}
Assume we are under the hypotheses of Theorem~\ref{thm:posterioriresult}. Then,  
there exist a change of variables $\mathcal{C}(\xi,\eta,\varphi)=(x,y,\th)$ and a reparametrization $\mathcal{P}(v,\psi)=(u,\v)$ such that equation~\eqref{inveqtheoremmaps} becomes
\begin{equation}\label{inveqprop}
\widehat{F}\circ (\widehat{K}^\leq + \widehat{\Delta}) - (\widehat{K}^\leq +\widehat{\Delta}) \circ \widehat{R} =0
\end{equation}
with 
$$
\widehat{F}= (\widehat{F}_\xi, \widehat{F}_\eta, \widehat{F}_\varphi)
= 
\mathcal{C}^{-1} \circ F \circ \mathcal{C}
$$  
and 
$$
\widehat{K}^\leq
= (\widehat{K}^\leq_\xi, \widehat{K}^\leq_\eta, \widehat{K}^\leq_\varphi)
= \mathcal{C}^{-1} \circ K^\leq \circ \mathcal{P},\qquad 
\widehat{R}
= (\widehat{R}_v, \widehat{R}_\psi)
=\mathcal{P}^{-1} \circ R\circ \mathcal{P}
$$
satisfying the corresponding conditions in Theorem~\ref{thm:posterioriresult}
and 
\begin{equation}\label{eq:approximatedinvariancewidehat}
\widehat{E}^\leq(v,\psi) := \widehat{F} \circ \widehat{K}^{\leq} (v,\psi)- \widehat{K}^{\leq } \circ \widehat{R} (v,\psi)=\OO(\|v\|^q).
\end{equation}
Moreover, 
$$
\partial_\varphi \widehat{F}_{\xi,\eta}, \;\partial_\varphi \widehat{F}_\varphi-\mathrm{Id}=\OO(\|(\xi,\eta)\|^q) 
$$
and
$$
\partial_{\psi}\widehat{K}^\le_{\xi,\eta}, \; \partial_\psi \widehat{K}^\le_\varphi - \mathrm{Id}, \; \partial_\psi \widehat{R}_v, \; \partial_\psi \widehat{R}_\psi - \mathrm{Id} = \OO(\|v\|^q).
$$
\end{proposition}

The main part of this section is devoted to prove this proposition. 
After the proof is complete we state Proposition \ref {prop:existencepartialmap} 
and deduce Theorem~\ref{thm:posterioriresult}   from it.
First, we perform several steps of  averaging to remove the dependence of $F$ on the angles up to order $q-1$.
 
\begin{lemma}\label{lem:Ftilde}
Let $F$ be a map of the form~\eqref{system} satisfying conditions (i)-(v), with $\omega $ Diophantine. 
Then, there exists a near to the identity change of variables $\mathcal{C}:\mathcal{U}'\times \T^d \to \mathcal{U}\times  \T^d$,
where $\U' $ is a domain slightly smaller
than $\U$ such that $0\in \overline{\U}'$,
which transforms $F$ into 
\begin{equation*}
\widehat{F} (\xi,\eta,\varphi)
=
\begin{pmatrix} \xi  +   \fNm (\xi,\eta) + \widehat{f}^{\geq N+1} (\xi,\eta) \\
\eta  + \gNm(\xi,\eta)+ \widehat{g}^{\geq M+1}(\xi,\eta) \\
\varphi +\omega  +  \hPm(\xi,\eta )+ \widehat{h}^{\geq P+1} (\xi,\eta)  
\end{pmatrix} + \widehat{F}^{\geq q}(\xi,\eta,\varphi)
\end{equation*}
with $\widehat{F}^{\geq q} (\xi,\eta,\varphi)=\OO(\|(\xi,\eta)\|^{q})$.
The change has
the form
$$
(x,y,\th)=\mathcal{C}(\xi,\eta,\varphi)=(\xi,\eta,\varphi) + \widehat{\mathcal{C}}(\xi,\eta,\varphi)
$$
with  
$$
\widehat{\mathcal{C}} (\xi,\eta,\varphi)= \left ( \sum_{j=N}^{q-1} \mathcal{C}^j_x(\xi,\eta,\varphi),  
 \sum_{j=M}^{q-1} \mathcal{C}^j_y(\xi,\eta,\varphi),\sum_{j=P}^{q-1} \mathcal{C}^j_\theta(\xi,\eta,\varphi)
\right )
$$
and the terms $\mathcal{C}^j$ are homogeneous functions of degree $j$ in the $(\xi,\eta)$-variables. 

Moreover, $\mathcal{C}$ and $\widehat{F}$ are analytic.
\end{lemma}
\begin{proof}
If $M<N$, first we perform a change of variables of the form
$$
(x,y,\th)=(\xi,   \eta + \mathcal{C}_y^M(\xi,\eta,\varphi),\varphi),
$$
where  $\mathcal{C}_y^M$ is a homogeneous function of degree $M$ in $(\xi, \eta)$ to be determined. 
The transformed map, denoted by $F^{(1)}$,  keeps the same form as $F$ for the $x,\theta$ components since 
$$
f^N(\xi,\eta+\mathcal{C}_y^M(\xi,\eta,\varphi),\varphi)=
f^N(\xi,\eta,\varphi)+ \OO(\|(\xi,\eta)\|^{N-1+M})
=f^N(\xi,\eta,\varphi)+\OO(\|(\xi,\eta)\|^{N+1}),
$$
$$
h^P(\xi,\eta+\mathcal{C}_y^M(\xi,\eta,\varphi),\varphi)=
h^P(\xi,\eta,\varphi)+ \OO(\|(\xi,\eta)\|^{P-1+M})
=h^P(\xi,\eta,\varphi)+\OO(\|(\xi,\eta)\|^{P+1}).
$$
With respect to the $\eta$ component we obtain
$$
F_\eta^{(1)}=\eta+\mathcal{C}^M_y(\xi,\eta,\varphi) - \mathcal{C}^M_y (\xi,\eta,\varphi+\omega) + \overline{g}^{M}(\xi,\eta) + \widetilde{g}^M(\xi,\eta,\varphi) + \OO(\|(\xi,\eta)\|^{M+1}).
$$
We consider the equation 
$$
\mathcal{C}_y ^M(\xi,\eta,\varphi+\omega) - \mathcal{C}^M_y (\xi,\eta,\varphi)= \widetilde g^M(\xi,\eta,\varphi)
$$
and, by Theorem \ref{thm:smalldivisors}, we take $\mathcal{C}_y ^M=\mathcal{D}[\widetilde{g}^M]$ to have
$$
F_\eta^{(1)} = \eta + \overline{g}^{M}(\xi,\eta) + \widehat g^{\geq M+1}(\xi,\eta,\varphi).
$$
Then, $F^{(1)}$ still satisfies conditions (i)-(v). In particular, now conditions (iii), (v) only depend on 
$\fNm (\xi,0)$ which remains unchanged.
Clearly, condition (iv) is satisfied by $\gNm(\xi,0)$. 
We repeat this procedure $(N-M)$-times to get a new map (renaming the variables by $(x,y,\th)$ and the map by $F$) 
such that 
$$
F_y (x,y,\th)=y + \sum_{j=M}^{N-1} \overline{g}^{j}(x,y) + g^{\geq N}(x,y,\th),
$$
where $g^j$ are homogeneous functions of degree $j$, $g^{\geq N}=\OO(\|(x,y)\|^{N})$ and $F$ satisfies conditions (i)-(v). 

Now, if $P<N$, we deal with the $\th$ component and we consider a change of coordinates
$$
(x,y,\th)=(\xi,\eta,\varphi+\mathcal{C}^P_\theta(\xi,\eta,\varphi)),
$$
where $\mathcal{C}^P_\theta$ is a homogeneous function of degree $P$ in $(\xi,\eta)$. 
The components $(\xi,\eta)$ of the transformed map, denoted again by $F^{(1)}$, satisfy conditions (i)-(ii) and  
\begin{align*}
F_\xi^{(1)}&=\xi+  {f}^{N}(\xi,\eta,\varphi) + \OO(\|(\xi,\eta)\|^{N+1}), \\ 
    F_\eta^{(1)}&=\eta+ \sum_{j=M}^{N-1} \overline{g}^{j}(\xi,\eta) + \widehat{g}^{\geq N}(\xi,\eta,\varphi), \\ 
    F_\varphi^{(1)}&= \varphi+\omega+\mathcal{C}^P_\theta(\xi,\eta,\varphi) - \mathcal{C}^P_\theta (\xi,\eta,\varphi+\omega) + \overline{h}^{P}(\xi,\eta) + \widetilde{h}^P(\xi,\eta,\varphi) + \OO(\|(\xi,\eta)\|^{P+1}).
\end{align*}
Therefore, choosing $\mathcal{C}^P_\theta=\mathcal{D}[\widetilde{h}^P]$, we have that 
$$
F_\varphi^{(1)}=\varphi+\omega+ \overline{h}^P(\xi,\eta) + \OO(\|(\xi,\eta)\|^{P+1}).
$$
Repeating this procedure $(N-P)$-times we obtain a map $G^{(0)}(x,y,\th)$ such that 
$$
\partial_\th G^{(0)}-\Id=\OO(\|(x,y)\|^{N}). 
$$

Next, we look for  a change of variables of the form
$$
(x,y,\th)=(\xi + \mathcal{C}_x^N(\xi,\eta,\varphi), \eta + \mathcal{C}_y^N(\xi,\eta,\varphi), \varphi+ \mathcal{C}_\theta^N(\xi,\eta,\varphi)),
$$
where $\mathcal{C}_x^N$, $\mathcal{C}_y^N$ and $\mathcal{C}_\theta^N$ are homogeneous functions with respect to $\xi,\eta$ of degree $N$,
to transform $G$ to $G^{(1)}$. We impose the conditions
\begin{align*}
& 
\mathcal{C}^N_x(\xi,\eta,\varphi) - \mathcal{C}^N_x (\xi,\eta,\varphi+\omega) + \overline{f}^{N}(\xi,\eta) + \widetilde{f}^N(\xi,\eta,\varphi) = \OO(\|(\xi,\eta)\|^{N+1}), \\ &  
\mathcal{C}^N_y(\xi,\eta,\varphi) - \mathcal{C}^N_y (\xi,\eta,\varphi+\omega) + \sum_{j=M}^{N}\overline{g}^{j}(\xi,\eta) + \widetilde{g}^N(\xi,\eta,\varphi) = \OO(\|(\xi,\eta)\|^{N+1}),
 \\&
\mathcal{C}^P_\theta(\xi,\eta,\varphi) - \mathcal{C}^P_\theta (\xi,\eta,\varphi+\omega) + \sum_{j=P}^{N}\overline{h}^{j}(\xi,\eta) + \widetilde{h}^N(\xi,\eta,\varphi) = \OO(\|(\xi,\eta)\|^{N+1}).
\end{align*}
As before, taking $\mathcal{C}^N_{x}=\mathcal{D}[\widetilde{f}^N]$, $\mathcal{C}^N_y=\mathcal{D}[\widetilde{g}^N]$ and $\mathcal{C}^N_\theta=\mathcal{D}[\widetilde{h}^N]$ which are analytic and have the right order, we obtain 
$$
\partial_\varphi G^{(1)}_{\xi,\eta} =\OO(\|(\xi,\eta)\|^{N+1}),
\qquad \partial_\varphi G^{(1)}_\varphi-\Id=\OO(\|(\xi,\eta)\|^{N+1}) . 
$$
Since $G^{(0)}$ is a sum of homogeneous functions up to degree $q-1$ plus a remainder of order $q$, we can repeat this procedure $(q-N)$-times obtaining that $\widehat{F}:=G^{(q-N)}$ satisfies
$$
\partial_\varphi \widehat{F}_{\xi,\eta}=\OO(\|(\xi,\eta)\|^{q}),   \qquad 
\partial_\varphi \widehat{F}_{\varphi}-\Id=\OO(\|(\xi,\eta)\|^{q}).
$$ 
The change $\mathcal{C}$ in the statement is the composition of all previous changes. Since $\mathcal{C}$ is close to the identity, it sends $ \U'  \times \T^d_{\sigma'}  $ to
$ \U \times \T^d_{\sigma} $, where $ \U'$
is a slightly smaller domain than $\U$ and $\sigma' <\sigma$.
\end{proof}  

In the following lemma, which is a straightforward consequence of Lemma~\ref{lem:Ftilde}, we make a better choice of the parameters $(u,\v)$ which will allow us to find a new reparametrization $R$ such that its terms of order less than $q$ do not depend on the angular variables.
\begin{lemma}\label{lem:changeparameter}  
Assume that $R$ is analytic and satisfies the conditions for $F$ in Theorem~\ref{thm:posterioriresult} for some $\rho_0>0$. Then, there exist $\rho>0$ and an analytic change of parameters  $\mathcal{P}:\wh V_{\rho'}\times \T^d_{\sigma'}  \to V_{\rho_0}\times \T^d_\sigma$ of the form $(u,\v)= \mathcal{P}(v,\psi)= (v,\psi) + \widehat{\mathcal{P}}(v,\psi)$ with
$$
\widehat{\mathcal{P}}_u(v,\psi)=\sum_{j=N}^{q-1} \mathcal{P}^j_u(v,\psi) , \qquad 
\widehat{\mathcal{P}}_\Theta (v,\psi)=\sum_{j=P}^{q-1} \mathcal{P}^j_{\Theta}(v,\psi),
$$
where $\widehat V_{\rho'}$ is a slightly smaller cone, 
$\sigma'< \sigma$ and $\mathcal{P}^j$ are homogeneous functions of degree $j$,  with respect to $v$,
such that 
$$
\widehat{R}(v,\psi):=\mathcal{P}^{-1} \circ R \circ \mathcal{P}(v,\psi) = \widecheck{R}(v,\psi)+ \OO(\|v\|^q)
$$
with 
$$
\widecheck{R}(v,\psi):=(v +  {R}_v(v),\psi + \omega
+ {R}_\psi(v)) = (v+ \overline{f}^N(v,0) + \OO(\|v\|^{N+1}), \psi+\omega+ \OO(\|v\|^P)).
$$
In addition, both $\widecheck{R}$ and $\widehat{R}$ are analytic and $\widecheck{R}$ is a sum of homogeneous functions in $v$ up to order $q-1$.
\end{lemma}
\begin{proof}
The claim follows applying Lemma~\ref{lem:Ftilde} to $R$ instead of $F$, taking into account that  $R$ is a map  independent of $y$ and that does not have $y$-component (that is, $m=0$). 
\end{proof}

Now we  apply Lemmas~\ref{lem:Ftilde} and~\ref{lem:changeparameter} to prove Proposition~\ref{prop:preliminariesaproximationmap}.  

\begin{proof}[Proof of  Proposition~\ref{prop:preliminariesaproximationmap}]
We set $F,K^\leq$ and $R$ satisfying all the conditions of Theorem~\ref{thm:posterioriresult}. 
Let $\mathcal{C}$ and $\widehat{F}$ those provided by Lemma~\ref{lem:Ftilde}. We notice that, since 
$$
\mathcal{C} \circ \widehat{F} \circ \mathcal{C}^{-1}\circ K^\leq
- K^\leq \circ R= E^{\leq}
$$
we have that  
$$
\widehat{F} \circ \mathcal{C}^{-1} \circ K^\leq =\mathcal{C}^{-1} (K^\leq \circ R + E^{\leq }).
$$
Then,  $\widecheck{K}^{\leq}  := \mathcal{C}^{-1}  \circ K^\leq $
satisfies the conditions in Theorem~\ref{thm:posterioriresult}:
$$
\widecheck{K}^{\leq}_{\xi,\eta} (u,\v) - (u,0) =\OO(\|u\|^2), \qquad \widecheck{K}^{\leq}_\varphi(u,\v) - \v = \OO (\|u\|))
$$ 
and, by the mean value theorem, the new remainder 
\begin{equation*}
\widecheck{E}^{\leq}  :=\widehat{F} \circ \widecheck{K}^{\leq}  - \widecheck{K}^{\leq} \circ R  = \mathcal{C}^{-1} \circ (K^\leq \circ R + E^{\leq }) - \mathcal{C}^{-1} \circ K^\leq \circ R =  \OO(\|u\|^{q}),
\end{equation*}
(see~\eqref{condEth}).

Next, we consider the close to the identity change of parameters in Lemma~\ref{lem:changeparameter}, $(u,\v)=\mathcal{P}(v,\psi)$. We have that 
$$
\widehat{F}\circ \widecheck{K}^{\leq} (\mathcal{P}(v,\psi)) - 
\widecheck{K}^{\leq} \circ R(\mathcal{P}(v,\psi)) =\widecheck{E}^{\leq } (\mathcal{P}(v,\psi)).
$$
We define $\widehat{K}^\leq  =\widecheck{K}^\leq \circ \mathcal{P} $ and $\widehat{E}^\leq = \widecheck{E}^{\leq} \circ \mathcal{P}$. Then, the above equality reads  
\begin{equation}\label{prop:Fhattilde}
\widehat{F} \circ \widehat{K}^\leq    -  \widehat{K}^\leq  \circ \widehat{R}  = \widehat{E}^{\leq}  
\end{equation}
with $\widehat{R}=\mathcal{P}^{-1} \circ R \circ \mathcal{P}$ defined in Lemma~\ref{lem:changeparameter} and $\widehat{E}^{\leq}   = \OO(\|v\|^{q})$. We emphasize that $\widehat{K}^\leq $ also satisfies the conditions in Theorem~\ref{thm:posterioriresult}, namely it is a sum of homogeneous functions in $v$ and
\begin{equation}\label{propwidehatK}
\widehat{K}^\leq_{\xi,\eta} (v,\psi) - (v,0) =\OO(\|v\|^2), \qquad \widehat{K}^\leq_\varphi(v,\psi) - \psi = \OO (\|v\|).
\end{equation} 

It only remains to check that $\partial_\psi \widehat{K}^\leq_{\xi,\eta} =\OO(\|v\|^q)$ and $\partial_\psi \widehat{K}^\leq_\varphi -\mathrm{Id} = \OO(\|v\|^q)$.
To do so we write $\widehat{F}= \widehat{F}^{\leq q -1} + \widehat{F}^{\geq q}$ and then~\eqref{prop:Fhattilde} becomes
\begin{equation}\label{propFhatl1}
\widehat{F}^{\leq q -1} \circ \widehat{K}^\leq  - \widehat{K}^\leq 
\circ  \widehat{R} =\widehat{E}^\leq  -\widehat{F}^{\geq q} \circ \widehat{K}^{\leq }.
\end{equation}
Since by Lemma~\ref{lem:changeparameter}, 
$\widehat{R} = \widecheck{R} + \OO(\|v\|^q)$ and by Lemma~\ref{lem:Ftilde}, $\widehat{F}^{\leq q-1}-(0,0,\varphi)$ does not depend on the angle $\varphi$,  equation~\eqref{propFhatl1} can be written as:
\begin{equation}\label{propFhatl1better}
 \widehat{F}^{\leq q -1}(\widehat{K}^\leq _{\xi,\eta}(v,\psi),\widehat{K}^\leq _{\varphi}(v,\psi))-
\widehat{K}^\leq \circ \widecheck{R}(v,\psi) =\OO(\|v\|^q),
\end{equation}
where  $\widecheck{R}(v,\psi)=(v+ {R}_v(v), \psi+\omega + {R}_\psi(v))$. 

Now, taking the derivative with respect to $\psi$ in both sides of~\eqref{propFhatl1better}, using that 
$D \widehat{F}^{\leq q -1} = \Id + (\OO(\|v\|^{N-1}),\; \OO(\|v\|^{M-1}),\; \OO(\|v\|^{P-1}))$ and that $\partial _\psi \widecheck{R} = (0,\; \Id)$, 
we obtain  
\begin{equation}\label{partialK1}
\partial_\psi \widehat{K}^\leq (v,\psi) - \partial_\psi \widehat{K}^\leq (\widecheck{R}(v,\psi))
= 
\left (\begin{array}{c}
\OO(\|v\|^q) + \OO(\|v\|^{N-1}\|\partial_{\psi}\widehat{K}^\leq_{\xi,\eta}(v,\psi)\|) 
\\ 
\OO(\|v\|^q) + \OO(\|v\|^{M-1}\|\partial_{\psi}\widehat{K}^\leq_{\xi,\eta}(v,\psi)\|)
\\
\OO(\|v\|^q) + \OO(\|v\|^{P-1}\|\partial_{\psi}\widehat{K}^\leq_{\xi,\eta}(v,\psi)\|)
\end{array} \right ).
\end{equation}
On the other hand, using the properties of $\widehat{K}^\leq$ in~\eqref{propwidehatK} and the ones of $\widecheck{R}$ in Lemma~\ref{lem:changeparameter}, by Taylor's theorem we have that 
\begin{equation}\label{partialK2}
\begin{aligned}
\partial_\psi \widehat{K}^\leq (v+ {R}_v(v),\,     &\psi + \omega +  {R}_\psi(v)) = 
\partial_{\psi}\widehat{K}^\leq (v,\psi +\omega)  \\ &+ 
\left (\begin{array}{c} \OO(\|v\|^{N} \|\partial_{\psi,v}^2 \widehat{K}^\leq_{\xi}(v,\psi)\|) + 
\OO(\|v\|^P \|\partial_{\psi,\psi}^2 \widehat{K}^\leq_{\xi}(\u,\psi)\|) \\ 
\OO(\|v\|^{N}\|\partial_{\psi,v}^2 \widehat{K}^\leq_{\eta}(v,\psi)\|) +
\OO(\|v\|^P \|\partial_{\psi,\psi}^2 \widehat{K}^\leq_{\eta}(v,\psi)\|) \\
\OO(\|v\|^N \|\partial_{\psi,v}^2 \widehat{K}^\leq_{\varphi}(v,\psi)\|) + 
\OO(\|v\|^P\|\partial_{\psi,\psi}^2 \widehat{K}^\leq_{\varphi}(v,\psi)\|)
\end{array}\right ).
\end{aligned}
\end{equation}
Notice that, using $N\geq P\geq 1$, $N\ge M \ge 2$,  properties~\eqref{propwidehatK} of $\widehat{K}$, that 
$ {R}_v=\OO(\|u\|^N)$ and $ {R}_\psi=\OO(\|v\|^P)$, at least, we obtain that
$$
\partial_\psi \widehat{K}^\leq(v,\psi)-\partial_\psi \widehat{K}^{\leq}(v,\psi+\omega)=\big (\OO(\|v\|^3), \OO(\|v\|^3) , \OO(\|v\|^2)\big ).
$$ 
Here, to estimate the orders of 
the first and second derivatives of $\widehat{K}^\leq$ we have used that $\widehat{K}^\leq$ satisfies $\mathcal{C}\circ \widehat{K}^\leq = K^{\leq}$ with $\mathcal{C},K^\leq$ being sums of homogeneous functions with respect to $(\xi, \eta)$ and $u$, respectively, and that all of them are analytic in $\U \times \T^d_\sigma$ for some $\U$ and  $\sigma>0$. 

Therefore, since $\partial_\psi \widehat{K}_{\xi,\eta}^\leq $ and $\partial_\psi \widehat{K}_{\varphi}^\leq-\mathrm{Id}$ have zero average, by Theorem~\ref{thm:smalldivisors},
$$
\partial_\psi \widehat{K}^\leq(v,\psi)-(0,0,\Id)=\big (\OO(\|v\|^3), \OO(\|v\|^3) , \OO(\|v\|^{2})\big ).
$$
Assume by induction that 
$$
\partial_\psi \widehat{K}^\leq(v,\psi)-(0,0,\Id)=\big (\OO(\|v\|^l), \OO(\|v\|^l) , \OO(\|v\|^{l-1})\big )
$$  
for $l=2,\cdots, q-1$. Then, using~\eqref{partialK1} and~\eqref{partialK2} we have that
$$
\partial_\psi \widehat{K}^\leq(v,\psi) - \partial_{\psi}\widehat{K}^\leq(v,\psi + \omega) = 
\left ( \begin{array}{c} \OO(\|v\|^q) + \OO(\|v\|^{N-1+l})+
\OO(\|v\|^{P+l}) \\
\OO(\|v\|^q) + \OO(\|v\|^{M-1+l})+\OO(\|v\|^{P+l}) \\
\OO(\|v\|^q) + \OO(\|v\|^{P-1+l})+ \OO(\|v\|^{P+l-1})
\end{array}\right ).
$$
Now, using Theorem~\ref{thm:smalldivisors} and that $N\geq M\ge 2$, $N\ge P\geq 1$, we conclude that
$$
\partial_\psi \widehat{K}^\leq(v,\psi)-(0,0,\mathrm{Id})=\big (\OO(\|v\|^{l+1}), \OO(\|v\|^{l+1}) , \OO(\|v\|^{l})\big ).
$$ 
Therefore, when $l=q-1$,   
$$
\partial_\psi \widehat{K}^\leq(v,\psi)-(0,0,\mathrm{Id})=\big (\OO(\|v\|^{q}), \OO(\|v\|^{q}) , \OO(\|v\|^{q-1})\big ).
$$ 
To finish, we notice that, applying once more~\eqref{partialK1} and~\eqref{partialK2}, we obtain that in fact, 
$\partial_\psi \widehat{K}^\leq_{\varphi}(v,\psi)-\Id=\mathcal{O}(\|v\|^q)$. 

By construction of $\wh F$ and Remark~\ref{remark:rhosmall} the constants   $\widehat{a}_f $, $\widehat{A}_f$ and $\widehat{b}_f$ of $\widehat{F}$ for $\rho'\leq \rho$ satisfy condition (iii) and $\widehat{A}_f> A_f > b_f \max\{1,N-P\}  =\wh b_f \max\{1,N-P\} $.
Now, Proposition \ref{prop:preliminariesaproximationmap}
follows from Lemmas \ref{lem:Ftilde} and \ref{lem:changeparameter}.
\end{proof}
 
We state an intermediate result, whose technical proof is postposed to Sections~\ref{sec:fixedpointpost} and~\ref{sec:fixedpointana}, that implies the existence part of Theorem~\ref{thm:posterioriresult} as a corollary. Indeed, formula~\eqref{propFhatl1better} suggests that we can use a simpler $R$ to prove the result. 
\begin{proposition}\label{prop:existencepartialmap}
Let $\widehat{F}$, $\widehat{K}$ and $\widehat{R}$  satisfy the conditions on Proposition~\ref{prop:preliminariesaproximationmap}. Let $\widecheck{R}(v,\psi)=(\widecheck{R}_v(v), \widecheck{R}_\psi(v,\psi))=(v+{R}_v(v),\psi+\omega + {R}_\psi(v))$, introduced in Lemma \ref{lem:changeparameter}, that satisfies
$$
\widehat{R}(v,\psi)-\widecheck{R}(v,\psi)=\OO(\|v\|^q).
$$
Then, the invariance equation
\begin{equation}\label{eqinvfinal}
\widehat{F} \circ (\widehat{K}^\leq + \widecheck{\Delta}) - (\widehat{K}^\leq + \widecheck{\Delta}) \circ \widecheck{R}=0 
\end{equation}
has an analytic solution $\widecheck \Delta$ such that $\widecheck\Delta_{\xi,\eta}=\OO(\|v\|^{q-N+1})$, $\widecheck\Delta_{\varphi} = \OO(\|v\|^{q-2N+P+1})$.
\end{proposition}
\begin{proof}[Proof of the claim on the existence of the parametrization in Theorem~\ref{thm:posterioriresult} from Proposition \ref{prop:existencepartialmap}] We first note that, by Proposition~\ref{prop:preliminariesaproximationmap}, to prove Theorem~\ref{thm:posterioriresult} we only need to solve the invariance equation~\eqref{inveqprop}. 
We note that the difference between the invariance equation~\eqref{eqinvfinal} in Proposition~\ref{prop:existencepartialmap} and~\eqref{inveqprop} is just the dynamics on the invariant manifold, namely in the latter is $\widecheck{R}$ while in the former is $\widehat{R}$. To overcome this  issue we apply Proposition~\ref{prop:existencepartialmap} to $\widehat{R}(v,\psi)$
considered as a map $U\times \T^d\subset \R^n\times \T^d \to \R^n\times \T^d$ instead of the map 
$\widehat F: \U\times \T^d\subset \R^n\times \R^m\times\T^d \to \R^n\times\R^m \times \T^d$ taking $\widehat{K}^\le=\Id $. Then,  $\widehat{R}$ satisfies the hypotheses of Proposition \ref{prop:preliminariesaproximationmap} with $m=0$. Note that in particular, $R_v(v) = \overline{f}(u,0)$.
Indeed, since $\widehat{R}(v,\psi) - \widecheck{R}(v,\psi) = \OO(\|v\|^q)$, we have that $\widehat{R} \circ \mathrm{Id}- \mathrm{Id} \circ \widecheck{R} = \OO(\|v\|^q)$. Then,  there exists   $\Delta_R$ such that  
$$
\widehat{R} \circ (\mathrm{Id} + {\Delta_R})- (\mathrm{Id} + {\Delta_R})\circ \widecheck{R}=0
$$
and 
$$
 {\Delta_R}(v,\psi)=\big (\mathcal{O}(\|v\|^{q-N+1}), \mathcal{O}(\|v\|^{q-2N+P+1})\big ).
$$
We define $\Psi = \mathrm{Id} +  {\Delta_R}$ and 
we have that $\widehat{R} \circ \Psi = \Psi \circ \widecheck{R}$, so that $\widehat{R}$ and $\widecheck{R}$ are conjugate. Now, let $\widecheck{\Delta}$ be a solution of~\eqref{eqinvfinal}. 
We introduce
$$
\widehat{\Delta}=\widehat{K}^\leq \circ \Psi^{-1} - \widehat{K}^\leq + \widecheck{\Delta}\circ \Psi^{-1}= \big (\mathcal{O}(\|v\|^{q-N+1}),\mathcal{O}(\|v\|^{q-N+1}), \mathcal{O}(\|v\|^{q-2N+P+1})\big ).
$$ 
Then,  
\begin{align*}
0&=\widehat{F}\circ (\widehat{K}^\leq+ \widecheck{\Delta})\circ \Psi^{-1}  - (\widehat{K}^\leq+\widecheck{\Delta}) \circ \widecheck{R} \circ \Psi^{-1} \\& =\widehat{F}\circ (\widehat{K}^\leq+ \widecheck{\Delta}) \circ \Psi^{-1}  - (\widehat{K}^\leq+\widecheck{\Delta}) \circ \Psi^{-1} \circ \widehat{R}  \\
&=\widehat{F}\circ (\widehat{K}^{\leq} + \widehat{\Delta}) - (\widehat{K}^\leq + \widehat{\Delta}) \circ \widehat{R}.
\end{align*}
This implies the existence result in Theorem~\ref{thm:posterioriresult}, since all changes of variables and parameters are analytic.
\end{proof}
\begin{remark}
We postpose the proof of the characterization of the local stable invariant manifold $W^{\mathrm{s}}_{\Vext_{\rho,\beta}}$ in~\eqref{darrerstatementdethm28} to Section~\ref{sec:proofthexistence}.
\end{remark}

\subsection{The functional equation for $K$}\label{sec:fixedpointpost}
 We will prove Proposition~\ref{prop:existencepartialmap} solving  a fixed point
equation derived from \eqref{inveqprop}.
 
The first (non-trivial) step is to find the appropriate fixed point equation. 
As we did previously, we decompose $\widehat{F}= \widehat{F}^{\leq q-1} + \widehat{F}^{\geq q}$ with $\widehat{F}^{\leq q-1} -\Id$ independent on $\varphi$. Denoting $\zeta=(\xi,\eta)$, $D\widehat{F}^{\leq q-1}$ has the form
$$
D\widehat{F}^{\leq q-1}(\zeta ) = \left (\begin{array}{ccc} 
\mathrm{Id}+ \partial_{\xi}\fNm(\zeta)+ \partial_{\xi}\widehat{f}^{\geq N+1} (\zeta)& \partial_{\eta}\fNm(\zeta)+ \partial_{\eta }\widehat{f}^{\geq N+1}(\zeta) & 0 \\ 
\partial_{\xi}\gNm(\zeta)  + \partial_{\xi}\widehat{g}^{\geq M+1}(\zeta) & \mathrm{Id}+ \partial_{\eta}\gNm(\zeta)+ \partial_{\eta }\widehat{g}^{\geq M+1}(\zeta) & 0 \\
\partial_{\xi}\hPm(\zeta)+ \partial_{\xi}\widehat{h}^{\geq P+1}(\zeta) & 
\partial_{\eta}\hPm(\zeta)+ \partial_{\eta}\widehat{h}^{\geq P+1}(\zeta) & \mathrm{Id}
\end{array}\right ).
$$
Therefore, we can write
\begin{equation}\label{decompositionDF}
D\widehat{F}^{\leq q-1} (\zeta )
= \mathbf{M}(\zeta) + \mathbf{N}(\zeta ):=\left (\begin{array}{cc}
\mathrm{Id}+ \CC(\zeta) & 0 \\   
0 &  \mathrm{Id}
\end{array}\right ) + \left (\begin{array}{cc}
0 & 0 \\ 
\cbf(\zeta )  &  0
\end{array}\right ).
\end{equation}
Notice that \eqref{decompositionDF} defines implicitly $\mathbf{M}(\zeta)$, $\mathbf{N}(\zeta)$, $\CC(\zeta)$ and $\mathbf{c}(\zeta)$. 
We also decompose 
$$
\widehat{K}^\leq(v,\psi) = \widehat{K}^{\leq q-1}(v,\psi) +  \widehat{K}^{\geq q}(v,\psi), \qquad 
\widehat{K}^{\geq q} (v,\psi) = \OO(\|v\|^q)
$$
with $\widehat{K}^{\leq q-1}(v,\psi)$ of degree $q-1$ with respect to $v$ and  $\widehat{K}^{\leq q-1}(v,\psi) -(v,0,\psi)$ independent of $\psi$. 
We decompose the condition for $\widecheck{\Delta}$ as
\begin{align*}
    0=& \widehat{F}\circ(\widehat{K}^\leq+ \widecheck{\Delta}) - (\widehat{K}^\leq+\widecheck{\Delta })\circ \widecheck{R} \\
     =& \widehat{F}^{\leq q -1} \circ \widehat{K}^\leq - \widehat{K}^\leq\circ  \widecheck{R}
     \\ &+ \widehat{F}^{\geq q} (\widehat{K}^\leq+ \widecheck{\Delta}) \\
     & + \widehat{F}^{\leq q -1} \circ (\widehat{K}^\leq+ \widecheck{\Delta }) - \widehat{F}^{\leq q -1} \circ \widehat{K}^\leq - D\widehat{F}^{\leq q -1}(\widehat{K}^\leq)\widecheck{\Delta }
     \\ &+D\widehat{F}^{\le q-1}(\widehat{K}^{\leq})\widecheck{\Delta } - \mathbf{M}(\widehat{K}^{\leq q-1} ) \widecheck{\Delta } \\
     &+ \mathbf{M}(\widehat{K}^{\leq q-1}  )\widecheck{\Delta } - \widecheck{\Delta } \circ \widecheck{R}.
\end{align*}
We introduce the operators  
\begin{align}
\mathcal{L}[\widecheck{\Delta }]  = & \mathbf{M}(\widehat{K}^{\leq q-1} )\widecheck{\Delta}  - \widecheck{\Delta }  \circ \widecheck{R} , \notag \\ 
\mathcal{N}[\widecheck{\Delta }] = &\widehat{F}^{\leq q -1} \circ \widehat{K}^\leq - \widehat{K}^\leq\circ  \widecheck{R}
   + \widehat{F}^{\geq q} \circ (\widehat{K}^\leq+ \widecheck{\Delta}) 
        +D\widehat{F}^{\le q-1}(\widehat{K}^{\leq})\widecheck{\Delta} - \mathbf{M}(\widehat{K}^{\leq q-1} ) \widecheck{\Delta } \notag 
        \\ &+ \widehat{F}^{\leq q -1} \circ (\widehat{K}^\leq+ \widecheck{\Delta }) - \widehat{F}^{\leq q -1} \circ \widehat{K}^\leq - D\widehat{F}^{\leq q -1}(\widehat{K}^\leq)\widecheck{\Delta } ,
\label{def:Nonlinearmap} 
\end{align}
and we rewrite the condition for $\widecheck{\Delta }$ as
\begin{equation}\label{eq:fixedpoint1} 
\mathcal{L} [\widecheck{\Delta }] = - 
\mathcal{N}[\widecheck{\Delta }].
\end{equation}
In order to express the above equation as a fixed point equation we need to invert the linear operator $\mathcal{L}$ in an appropriate Banach space. Actually, we will find a right inverse of it. In this section we proceed formally. In the next one we provide the necessary estimates. We have to solve the equation
\begin{equation} \label{eq:inversaL}
\mathcal{L} [\widecheck{\Delta}] = T
\end{equation}
with $T$ in some functional space. First, we expand $\widecheck{\Delta}$ and $T$ in Fourier series:
$$\widecheck{\Delta}(v,\psi) = \sum_{k\in \mathbb{Z}^d} \Delta^{(k)}(v) e^{2\pi ik \cdot \psi},
\qquad T(v,\psi) = \sum_{k\in \mathbb{Z}^d} T^{(k)}(v) e^{ 2\pi ik \cdot \psi}.
$$ 
We recall that $\widehat{F}^{\leq q-1}- \mathrm{Id}$ does not depend on $\varphi$ (Lemma~\ref{lem:Ftilde}) and that $\widehat{K}^{\leq q-1} (v,\psi)-(v,0,\psi)$ does not depend on $\psi$ (Proposition \ref{prop:preliminariesaproximationmap}). 
The block structure of the matrix $\mathbf{M}$ permits to uncouple equation \eqref{eq:inversaL} into two equations, one for the $(\xi,\eta)$-components,
$\widecheck{\Delta}_{\xi,\eta}$, $T_{\xi,\eta}$,  
and the other for the $\varphi$-components, $\widecheck{\Delta}_{\varphi}$, $T_{\varphi}$.
Therefore, we have to solve
\begin{align*}
[\Id + \CC(\widehat{K}^{\leq q-1}_{\xi,\eta} )] \widecheck{\Delta}_{\xi,\eta}  - \widecheck{\Delta}_{\xi,\eta}  \circ \widecheck{R} & = T_{\xi,\eta}, \\
\widecheck{\Delta}_{\varphi}  - \widecheck{\Delta}_\varphi  \circ \widecheck{R} & = T_{\varphi}. 
\end{align*}
For the Fourier coefficients, since $\widecheck{R}(v,\psi)=(v+ {R}_v(v), \psi + \omega+  {R}_\psi(v))$, we have
\begin{align}
[\Id + \CC(\widehat{K}^{\leq q-1}_{\xi,\eta} )] {\Delta}^{(k)}_{\xi,\eta}  - e^{2\pi ik (\omega + {R}_\psi(v))}   \Delta_{\xi,\eta}^{(k)} \circ \widecheck{R}_v & = T^{(k)}_{\xi,\eta} \label{eq:coef-k-xi-eta}, \qquad k\in \Z^d, \\ \Delta^{(k)}_\varphi  - e^{2\pi ik (\omega + {R}_\psi(v))} {\Delta}_\varphi^{(k)}  \circ \widecheck{R}_v & = T_{\varphi}^{(k)}, \qquad k\in \Z^d. \label{eq:coef-k-psi}
\end{align}
We denote $\mathcal{L}^{(k)}_{\xi,\eta}\big [\Delta^{(k)}_{\xi,\eta}\big ]$ and 
$\mathcal{L}^{(k)}_{\varphi}\big [\Delta^{(k)}_{\varphi}\big ]$ the left hand sides of \eqref{eq:coef-k-xi-eta} and \eqref{eq:coef-k-psi}, respectively.
The corresponding (formal) inverses $\mathcal{S}^{(k)}_{\xi,\eta}$ and $\mathcal{S}^{(k)}_{\varphi}$ 
are
\begin{align*}
\mathcal{S}^{(k)}_{\xi,\eta}\big [T_{\xi,\eta}^{(k)}\big ](v)  & =\sum_{j=0}^\infty \left [ \prod_{l=0}^j \big ((\mathrm{Id}+\CC\circ \widehat{K}_{\xi,\eta}^{\leq q-1} \circ \widecheck{R}_v^l\big )^{-1}\right ] e^{2\pi ik \big (j\omega + \sum_{l=0}^{j-1}{R}_\psi \circ \widecheck{R}_v^l\big )} T_{\xi,\eta}^{(k)}\circ \widecheck{R}_v^j, 
\\
\mathcal{S}^{(k)}_{\varphi} \big [T_{\varphi}^{(k)}\big ] (v) &= 
\sum_{j=0}^\infty  e^{2\pi ik \big (j\omega + \sum_{l=0}^{j-1}{R}_\psi \circ \widecheck{R}_v^l\big )} T_{\varphi}^{(k)}\circ \widecheck{R}_v^j , 
\end{align*}
where
$\widecheck{R}_v^l = \widecheck{R}_v \circ \dots {\hspace{-0.3cm}^{{\vspace{-0.4cm}l)}}} \circ\widecheck{R}_v$. 
Let
$$
\mathcal{S}^{(k)}\big [T^{(k)} \big ] =\big (\mathcal{S}^{(k)}_{\xi,\eta} \big [T^{(k)}_{\xi,\eta}\big ], \mathcal{S}^{(k)}_{\varphi} \big [T^{(k)}_\varphi\big ]\big ).
$$
Then, the operator $\mathcal{L}$ has a formal right inverse given by
\begin{align}\label{def:Slinear}
\mathcal{S}[T](v,\psi) &= \sum_{k\in \mathbb{Z}^d}  e^{2 \pi ik \cdot \psi} \mathcal{S}^{(k)}\big [T^{(k)}\big ](v) \\ 
&= \sum_{j=0}^\infty \left [ \prod_{l=0}^j \big (\mathbf{M} ( \widehat{K}_{\xi,\eta}^{\leq q-1}( \widecheck{R}_v^l (v)))\big )^{-1}\right ]   T\Big(\widecheck{R}_v^j(v),j\omega +\psi + \sum_{l=0}^{j-1}{R}_\psi ( \widecheck{R}_v^l(v))\Big)  \notag
\end{align}
or equivalently
\begin{equation}\label{def:Slinearanalytic}
 \mathcal{S}[T](v,\psi) =  \sum_{j=0}^\infty \left [ \prod_{l=0}^j \big (\mathbf{M} ( \widehat{K}_{\xi,\eta}^{\leq q-1}( \widecheck{R}_v^l (v)))\big )^{-1}\right ]   T  \circ \widecheck{R}^j(v,\psi) .
\end{equation}
Having defined $\mathcal{S}$, we can consider the equation 
\begin{equation}\label{eq:fixedpoint2} 
\widecheck{\Delta} = \mathcal{F}[\widecheck{\Delta} ]:=
-\mathcal{S} \circ \mathcal{N}[\widecheck{\Delta}].  
\end{equation}
Clearly, if $\widecheck{\Delta}$ is solution of 
\eqref{eq:fixedpoint2}, it is also solution of \eqref{eq:fixedpoint1}.


\subsection{Solution of the fixed point equation} \label{sec:fixedpointana}
To prove that the fixed point equation~\eqref{eq:fixedpoint2} has a solution in a suitable Banach space we need to study both the linear 
operator $\mathcal{S}$, defined in~\eqref{def:Slinear},  and the nonlinear operator $\mathcal{N}$, defined in \eqref{def:Nonlinearmap}.
This will be done in Sections~\ref{sec:linearoperatordif} and~\ref{sec:nonlinearoperatordif}, respectively.

We recall that  the operator $\mathcal{S}$ depends on    
\begin{equation}\label{defwidecheckRanalitic}
\widecheck{R}(v,\psi)=(\widecheck{R}_v(v,\psi), \widecheck{R}_\psi (v,\psi))=(v+R_v(v),\psi+\omega +R_\psi(v)),
\end{equation}
where 
$$
R_v(v)=\overline{f}^N(v,0)+ w^{\ge N+1} (v),
\qquad \text{with}\qquad w^{\ge N+1} (v) = \OO(\|v\|^{N+1}),
$$
and $R_\psi(v) = \OO(\|v\|^P)$ are sums of homogeneous functions in $v$ of degree at most $q-1$. 

For positive $\rho$, $\gamma$ and $\sigma$ we define the sets
$$
\Omega_\rho(\gamma)=\{ v\in \C^n \mid\,v =(\Re v, \Im v)\in \R^n\times \R^n ,\, \Re v\in V_\rho,\,
\|\Im v\|\leq  \gamma \|\Re v\|\}
$$
and
$$
\Gamma_\rho(\gamma,\sigma)=\left \{ (v,\psi) \in \Omega_\rho(\gamma) \times \T_\sigma^d\mid \, \|\Im \psi \| + \sum_{l=0}^\infty \|\Im R_\psi(\widecheck{R}_v^l(v)) \| < \sigma\right \}.
$$

From now on we fix constants 
$a,b$ and $A,B,D$ such that
\begin{equation}\label{defabACanalytic}
0<a<a_f,\quad b>b_f,\quad   B<B_g, \quad D<D_f,\quad A<A_f
,\quad   A > b \max\{1,N-P\} 
\end{equation}
and, if $E' > E > \max \{-B,-D, E^*\} $,
\begin{equation} \label{fitadeEprima}
E' < \frac{N-4/3}{N-5/3} E^*  .  
\end{equation}
with $a_f,b_f, A_f, D_f, B_g, E^*$ defined in~\eqref{defa}, \eqref{defbCf} \eqref{defABg} and \eqref{defEast},  respectively. 
Taking the norm $\|z\| = \max \{ \|\Re z \|, \|\Im z \|\}$ in $\C^n$, 
we have that if $\A$ is a complex $n\times n$ matrix and $\A= \A_1+i \A_2$ with 
$\A_1, \A_2 $ real matrices, then $\|\A\| \le \|\A_1\|+\|\A_2\|$. 

By definition~\eqref{defbCf} of $D_f$, 
 for $\zeta \in V_{\rho,\beta}$ we have that
\begin{equation}\label{bound:textbfM1}
\left \| \big (\Id + D_\xi \overline{f}^N (\zeta) + D_\xi \widehat{f}^{\geq N+1} (\zeta) \big )^{-1}  
\right \| \leq  1 - (D_f - (\beta +\rho ) \Mg )\|\xi\|^{N-1}  
\leq  1 - D\|\xi\|^{N-1}  
\end{equation}
and, by definition~\eqref{defABg} of $B_g$,
\begin{equation}\label{bound:textbfM2}
\begin{aligned}
\left \|   \big (\Id + D_\eta \overline{g}^M (\zeta) + D_\eta \widehat{g}^{\geq M+1} (\zeta) \big )^{-1} \right \| & \leq 
\|\Id - D_\eta \overline{g}^M(\zeta) \|
+ \Mg \|\xi \|^M 
\\
& \leq \| \Id - D_\eta \overline{g}^M(\xi,0)\| + \Mg \beta \|\xi\|^{M-1} + 
\Mg \|\xi \|^M \\
& \leq 1 - \big (B_g  -  (\beta +\rho )\Mg \big ) \|\xi\|^{M-1}
\\
& \leq 1 - B \|\xi\|^{M-1}.
\end{aligned}
\end{equation}
Moreover, $\| D_\eta \overline{g}^M(\zeta)+ D_\eta \widehat{g}^{\geq M+1} (\zeta) \| \le \beta \Mg \|\xi\|^{M-1}$.

Then, by definition~\eqref{decompositionDF} of $\mathbf{M}$, and bounds~\eqref{bound:textbfM1} and~\eqref{bound:textbfM2}, we obtain that, if $\rho,\beta$ are small enough (depending on $B$ and $D$),
\[
\| \big (\mathbf{M}(\xi,\eta)\big )^{-1} \|\leq 1+ E \|\xi\|^{N-1}, \qquad (\xi,\eta) \in V_{\rho, \beta}.
\]
Also, a computation shows that if 
$\xi \in \Omega_\rho(\gamma), \, \|\eta\| \leq \beta \| \xi \|$, and
$\gamma $ is small we have
$\| \Re \big (\mathbf{M}(\xi,\eta)\big )^{-1} \|\leq 1+ (E +\OO(\gamma))\|\xi\|^{N-1}$
and 
$\| \Im \big (\mathbf{M}(\xi,\eta)\big )^{-1} \|\leq \Mg \gamma  \|\Re \big (\mathbf{M}(\xi,\eta)\big )^{-1} -\Id \|$.

Therefore
\begin{equation}\label{propietatCanalytic}
\| \big (\mathbf{M}(\xi,\eta)\big )^{-1} \|\leq 1+ E' \|\xi\|^{N-1}, \qquad \xi \in \Omega_\rho(\gamma), \quad  \|\eta\| \leq \beta \| \xi \|.
\end{equation}
 
The next result is the key to control the iterates of $\widecheck{R}$. Its proof is deferred to Appendix~\ref{apendixA}.  
 
\begin{lemma}\label{lem:calR} 
Assume $A_f>b_f$.
Let $\widecheck{R}$ be as~\eqref{defwidecheckRanalitic} and $a,b,A$ be constants satisfying~\eqref{defabACanalytic} with $A>b$. Fix constants
$a^*  <a(N-1)$ and $b^*  >b(N-1)$. Then,
there exist positive $\rho, \gamma$ small enough such that 
\begin{itemize}
\item[(1)] The set $\Omega_\rho(\Gamma)$ is invariant by $\widecheck{R}_v$, that is
 $$
    \widecheck{R}_v (\Omega_\rho(\gamma) ) \subset \Omega_\rho(\gamma).
    $$
\item[(2)]
For $k\geq 0$ and $v\in \Omega_\rho(\gamma)$:   
\begin{equation}\label{propcalRlema}
\frac{\|v\|}{\big [ 1+ kb^*   \|v\|^{N-1}\big ]^{\frac{1}{N-1}}}\leq \|\widecheck{R}_v^k(v)\| \leq \frac{\|v\|}{\big [ 1+ ka^*   \|v\|^{N-1}\big ]^{\frac{1}{N-1}}}.
\end{equation}
\item [(3)]
If $A/b>\max\{1,N-P\}$, then there exists some constant $\Mg>0$
such that
\begin{equation}\label{propImRlema} 
\sum_{l=0}^\infty \|\Im {R}_\psi (\widecheck{R}_v^{l}(v))\|\leq \Mg \frac{\|\Im v\|}{\|v\|^{N-P}}, \qquad v\in \Omega_\rho(\gamma).
\end{equation}
\item [(4)]
If $A/b>\max\{1,N-P\}$, then
there exists $\sigma>0$ such that 
\begin{equation}\label{prop:Wcontingut} 
\widecheck{R}(\Gamma_\rho(\gamma,\sigma)) \subset \Gamma_\rho(\gamma,\sigma).
\end{equation}
\end{itemize}
\end{lemma}
 
\begin{remark}
We notice that if $A_f>b_f$, then we can always take $A<A_f$ and $b>b_f$ satisfying $A>b$. 

We emphasize that when $n=1$, $A_f=Nb_f>b_f$ (this does not happen, in general,  when $n>1$, see~\cite{BFM2020a,BFM2020b} for examples). Then, ~\eqref{propImRlema} and~\eqref{prop:Wcontingut} always hold true in the one dimensional case, since we can choose the values of $A$  and $b$ satisfying the hypotheses of Lemma~\ref{lem:calR}. 
\end{remark}

\begin{remark}  
If $P\geq N-1$, the set $\Gamma_\rho(\gamma,\sigma)$ contains $\Omega_{\rho}(\gamma') \times \T^d_{\sigma'}$ for some
$\gamma' \leq \gamma$ and $\sigma'\leq \sigma$.

When $1\leq P<N-1$, the set $\Gamma_\rho(\gamma,\sigma)$ contains the points
$(v,\psi)$ satisfying $\Re v\in V_{\rho}$, $\|\Im v\|\leq \gamma' \|v\|^{N-P}$, $\psi \in \T_{\sigma'}^d$, for some
$\gamma' \leq \gamma$ and $\sigma'\leq \sigma$.

The previous work~\cite{BFM20} deals with the case $n=1$ and $P=N$. Lemma~\ref{lem:calR} is the main tool to generalize the results in~\cite{BFM20} to the case $1\le P<N$. 
\end{remark}

\begin{remark}
Lemma~\ref{lem:calR} holds true uniformly in $\lambda\in \C$ in compact subsets of $ \Lambda_\C$ where $\Lambda_\C$ is a suitable complex extension of $\Lambda$ to an open subset of 
$\C^p$. 
\end{remark}
 
We introduce the spaces $\mathcal{X}_s$, 
$s\in \Z$, we will deal with below.  Given  $0< \rho, \gamma\le 1$, and $ \sigma>0$  we define
$$
\mathcal{X}_s=\left\{ h:\Gamma_\rho(\gamma,\sigma) \to \C^{\ell}\;|\;  h \; \text{real analytic }, \, \|h\|_s:= \sup_{(v,\psi) \in \Gamma_\rho(\gamma,\sigma)} \frac{\|h(v,\psi )\| }{\|v\|^s}<\infty\right\}
$$
with $\ell\ge 1$ 
(if some component  $h_\varphi$ of $h$ takes values on $\T^d$, we will assume that the component $h_\varphi$ of $h$ considered as an element of $\mathcal{X}_s$, takes values on the universal covering $\C^d$ of $\T^d$). 
With the above introduced norms $\mathcal{X}_s$ are Banach spaces.
It is immediate to see that if $s<t$, then $\mathcal{X}_t\subset \mathcal{X}_s$ and if $h\in \mathcal{X}_t$ then
$\| h\|_s \le \rho^{t-s} \| h\|_t $. Furthermore, if $h\in \XX_t$ and $g\in \XX_s$ then $h\cdot g \in \XX_{t+s}$
and $\|h\cdot g \|_{t+s} \le \|h \|_{t} \| g \|_{s} $.

Moreover, given $r\in \Z$ and $\nu>0$, we introduce the product Banach space 
$$\mathcal{X}^\times_r :=
\mathcal{X}_{r-N+1}  \times \mathcal{X}_{r-N+1} \times \mathcal{X}_{r-2N+P+1}
$$
endowed with the norm
$$
\| h\|^\times_r:=
\| h_\xi\|_{r-N+1}  +\| h_\eta\|_{r-N+1}   + \nu\| h_\varphi\|_{r-2N+P+1}
$$
for some $\nu>0$. To make this norm more flexible we keep $\nu$ as a parameter. 
Below we will use a value of $\nu$  satisfying a certain smallness condition.

\subsubsection{The linear operator $\mathcal{S}$}\label{sec:linearoperatordif}

\begin{lemma} \label{lem:linearoperatordif}
Assume that $A_f>b_f\max\{1,N-P\}$. 
Then, if 
$$
s>\max\left \{ N-1 +\frac{N-1}{N-5/3}\frac{E^*}{a_f}, 2N-P-1\right \},
$$ 
the linear operator $\mathcal{S} :\mathcal{X}^\times_{s+N-1} \to \mathcal{X}^\times_{s}$ formally  
introduced in~\eqref{def:Slinearanalytic} is well defined and bounded. 
\end{lemma}
\begin{proof} 
Let $0<a< a_f$.
We fix  $a^*  , b^*  ,\rho,\gamma$ satisfying the conditions in Lemma~\ref{lem:calR}, and moreover, $(N-4/3) a_f < a^*  $ and $\beta$ small.  Then, $\Gamma_\rho(\gamma,\sigma)$ is invariant by $\widecheck{R}$. 
Given $T\in \mathcal{X}_{r}$ for some $r$  we have 
\begin{equation*}
 \big \|T\big (\widecheck{R}^j(v,\psi)) \big \|  \leq \|T\|_{r}  \|\widecheck{R}_v^j(v,\psi)\| ^{r} \leq \|T\|_r \frac{\|v\|^{r }}{(1+ja^*   \|v\|^{N-1})^{\frac{r}{N-1}}}.
\end{equation*}  
From \eqref{decompositionDF} we also introduce $\mathbf{M}_1=\mathrm{Id}+ \CC(\zeta)$. Then, 
$$
\mathbf{M}^{-1}=
\left (\begin{array}{cc}
\mathbf{M}_1^{-1} & 0 \\   
0 &  \mathrm{Id}
\end{array}\right ).
$$
Now let $T=(T_{\xi},T_{\eta},T_\varphi)\in \XX^\times_{s+N-1} $.
From the definition of $\mathcal{S}$ in \eqref{def:Slinearanalytic}
we also have
\begin{align*}
(\mathcal{S}[T])_{\xi,\eta}(v,\psi) & =  \sum_{j=0}^\infty \left [ \prod_{l=0}^j \big (\mathbf{M}_1 ( \widehat{K}_{\xi,\eta}^{\leq q-1}( \widecheck{R}_v^l (v)))\big )^{-1}\right ]   T_{\xi,\eta}  \circ \widecheck{R}^j(v,\psi) , 
\\
(\mathcal{S}[T])_{\varphi}(v,\psi) & =  \sum_{j=0}^\infty    T_{\varphi}  \circ \widecheck{R}^j(v,\psi).
\end{align*}
Now, using~\eqref{propietatCanalytic} and that 
$\widehat{K}^{\leq q-1}_{\xi,\eta}(v)-(v,0)=\OO(\|v\|^2)$, we bound 
\begin{equation}\label{bound:Manalytic}
\left \|\prod_{l=0}^j \big (\mathbf{M}_1\circ \widehat{K}^{\leq q-1}_{\xi,\eta} \circ \widecheck{R}_v^l (v)\big )^{-1} \right \| \leq \prod_{l=0}^j(1+E'\|\widecheck{R}_v^l(v)\|^{N-1}) .
\end{equation}  
To bound \eqref{bound:Manalytic} we will use the formal identity
$ \prod r_j = \exp \sum \log  r_j$, for $r_j>0$.
Again by Lemma~\ref{lem:calR},  
\begin{align*}
\sum_{l=0}^j\log (1+E'\|\widecheck{R}_v^l(v)\|^{N-1}) &\leq E'\sum_{l=0}^j \|\widecheck{R}_v^l(v)\|^{N-1} 
\leq E'\|v\|^{N-1} \sum_{l=0}^j \frac{1}{1+la^*   \|v\|^{N-1}}
\\& \leq E'\|v\|^{N-1} + \frac{E'}{a^*  } \log \big (1+ j a^*   \|v\|^{N-1}).
\end{align*}
Therefore,
$$
\left \|\prod_{l=0}^j \big (\mathbf{M}_1\circ \widehat{K}^{\leq q-1}_{\xi,\eta} \circ R_v^l (v)\big )^{-1}\right \| 
\leq \exp(E'\rho^{N-1}) (1+ja^*   \|v\|^{N-1})^{E'/a^*  }.
$$
Then, using that 
$s/(N-1) -E'/a^*   >1+E^*/((N-5/3)a_f) -E'/((N-4/3)a_f) >1 $, 
by Lemma \ref{fitadeEprima}, 
and that  $T_\xi,\, T_\eta \in \XX_{s}$, 
we obtain 
\begin{align*}
\|(\mathcal{S }[T])_{\xi,\eta}(v,\psi)\| \leq   \|T_{\xi,\eta}\|_{s}  \|v\|^{s} \sum_{j=0}^\infty 
 \frac{\exp(E'\rho^{N-1})}{(1+ja^*   \|v\|^{N-1})^{\frac{s}{N-1}-\frac{E}{a^*  }}} \leq 
 \Mg \|T_{\xi,\eta}\|_{s+N-1}  \|v\|^{s-N+1},
\end{align*}
and similarly, since  $T_\varphi \in \XX_{s-N+P}$,
\begin{align*}
\|(\mathcal{S }[T])_{\varphi}(v,\psi)\| \leq  \Mg \|T_{\varphi}\|_{s-N+P}  \|v\|^{s-2N+P-1 }. 
\end{align*}
Then, we immediately get 
$$
\|\mathcal{S }[T]\|^\times_{s}  \leq  \Mg \|T\|^\times_{s+N-1} .
$$
\end{proof}


\subsubsection{The nonlinear operator $\mathcal{N}$} \label{sec:nonlinearoperatordif}

We denote by $\overline{B}^\times_{r,\delta}$  the closed ball of radius $\delta$ of $\mathcal{X}^\times _r$.
\begin{lemma}\label{lem:nonlinearpart_fixedpoint}
Assume $q\ge N$ and $\delta$ is so small that if $h \in \overline{B}^\times_{q,\delta} $ the range of
$\widehat{K}^\leq+ {h}$ is contained in the domain of  $\widehat{F}$.
Then, if $\delta$ is small, the operator $\NN$ sends the ball
$\overline{B}^\times_{q,\delta}\subset \mathcal{X}^\times _{q}$ into $\mathcal{X}^\times _{q+N-1}$. 
Moreover, if $q\ge \max \{ 2N-P, 2N-M+1\}$ and $\nu =\sqrt{\rho}$, 
$$\Lip \mathcal{N}_{ \mid \overline{B}^\times _{q,\delta}} \le \MM (\rho^{1/2} + \rho). $$
\end{lemma}

\begin{proof}
Let $h\in \XX^\times_q$.
Note that the condition on $q$ implies $q\ge N+1$.
Taking into account the definition of $\NN$ in \eqref{def:Nonlinearmap} we decompose $\NN(h) = N_1+N_2+N_3+N_4$ with 
\begin{align*}
N_1 &= \widehat{F}^{\leq q -1} \circ \widehat{K}^\leq - \widehat{K}^\leq\circ  \widecheck{R}, \\
N_2 &=   \widehat{F}^{\geq q} \circ (\widehat{K}^\leq+ {h}),\\
N_3 & = D\widehat{F}^{\le q-1}(\widehat{K}^{\leq})h - \mathbf{M}(\widehat{K}^{\leq q-1} ) h, \\
N_4 & = \widehat{F}^{\leq q -1} \circ (\widehat{K}^\leq+ h) - \widehat{F}^{\leq q -1} \circ \widehat{K}^\leq - D\widehat{F}^{\leq q -1}(\widehat{K}^\leq)h .  
\end{align*} 
Since $\widehat{R}$  and $\widehat{K}^\leq$ satisfy the approximate invariance equation~\eqref{eq:approximatedinvariancewidehat} and  $\widehat{R}-\widecheck{R}=\OO(\|v\|^q)$, $N_1 \in \mathcal{X}_{q} \times \mathcal{X}_{q} \times \mathcal{X}_{q}\subset \mathcal{X}^\times _{q+N-1}$. 
Clearly, we also have
 $N_2 \in \mathcal{X}_{q} \times \mathcal{X}_{q} \times \mathcal{X}_{q}$.
On the other hand, from \eqref{decompositionDF} 
$$
N_3=  \left (\begin{array}{cc}
0 & 0 \\ 
\cbf(\widehat{K}^{\leq }_{\xi,\eta})  &  0
\end{array}\right ) \left (\begin{array}{c}
h_{\xi,\eta}  \\ 
h_{\varphi}
\end{array}\right ) 
+
\left (\begin{array}{cc}
[\CC (\widehat{K}_{\xi,\eta}^\leq)  - 
\CC (\widehat{K}^{\leq q -1}_{\xi,\eta})]
& 0
\\
0 & 0 \end{array}\right ) 
\left (\begin{array}{c}
h_{\xi,\eta}  \\ 
h_{\varphi}
\end{array}\right ) .
$$
Since
\begin{equation} \label{ordreincrementC}
\CC (\widehat{K}_{\xi,\eta})  - 
\CC (\widehat{K}^{\leq q -1}_{\xi,\eta})
= 
\left (\begin{array}{cc}
\OO_{2q-1}& \OO_{2q-1}  
\\
\OO_{2q-N+M-1} & \OO_{2q-N+M-1} 
\end{array}\right )
\end{equation}
and using the conditions on $q$, we have
$N_3\in \mathcal{X}_{q} \times \mathcal{X}_{q} \times \mathcal{X}_{P-1+q-N+1} = \mathcal{X}^\times _{q+N-1} $.

For $N_4$, we write
$$
(N_4)_{\xi,\eta,\varphi } = \frac12 \int_0^1 (1-t) D^2\wh F^{\leq q -1} _{\xi,\eta,\varphi }
(\widehat{K}^\leq+ t\, h) h^{\otimes 2} \, dt.
$$
Using that $\wh F^{\leq q -1} -(0,0,\varphi)$ does not depend on 
$\varphi$, we have $(N_4)_{\xi}\in \XX_{q+(q-N)}$, 
$(N_4)_{\eta}\in \XX_{q+(q+M-2N)}$
and 
$(N_4)_{\varphi}\in \XX_{q-N+P+(q-N)}$.
Then,  $\NN(h) \in \mathcal{X}^\times _{q+N-1}$.

Now we look for the Lipschitz constant of $\mathcal{N}$ restricted to $\overline{B}^\times _{q,\delta}$. 
Given $h,g\in \overline{B}^\times _{q,\delta}\subset \XX^\times _q $
we decompose $\NN (h) -\NN (g) = T_1 + T_2 + T_3 +T_4+T_5$ with 
\begin{align*}
T_1 
& = \widehat{F}^{\geq q} \circ (\widehat{K}^\leq+ {h})-\widehat{F}^{\geq q} \circ (\widehat{K}^\leq+ g),
\\
T_2 
&= D\widehat{F}^{\le q-1}(\widehat{K}^{\leq})h - \mathbf{M}(\widehat{K}^{\leq q-1} ) h -[D\widehat{F}^{\le q-1}(\widehat{K}^{\leq})g - \mathbf{M}(\widehat{K}^{\leq q-1} )g], 
\\
T_3 &= 
\frac12 \int_0^1 (1-t) [D^2\wh F^{\leq q -1} _{\xi,\eta,\varphi }
(\widehat{K}^\leq+ t\, h)
- D^2\wh F^{\leq q -1} _{\xi,\eta,\varphi }
(\widehat{K}^\leq+ t\, g)] h^{\otimes 2} \, dt,
\\ 
T_4 &=
\frac12 \int_0^1 (1-t) D^2\wh F^{\leq q -1} _{\xi,\eta,\varphi }
(\widehat{K}^\leq+ t\, g)  (h,h - g  )\, dt,
\\ 
T_5 &=
\frac12 \int_0^1 (1-t) D^2\wh F^{\leq q -1} _{\xi,\eta,\varphi }
(\widehat{K}^\leq+ t\, g)  (h - g,g )\, dt.
\end{align*}
We have
$$
T_1 = 
\left (\begin{array}{ccc}
\OO _{q-1} & \OO_{q-1} &  \OO_{q}\\
\OO _{q-1} & \OO_{q-1} &  \OO_{q}\\
\OO _{q-1} & \OO_{q-1} &  \OO_{q}
\end{array}\right )  
\left (\begin{array}{c}
h_\xi -g_\xi\\
h_\eta -g_\eta\\
h_\varphi -g_\varphi
\end{array}\right ),
$$
where $\OO_k$ stands for terms of order $k$ in $v$.
Therefore,
\begin{align*}
\|(T_1)_{\xi,\eta,\varphi }(v,\psi)\|\le & \MM \|v\|^{q-1}(\|h_\xi -g_\xi\|_{q-N+1}+\|h_\eta -g_\eta\|_{q-N+1}) \|v\|^{q-N+1}
\\
& + \MM \|v\|^{q}\|h_\varphi -g_\varphi\|_{q-2N+P+1} \|v\|^{q-2N+P+1}
\end{align*}
and hence 
$$
\|T_1\|^\times_{q+N-1} \le  \MM (\rho^{q-N}+\nu^{-1}\rho^{q-2N+P+1}+\nu \rho^{q-P}+\rho^{q-N+1})  \|h -g\|^\times_{q}. 
$$

From the definition of $\mathbf{M}$ we have 
$$
T_2= T_2^1+T_2^2 := 
\left (\begin{array}{l}
0 \\
0 \\
\cbf(\widehat{K}^{\leq}_{\xi,\eta}) 
\big (h_{\xi,\eta} -g_{\xi,\eta} \big )
\end{array}\right )  
+
\left (\begin{array}{l}
[\CC (\widehat{K}_{\xi,\eta})  - 
\CC (\widehat{K}^{\leq q -1}_{\xi,\eta})]
\big (h_{\xi,\eta} -g_{\xi,\eta}\big  )
\\
0 \end{array}\right )  .
$$
Therefore,
 $$
\|(T_2^1)_\varphi (v,\psi)\| \le \MM \|v\|^{P-1}(\|h_\xi -g_\xi\|_{q-N+1}+\|h_\eta -g_\eta\|_{q-N+1}) \|v\|^{q-N+1}  
 $$
and then $\|T_2^1\|^\times_{q+N-1} \le \MM \nu  \|h -g\|^\times_{q} $.
Taking into account \eqref{ordreincrementC} 
we also have
$$
\|(T_2^2)_\xi (v,\psi)\|  \le \MM \|v\|^{q+N-2} (\|h_\xi -g_\xi\|_{q-N+1} +\|h_\eta -g_\eta\|_{q-N+1}) \|v\|^{q-N+1}  
$$
and then 
$
\|(T_2^2)_\xi\|_{q} \le \MM \rho^{q-1} \|h -g\|^\times_q.
$
Analogously, 
$\|(T_2^2)_\eta\|_{q} \le \MM \rho^{q+M-N-1} \|h -g\|^\times_q$.
Then,
$$
\|T_2^2\|^\times_{q+N-1} \le \MM \rho^{q+M-N-1} \|h -g\|^\times_q.
$$
Next, we recall that $\wh F^{\leq q -1} - (0,0,\varphi)$ does not depend on 
$\varphi$.  
Concerning $T_3$, using the third derivatives of $\widehat F^{\leq q -1}$ and the conditions on $q$, 
\begin{align*}
\|(T_3)_{\xi}(v,\psi)\| \le &   \MM \|v\|^{N-3}  
 (\|h_\xi -g_\xi\|_{q-N+1} \\ &+\|h_\eta -g_\eta\|_{q-N+1}) \|v\|^{q-N+1} 
 \|h_{\xi,\eta} \|^2_{q-N+1}  \|v\|^{2(q-N+1)} 
 \end{align*}
 and $\|(T_3)_{\xi}\|_q \le    \MM  \rho^{2(q-N)} \delta^2
\|h  -g \|^\times_{q}  $.
Analogously, 
$
\|(T_3)_{\eta}\|_q    \le    \MM \rho^{2q+M-3N} \delta^2
\|h  -g \|^\times_{q} 
$
and 
$
\|(T_3)_{\varphi}\|_{q+P-N}    \le    \MM \rho^{2(q-N)} \delta^2
\|h  -g \|^\times_{q} 
$
so that 
$$
\|T_3\|^\times_{q+N+1} \le \MM (\rho^{2q+M-3N}+\nu \rho^{2(q-N)}  )\delta ^2 \|h -g\|^\times_q.
$$
For $T_4$, 
$$
\|(T_4)_{\xi}(v,\psi)\| \le   \MM \|v\|^{N-2}   
 (\|h_\xi -g_\xi\|_{q-N+1} +\|h_\eta -g_\eta\|_{q-N+1}) \|v\|^{q-N+1} 
 \|h_{\xi,\eta} \|_{q-N+1}  \|v\|^{q-N+1} 
$$
and 
$
\|(T_4)_{\xi}\|_q    \le    \MM \rho^{q-N}  \delta
\|h  -g \|^\times_{q}. 
$
Analogously, 
$
\|(T_4)_{\eta}\|_q \le    \MM \rho^{q+M-2N} \delta
\|h  -g \|^\times_{q} 
$ and 
$
\|(T_4)_{\varphi}\|_{q+P-N} \le  \MM \rho^{q-N} \delta
\|h  -g \|^\times_{q} 
$
and
$$
\|T_4\|^\times_{q+N-1} \le \MM  ((1+\nu)\rho^{q-N}+ \rho^{q+M-2N} )\delta  \|h -g\|^\times_q.
$$
For $T_5$ we have the same estimate as for $T_4$. Taking into account the conditions on $q$ and that $\nu=\sqrt{\rho}$ we get the bound for the Lipschitz constant of $\NN$.

\end{proof}

\subsubsection{End of the proof of Proposition \ref{prop:existencepartialmap} }

Our goal is to prove that the fixed point equation~\eqref{eq:fixedpoint2} has  a solution belonging to $\mathcal{X}^\times_q$. For that we start by estimating  the first iterate of the operator $\F=-\mathcal{S} \circ \mathcal{N}$, starting with $\widecheck{\Delta}_0=0$, namely
$$
\widecheck{\Delta}_1=\F(0) = -\mathcal{S} \circ \mathcal{N}[0]. 
$$
We recall that $\widehat{R}$  and $\widehat{K}^\leq$ satisfy the approximate invariance equation~\eqref{eq:approximatedinvariancewidehat} and that $\widehat{R}-\widecheck{R}=\OO(\|v\|^q)$.
Therefore, using  definition  of $\mathcal{N}$ in~\eqref{def:Nonlinearmap} we have that
$$
 \mathcal{N}[0] =  \widehat{F}^{\leq q-1} \circ \widehat{K}^{\leq } - \widehat{K}^{\leq}\circ \widecheck{R} + \widehat{F}^{\geq q} \circ  \widehat{K}^\leq = \widehat{F}\circ \widehat{K}^\leq - \widehat{K}^\leq \circ \widehat{R} 
 = \OO(\|v\|^q)
$$
and as a consequence $\mathcal{N}[0]\in \mathcal{X}_{q} \times \mathcal{X}_{q} \times \mathcal{X}_{q}$. By Lemmas~\ref{lem:linearoperatordif} and \ref{lem:nonlinearpart_fixedpoint},
$\widecheck{\Delta}_1 \in \mathcal{X}^\times_q$.
We introduce the radius 
$$
\delta_0 :=2 \|\widecheck{\Delta}_1\|^\times_q
$$ 
and the closed ball $\overline{B}^\times _{q,\delta_0}$ of $\mathcal{X}^\times_q$ of radius $\delta_0$.

A standard argument shows that if $\rho $ is small, 
$\F (\overline{B}^\times _{q,\delta_0} ) \subset \overline{B}^\times _{q,\delta_0}$.
Indeed, let $\widecheck{\Delta} \in \overline{B}^\times _{q,\delta_0} $. Then, by Lemma~\ref{lem:nonlinearpart_fixedpoint},
\begin{align*}
\| \F (\widecheck{\Delta})\|^\times _q 
& \le \| \F (\widecheck{\Delta})-\F(0)\|^\times _q   +\|\F(0)\|^\times _q \\
& \le \Lip \F \|\widecheck{\Delta} \|^\times_q + \delta_0/2 
\le \MM \|\mathcal{S}\| (\sqrt{\rho} + \rho) \delta_0 +  \delta_0/2 \le \delta_0,
\end{align*}
if  $\rho $ is small.
Therefore we have a unique fixed point $\widecheck{\Delta}$
of $\F$ in $\overline{B}^\times _{q,\delta_0} \subset \mathcal{X}^\times_q $.

 
\subsection{Characterization of the stable manifold}
\label{sec:proofthexistence}
 
To finish the proof of Theorem~\ref{thm:posterioriresult} it remains to relate the parametrization $K(u,\v)$ with  $W^{\mathrm{s}}_{V_{\rho,\,\beta}}$. 
Assume that $K$ and $R$ are solutions of
$$
F\circ K - K\circ R=0
$$
with 
\begin{equation}
\label{ordres-K}
K_x(u,\v)-u = \mathcal{O}(\|u\|^2), \qquad
K_y(u,\v) = \mathcal{O}(\|u\|^2), 
\qquad K_\th(u,\v)- \v = \mathcal{O}(\|u\|),
\end{equation}
and
\begin{equation*}
R_u(u,\v)=u+\overline{f}^N(u,0)+\OO(\|u\|^{N+1}),
\qquad 
R_\v(u,\v)=\v+ \omega + \OO(\|u\|).
\end{equation*}

We first recall that by Proposition~\ref{prop:preliminariesaproximationmap},  performing several steps of averaging and changes of variables we can remove the dependence on $\th$ of $F$ up to any order. Moreover, we also have that the parametrization $K(u,\Theta)-(u,0,\Theta)$ and $R$ do not depend on $\Theta$ up to any order. 
We assume that we have removed this dependence up to order smaller or equal than $N$. In particular, after the corresponding change of variables, the new map $\widehat{F}$ reads as~\eqref{system} with $\widehat{f}^{\geq N}, \widehat{g}^{\geq M}$ as in~\eqref{sumhomogeneousfunctions} satisfying that 
$\widehat{f}^N=\overline{f}^N$ (the average of $f^N$) and $\widehat{g}^M= \overline{g}^M$ are functions independent of ${\th}$. Then, the new map has the same constants $a_f,b_f,A_f,B_g$ as the initial one.

We prove \eqref{darrerstatementdethm28} for $\widehat F$.
From now on, we remove the symbol $\;\widehat{\,}\;$ in the notation. Then, undoing the changes of variables, we have the claim for $F$. We first prove that, if $\rho,\beta$ are small,
$$
K \big (V_\rho \times \T^d\big ) \subset W^s_{\Vext_{\rho,\beta}},
$$
with $V_{\rho,\beta}$ defined in~\eqref{def:Vrhobeta}. 
We claim that, for $u\in V_\rho$ and $\v \in \T^d$,
\begin{equation}
\label{condicionsDSM}
F_{x}(K(u,\v))\in V_\rho \qquad \text{and} \qquad 
\| F_{y}(K(u,\v)) \| \le \beta \| F_{x}(K(u,\v))\| .
\end{equation}
Indeed, since $F_x(x,y,\theta) = x+\overline{f}^N(x,0) + 
\int_0^1 D_y\overline{f}^N(x,s y) y\, ds + f^{\ge N+1}(x,y,\th) $, using that $\overline{f}^N$ satisfies condition (v), namely 
$ \dist (x+\overline{f}^N(x,0), V_\rho^c) \ge a_V\|x\| ^N $, 
and \eqref{ordres-K} we easily obtain 
$$
\mathrm{dist} (F_x(K(u,\v)), V_\rho^c) \geq a_V |u|^N - \mathcal{M}   |u|^{N+1}
$$
and we conclude that $F_{x}(K(u,\v))\in V_\rho$. 
Also,
$$F_{y}(K(u,\v)) = K_y(u,\v)+ \overline{g}^M(K_{x,y}(u,\v))
+g^{\ge M+1}(K(u,\v)) = \OO(\|u\|^2)
$$
and 
$$\|F_{x}(K(u,\v))\| = \|K_x(u,\v)+ \overline{f}^N(K_{x,y}(u,\v))
+ {f}^{\ge N+1}(K(u,\v))\| \ge  \|u \| (1-\MM \|u\|)
$$
give the second condition in \eqref{condicionsDSM}.
Next we notice that, by the definition of $a_f$ in~\eqref{defa},
\begin{align*}
\|R_u(u,\v)\| &\leq \|u + \overline{f}^N(u,0)\| + \MM \|u\|^{N+1} \leq \|u\|- a_f\|u\|^{N} +\M\|u\|^{N+1} \\ & \leq \|u\| \left (1 - \frac{a_f}{2}\|u\|^{N-1}\right ) < \|u\|, 
\end{align*}
for $(u,\v) \in \big (V_\rho\setminus \{0\}\big ) \times \T^d$, 
if $\rho$ small. Using Lemma \ref{lem:calR} and an induction argument we get $\|R^{k+1}_u(u,\v)\|<\|R^k_u(u,\v)\| < \|u \|$ for $k\geq 1$ and thus $R^k_u(u,\v) \to 0$ as $k\to \infty$. Therefore, since $F^k \circ K = K\circ R^k$, for all $k\in \mathbb{N}$, we have that 
$$
\lim_{k\to \infty} F^k_{x,y}(K(u,\v))=0
$$
and this implies that $
K\big ( V_\rho \times \T^d \big ) \subset W_{\Vext_{\rho,\beta}}^{\mathrm{s}}. 
$ 

Now, assuming $B_g>0$ we will prove that, for a cone set $\widehat V$ close to $V$,  
$$
W_{\widehat \Vext_{\rho,\beta}}^{\mathrm{s}} \subset K\big (\widehat  V_\rho \times \T^d \big ). 
$$
To simplify the arguments we first check that the image of $K$ is (locally) the graph of a function $\mathcal{K}$ and then we will change variables to put the graph of $\mathcal{K}$ on the horizontal subspace.   
To check  that $\{K(u,\v)\}$ can be expressed as a graph we note that \eqref{ordres-K} implies that 
the map $K_{x,\theta}:(u, \v)\mapsto (K_x(u, \v), K_\theta(u, \v)) $
is locally invertible and hence we can write $(u,\v)= (K_{x,\th})^{-1} (x,\th)$ for $x$ in a slightly smaller cone set $\widehat V$. Therefore 
$\{K(u,\v)\mid\,(u,\v) \in \widehat V \times \T^d \}$, can be expressed as the graph  $y=K_{y} \circ K^{-1}_{x,\th}(x,\th)$ which has the same regularity as $K$. 
We define $\mathcal{K}=K_y \circ K^{-1}_{x,\th}$.
We emphasize that $\mathcal{K}$ does not depend on $\th$ till terms of order $N+1$.
Since $K_y(u,\v)=\mathcal{O}(\|u\|^2)$, it is also clear that $\|\KK (x,\th)\|= \mathcal{O}(\|x\|^2)$ so that, for $\rho$ small enough,  
$(x,\KK (x,\th),\th) \in \widehat \Vext_{\rho,\beta}=\widehat V_{\rho,\beta} \times \T^d$. 

We perform the close to the identity change of variables
$$
(x,y,\th)=(\xi,\eta +\mathcal{K}(\xi,\vartheta), \vartheta). 
$$
We have that $\mathcal{K}(\xi,\vartheta)=\mathcal{O}(\|\xi\|^2)$ and   $ D_\th \mathcal{K} (\xi,\vartheta)= \mathcal{O}(\|\xi\|^{N+1})$. The $\eta $-component of the transformed map $\mathcal{F}$ is given by
$$
\mathcal{F}_\eta (\xi,\eta,\vartheta)=\eta + \overline{g}^{M}(\xi,\eta)  + \widehat g^{\geq M+1}(\xi,\eta,\vartheta)
$$
for some $\widehat g^{\geq M+1}$. We have $\overline{g}^M(\xi,0)=0$ and $\widehat g^{\geq M+1}(\xi,0,\vartheta)=0$. Therefore, 
$$
\overline{g}^{M}(\xi,\eta) = G(\xi,\eta) \eta , \quad G(\xi,\eta)=\int_{0}^1 D_y \overline{g}^M(\xi, s \eta) \, ds,
\quad 
\widehat{g}^{\ge M+1}(\xi,\eta,\vartheta) = \widehat G(\xi,\eta,\vartheta) \eta.
$$

It is clear that $\eta =0$ corresponds to $y=\mathcal{K}(x,\th)$.
Therefore, it only remains to be checked that, if $(\xi,\eta,\vartheta)$ are such that 
$\mathcal{F}_{\xi,\eta}^k(\xi,\eta ,\vartheta) \in \widehat{V}_{\rho,\beta}$ for all $k\in \mathbb{N}$ and $\mathcal{F}_{\xi,\eta}^k(\xi,\eta ,\vartheta) \to 0$ then $\eta=0$. Indeed, by the definition of $B_g$ in~\eqref{defABg}, we have that 
\begin{align*}
\|\mathcal{F}_\eta^{-1} (\xi,\eta,\vartheta) \| & \leq \| \eta - G(\xi,\eta) \eta \| + \Mg\|\eta\| \|\xi\|^{M} \\ &\leq \|\Id - G(\xi,0) \|  \|\eta\| + \Mg\|\xi\|^{M-2} \|\eta\|^2 + \Mg \|\xi \|^{M} \|\eta\|
\\ & \leq \|\eta \| \big (1- B_g  \|\xi\|^{M-1} + \Mg(\rho + \beta) \|\xi\|^{M-1} \big ) \leq \|\eta\|,
\end{align*}
if $\rho,\beta$ are small enough. 
Therefore, $\|\mathcal{F}_\eta (\xi,\eta,\vartheta)\| \geq \|\eta\|$ if 
$(\xi,\eta,\vartheta) \in \widehat V_{\rho,\beta} \times \T^d$. Applying  this property in a  iterative way we have that, when
$\mathcal{F}_{\xi,\eta}^k(\xi,\eta,\vartheta) \in \widehat V_{\rho,\beta}$ for all $k\in \N$, then 
$$
\|\eta\| \leq \|\mathcal{F}_\eta ^k(\xi,\eta,\vartheta) \| \to 0 \qquad \text{as }\quad k\to \infty
$$
and, consequently, $\|\eta\|=0$.
 
 
\section{Approximation of the invariant manifolds}  \label{sec:approximation}
This section contains the proof of Theorems~\ref{thm:approximationmaps} and~\ref{thm:approximationflows}. First, in Section~\ref{subsectionstep2}, we will consider a first order partial differential equation which we will encounter as a cohomological equation in the inductive step to find the terms of the expansion of the parametrization $K$, and the function $R$ (for maps) or the vector field $Y$ (for differential equations). Then, in Sections~\ref{sec:inductionalgorithm} and~\ref{sec:cohomologicalflow} we prove the approximation results for maps and flows, respectively. We emphasize that we provide an explicit inductive algorithm for computing such approximations as finite sums of homogeneous functions in $u$.

 
\subsection{A first order partial differential equation with homogeneous coefficients}\label{subsectionstep2}

In this section we  recall Theorem 3.2 of~\cite{BFM2020b} which will be  a key result to solve the so-called cohomological equations. In that paper the result is stated for the differentiable and analytical cases. Here we reword it for the analytical case. 

Let $V\subset \R^n$ be a cone-like set, $0\in \partial V$,  $\m\in \Z$,
$\m\ge 1$ and $\Lambda \subset \R^p $. Let 
$\pa:\R^n\times \Lambda\to \R^k$, $\Qa:\R^n \times \Lambda\to \L(\R^k,\R^k)$ and
$\TT:V\times \Lambda\to \R^k$, and consider the equation
\begin{equation}\label{modellinearequation}
D h(\xu,\lambda) \cdot \pa(\xu,\lambda) - \Qa(\xu,\lambda) \cdot   h(\xu,\lambda) = \TT(\xu,\lambda)
\end{equation}
for $h:V\times \Lambda\to \R^k $.
Given $\rho>0$, we define the constants~$\apa, \bpa,\Apa$ and $\BB_\Qa$
by
\begin{equation}
\label{defconstantspa}
\begin{aligned}
\apa &:=-\sup_{ \xu \in V_\rho,\, \lambda\in \Lambda} \frac{\Vert  \xu+ \pa (\xu,\lambda)
\Vert -\Vert  \xu \Vert}{\Vert  \xu \Vert^{N}}, \quad 
&\bpa := \inf_{\lambda \in \Lambda} \sup_{ \xu \in V_{\rho}} \frac{\Vert \pa(\xu,\lambda) \Vert}{\Vert \xu\Vert^{N}},\\
\Apa &:= -\sup_{ \xu \in V_{\rho},\,\lambda \in \Lambda} \frac{\Vert \mathrm{Id}+ D\pa (\xu,\lambda)\Vert-1}{\Vert \xu \Vert^{N-1}}, 
& \BB_\Qa := -\sup_{ \xu\in V_{\rho},\,\lambda \in \Lambda } \frac{\Vert \mathrm{Id}- \Qa (\xu,\lambda)\Vert-1}{\Vert  \xu \Vert^{N-1}}.
\end{aligned}
\end{equation}

We assume there exists $\rho>0$  such that the following conditions hold
\begin{enumerate}[label=(\alph*)] 
    \item   $\pa,\Qa$ and $\TT$ are analytic homogeneous functions
in $V_\rho \times \Lambda$ of degrees $N, N-1$ and $\m+N$,  respectively. 
\item The constants $\apa, \Apa,\bpa$ satisfy
\begin{equation*}
\apa >0, \qquad \Apa > \bpa .
\end{equation*}
\item There exists a constant $ \CIP>0$ such that
\begin{equation*}
 \text{dist}\big( \xu + \pa( \xu,\lambda), \big (V_{\rho}\big )^c\big )\geq \CIP
\Vert
\xu\Vert^{N},\qquad \forall \xu \in V_\rho, \quad \forall \lambda \in \Lambda.
\end{equation*}
\item If $\BB_\Qa <0$ we assume that
\begin{equation}\label{hipgraus}
\m+1+ \frac{\BB_\Qa}{\apa} > 0.
\end{equation}
We will apply the next theorem 
for different $\Qa$'s and in some cases we may have $\BB_{\Qa}< 0$.
\end{enumerate}
We will have to consider complex extensions of $\Omega:=V\times \Lambda$ of the form  
\begin{align*}
\Omega_\C(\gamma)&:=\{(\xu,\lambda) \in \C^n\times \C^p\mid\,  (\Re \xu,\Re \lambda) \in V \times \Lambda, \;\; \Vert \Im \xu
\Vert < \gamma \Vert \Re u \Vert,\; \Vert \Im \lambda \|<\gamma^2 \}.
\end{align*}
Finally, let $\varphi_\pa$ be the flow of $\dot{u}=\pa(u,\lambda)$ and $\Psi_\Qa$ be the fundamental matrix solution of $\dot{z} = \Qa(\varphi_\pa(t;u,\lambda),\lambda) z$ such that $\Psi(0;u,\lambda)=\Id$. 
 
\begin{theorem}[Theorem 3.2 of~\cite{BFM2020b}]\label{thm:solutionlinearequation}
Assume that $\pa,\Qa,\TT$ satisfy hypotheses (a)-(d).
Then, equation \eqref{modellinearequation}
has a unique homogeneous solution of degree $\m+1$ given by 
\begin{equation}\label{defh0}
  h(\xu,\lambda) = \mathcal{H}_{\pa,\Qa}[h] := \int_{\infty}^0 \Psi^{-1}_\Qa( t; \xu,\lambda) \TT(\varphi_\pa( t; \xu,\lambda), \lambda) \,
dt,\qquad  \xu \in V, \quad \lambda \in \Lambda.
\end{equation}

If $\pa,\Qa,\TT$ are real analytic functions defined on the extended set $\Omega_\C(\gamma_0)$ for some $\gamma_0>0$, then there exists $0< \gamma\leq \gamma_0$ such that $h$ is a real analytic function on $\Omega_\C(\gamma) $.  
\end{theorem}

\subsection{The approximate solution of the invariance equation for maps}\label{sec:inductionalgorithm}

To simplify the notation, throughout this section we will not make explicit the dependence of the considered objects on $\lambda$. 
Also, at some places, we skip the dependence on their variables of some functions when it will not be possible confusion. We recall that the superscript in a function, for instance $G^{j}$,  indicates that $G^j$ is a homogeneous function of degree $j$ with respect to $u$ or $(x,y)$, i.e. with respect to all its variables except the angles and parameters. However,
when we use parentheses, $G^{(j)}$ indicates the expression of $G$ at the $j$ step of some iterative procedure.  
In this section we prove Theorem~\ref{thm:approximationmaps} by finding approximations, $K^{(j)},R^{(j)}$, as sums of homogeneous functions that can be  determined. The specific way to do so is precisely described. By an induction procedure, we prove that indeed the functions obtained with the 
proposed algorithm satisfy the approximate invariance equation~\eqref{eq:invthapproximation}.


\subsubsection{Iterative procedure: the cohomological equations}\label{sec:induction}

Although there is some freedom, we look for an approximation   $K^{(j)}=(K_x^{(j)},K_y^{(j)},K_\th^{(j)})$ of the parametrization of the invariant manifold and $R=(R_u^{(j)},R_\v^{(j)})$ of the form:
$$
K^{(1)}(u,\v )=\big (u+\widetilde{K}_x^N(u,\v),0,\v\big ),
\qquad
R^{(1)}(u,\v )=\big (u+\overline{R}_u^N(u),\v+\omega\big ),
$$
and
for $j\geq 2,$ 
$$
K^{(j)}(u,\v )=K^{(j-1)}(u,\v) + \K^{(j)}(u,\v),
\qquad
R^{(j)}(u,\v)= R^{(j-1)}(u,\v) + \RR^{(j)}(u,\v)
$$
with $\K^{(j)}$, $j\geq 2$, decomposed as the sum of an average and an oscillatory part (of different degrees):
\begin{align*}
\K^{(j)}_{x}(u,\v)&=\overline{K}_{x }^j(u)
+\widetilde{K}_{x }^{j+N-1}(u,\v),\\
\K^{(j)}_{y}(u,\v)&=\overline{K}_{y}^j(u)
+\widetilde{K}_{y }^{j+M-1}(u,\v),\\
\K^{(j)}_{\th}(u,\v)&=\overline{K}_{\th}^{j-1}(u)
+\widetilde{K}_{\th}^{j+P-2}(u,\v)    
\end{align*}
and similarly $\RR^{(j)} $ decomposed  as:
$$
\RR^{(j)}_u(u,\v)=\overline{R}^{j+N-1}_u(u) +
\widetilde{R}^{j+N-1}_u(u,\v),
\qquad
\RR^{(j)}_\v(u,\v)=\overline{R}^{j+P-2}_\v(u)
+\widetilde{R}^{j+P-2}_\v(u,\v),
$$
such that,  
\begin{equation}\label{eq_invj}
E^{(j)} := F \circ K^{(j)} - K^{(j)} \circ R^{(j)}
=\big (\OO(\|u\|^{j+N}) , \OO(\|u\|^{j+M}), \OO(\|u\|^{j+P-1})\big ).
\end{equation}
\begin{remark}\label{rem:jNMP}
Notice that property~\eqref{eq_invj} is not~\eqref{eq:invthapproximation} in the statement of Theorem~\ref{thm:approximationmaps}. We will obtain~\eqref{eq:invthapproximation} in Section \ref{sec:endoftheprooapproximation}. 
\end{remark}
 
First we check that the choices of $K^{(1)}$ and $R^{(1)}$ are such that~\eqref{eq_invj} becomes true for $j=1$. Indeed, we write
$$ 
F\circ K^{(1)}  -K^{(1)}\circ R^{(1)} 
=  \left (\begin{array}{c} \widetilde{K}_x^N + \fN(u+\widetilde{K}_x^N,0,\v) - \overline{R}_u^N
-\widetilde{K}_x^N(u+\overline{R}_u^N(u), \v+\omega)
 + \OO(\|u\|^{N+1})\\
\OO(\|u\|^{M+1})
\\
\OO(\|u\|^P)
\end{array}
\right) .
$$
Comparing the average and the oscillatory parts, we are lead to take
$
\overline{R}_u^N(u)=\fNm(u,0)
$
and $\widetilde{K}_x^N$ to be the zero average solution  $\mathcal{D}[\widetilde{f}^N]$ of the small divisors equation
$$
\widetilde{K}_x^N(u,\v+\omega) - \widetilde{K}_x^N(u,\v)
=\fNt(u,0,\v),
$$
and \eqref{eq_invj} holds true for $j=1$.
Assume, by induction, that we have determined $K^{(l)}$ and $ R^{(l)}$ for $1\leq l \leq j-1$, with $j\geq 2$
such that $E^{(j-1)}$ defined in~\eqref{eq_invj} satisfies
$$
E^{(j-1)} = \big(\OO(\|u\|^{j+N-1}) , \OO(\|u\|^{j+M-1}), \OO(\|u\|^{j+P-2})\big ).
$$
We decompose
\begin{align*}
E^{(j)}= & F\circ K^{(j)}-K^{(j)} \circ R^{(j)} \\
=& F\circ K^{(j-1)} - K^{(j-1)} \circ R^{(j-1)} \\
&+ F \circ K^{(j)} - F\circ K^{(j-1)} - DF\circ K^{(j-1)}\cdot \K^{(j)} \\
&+ DF \circ K^{(j-1)} \cdot \K^{(j)} - \K^{(j)} \circ R^{(j-1)} \\
&- K^{(j)}\circ R^{(j)} + K^{(j)}\circ R^{(j-1)},
\end{align*}
and we define
\begin{align*}
\mathcal{T}_1^{(j)} &= F \circ K^{(j)} - F\circ K^{(j-1)} - DF\circ K^{(j-1)}\cdot \K^{(j)}, \\
\mathcal{T}_2^{(j)} &=DF \circ K^{(j-1)} \cdot \K^{(j)} - \K^{(j)} \circ R^{(j-1)}, \\
\mathcal{T}_3^{(j)} &= K^{(j)}\circ R^{(j)}- K^{(j)}\circ R^{(j-1)}.
\end{align*}

Since $F$ can be expressed as a sum of homogeneous functions with respect to $(x,y)$ (condition (ii)), it is not difficult to check that  
$$
\mathcal{T}_1^{(j)} = \big (\OO(\|u\|^{N+j}),\OO( \|u\|^{M+j}), \OO(\|u\|^{P+j})\big ).
$$

Now we deal with $\mathcal{T}_2^{(j)}$. 
Using that $K^{(j-1)}(u,\v)=(u + \OO(\|u\|^2),\OO(\|u\|^2),\v + \OO(\|u\|))$, decomposition~\eqref{sumhomogeneousfunctions} in condition (ii) and condition (iv), we write
\begin{align*}
 DF \circ K^{(j-1)} =& \begin{pmatrix}
\Id + \partial_x (\fN+\fNg) & \partial_y (\fN+\fNg)  &
\partial_\th (\fN+\fNg) \\
 \partial_x (\gN+\gNg) & \Id +\partial_y (\gN+\gNg)  &
\partial_\th (\gN+\gNg) \\
\partial_x (\hP+\hPg) & \partial_y (\hP+\hPg)  &
\Id +\partial_\th (\hP+\hPg)
\end{pmatrix} \circ K^{(j-1)} \\  
=& \begin{pmatrix}
\Id + \partial_x \fN(u,0,\v) & \partial_ y \fN(u,0,\v) & 
\partial_\th \fN (u,0,\v) 
\\ 0 & \Id + \partial_y \gN(u,0,\v) & 0 \\
0 & 0 & \Id 
\end{pmatrix} + \mathbf{N}(u,\v)
\end{align*}
with 
$$
\mathbf{N}(u,\v)=\begin{pmatrix} \OO(\|u\|^{N}) & \OO(\|u\|^N ) & \OO(\|u\|^{N+1})\\ 
 \OO(\|u\|^{M}) & \OO(\|u\|^M ) & \OO(\|u\|^{M+1}) \\ 
\OO(\|u\|^{P-1} ) & \OO(\|u\|^{P-1}) & \OO(\|u\|^{P}) 
\end{pmatrix}.
$$
We note that, by (iv), $\partial_\th \gN \circ K^{(j-1)}= \OO(\|u\|^{M+1})$. 
Then,  
\begin{align*}
\mathcal{T}_2^{(j)} = & 
\begin{pmatrix}
\Id + \partial_x \fN(u,0,\v) & \partial_ y \fN(u,0,\v) & 
\partial_\th \fN (u,0,\v) 
\\ 0 & \Id + \partial_y \gN(u,0,\v) & 0 \\
0 & 0 & \Id 
\end{pmatrix}
\begin{pmatrix}
\overline{K}_x^j  + \widetilde{K}_x^{j+N-1} \\
\overline{K}_y^j  + \widetilde{K}_y^{j+M-1} \\
\overline{K}_\th^{j-1}  + \widetilde{K}_\th^{j+P-2}
\end{pmatrix}
\\
&-
\begin{pmatrix}
\overline{K}_x^j \circ R^{(j-1)}  + \widetilde{K}_x^{j+N-1}\circ R^{(j-1)} \\
\overline{K}_y^j \circ R^{(j-1)}  + \widetilde{K}_y^{j+M-1}\circ R^{(j-1)} \\
\overline{K}_\th^{j-1} \circ R^{(j-1)}  + \widetilde{K}_\th^{j+P-2}\circ R^{(j-1)} \\
\end{pmatrix} + \mathbf{N}(u,\v)\cdot \K^{(j)}.
\end{align*}
Notice that, since $R^{(j-1)}_u (u,\v)= u + \overline{R}^N_u(u) + \OO(\|u\|^{N+1})$ and 
$R^{(j-1)}_\v (u,\v)= \v + \omega + \OO(\|u\|)$, we have
\begin{align*}
\overline{K}_{x,y}^j \circ R^{(j-1)}(u,\v) &= \overline{K}_{x,y}^j (u) + D\overline{K}^j_{x,y}(u) \overline{R}^N_u(u) + \OO(\|u\|^{j+2N-2}), \\
\overline{K}_{\th}^{j-1} \circ R^{(j-1)}(u,\v) &= 
\overline{K}_{\th}^{j-1} (u)+
D \overline{K}_{\th}^{j-1} (u) \overline{R}^N_u(u) + \OO(\|u\|^{j+2N-3}),
\end{align*}
and, writing $\widetilde{\mathcal{K}}^{(j)}= (\widetilde{K}_x^{j+N-1}, \widetilde{K}_y^{j+M-1},\widetilde{K}_\th^{j+P-2})$, 
$$
\widetilde{\mathcal{K}}^{(j)}\circ R^{(j-1)}(u,\v) =
\widetilde{\mathcal{K}}^{(j)}(u,\v + \omega)+\big (\OO(\|u\|^{j+2N-1},\OO (\|u\|^{j+N+M-1}), \OO(\|u\|^{j+N+P-2})\big ).
$$
 
Therefore  
\begin{align*}
\mathcal{T}_{2,x}^{(j)} =& -D \overline{K}_x^j (u) \overline{R}^N_u(u)  + \partial_x \fN(u,0,\v) \overline{K}_x^j(u)
+\partial_y \fN(u,0,\v) \overline{K}_y^j(u) \\ &
+\partial_\th \fN(u,0,\v) (\overline{K}_\th^{j-1}(u) +\widetilde{K}_\th^{j+P-2} (u,\v))\\
&+ \widetilde{K}_x^{j+N-1}(u,\v) -\widetilde{K}_x^{j+N-1} (u,\v+\omega)+  
\OO(\|u\|^{j+N}), \\
\mathcal{T}_{2,y}^{(j)} =& -D \overline{K}_y^j (u) \overline{R}^N_u(u) 
+\partial_y \gN(u,0,\v) \overline{K}_y^j(u)
 + \widetilde{K}_y^{j+M-1}(u,\v) -\widetilde{K}_y^{j+M-1} (u,\v+\omega)  \\ &+  
\OO(\|u\|^{j+M}), \\
\mathcal{T}_{2,\th}^{(j)} =& -D \overline{K}_\th^{j-1} (u) \overline{R}^N_u(u)
 + \widetilde{K}_\th^{j+P-2}(u,\v) -\widetilde{K}_\th^{j+P-2} (u,\v+\omega)+
\OO(\|u\|^{j+P-1}).
\end{align*}
Finally, we write $\mathcal{T}_3^{(j)}$
\begin{align*}
\mathcal{T}_3^{(j)}= &K^{(j)} \circ R^{(j)} - K^{(j)} \circ R^{(j-1)} =
\int_{0}^1 DK^{(j)} \big(R^{(j-1)} + s \RR^{(j)}\big ) \RR^{(j)} \dd s
\\ = &\int_{0}^1
\begin{pmatrix}
1 + \OO(\|u\|) & \partial_\v \widetilde{K}^N_x + \OO(\|u\|^{N+1}) \\
\OO(\|u\|) &   \OO(\|u\|^{M+1}) \\
\OO(1)  & 1 + \OO(\|u\|)
\end{pmatrix}
\begin{pmatrix}
\RR_u^{(j)} \\ \RR_\v^{(j)}
\end{pmatrix} \dd s \\   
 = &\begin{pmatrix}
\overline{R}_u^{j+N-1} +\widetilde{R}_u^{j+N-1}+ 
\partial_\v \widetilde{K}^N_x(\overline{R}_\v^{j+P-2}+\widetilde{R}_\v^{j+P-2})+ \OO(\|u\|^{j+N}) \\  \OO(\|u\|^{j+M}) \\
\overline{R}_\v^{j+P-2}+\widetilde{R}_\v^{j+P-2}  +\OO(\|u\|^{j+P-1})
\end{pmatrix}.
\end{align*}

Since $F$ is expressed as a sum of homogeneous functions until degree $q-1$, we write 
\begin{equation}
\label{Ej-1}    
E^{(j-1)} = \big (E_{x}^{j+N-1}, E_y^{j+M-1}, E_\th^{j+P-2}\big )+ \big (\OO(\|u\|^{j+N}), \OO(\|u\|^{j+M}), \OO(\|u\|^{j+P-1})\big ).
\end{equation}
Therefore, since $E^{(j)}$ has to satisfy~\eqref{eq_invj}, namely:
$$
E^{(j)} =  \big (\OO(\|u\|^{j+N}), \OO(\|u\|^{j+M}), \OO(\|u\|^{j+P-1})\big ),
$$
from the previous estimates,
we impose the corresponding conditions on $E_x^{(j)}, E_y^{(j)}, E_\theta^{(j)}$. That is, for the $x$-component
\begin{equation}\label{firstcohomologicalx}
\begin{aligned}
    E_x^{j+N-1}(u,\v)&
-D \overline{K}_x^j (u) \overline{R}^N_u(u)
  - \overline{R}_u^{j+N-1}(u) - \widetilde{R}_u^{j+N-1}(u,\v)   \\ &+ \partial_x \fN(u,0,\v) \overline{K}_x^j(u)
+\partial_y \fN(u,0,\v) \overline{K}_y^j(u)
+\partial_\th \fN(u,0,\v) \overline{K}_\th^{j-1}(u)  \\
&+ \widetilde{K}_x^{j+N-1}(u,\v) -\widetilde{K}_x^{j+N-1} (u,\v+\omega) + \partial_\th \fN(u,0,\v) \widetilde{K}_\th^{j+P-2} (u,\v)  \\ &
-\partial_\v\widetilde{K}_x^N (u,\v) (\overline{R}_\v^{j+P-2}(u)+\widetilde{R}_\v^{j+P-2}(u,\v))
=\OO(\|u\|^{j+N}).
\end{aligned}
\end{equation}
Concerning the $y$-component
\begin{equation}\label{firstcohomologicaly}
\begin{aligned}
E_y^{j+M-1}(u,\v) 
-D \overline{K}_y^j (u) \overline{R}^N_u(u)
&+\partial_y \gN(u,0,\v) \overline{K}_y^j(u)
 \\
&  +\widetilde{K}_y^{j+M-1}(u,\v) - \widetilde{K}_y^{j+M-1}(u,\v+\omega)
=\OO(\|u\|^{j+M}).        
    \end{aligned}
\end{equation}
And finally, for the $\theta$-component
\begin{equation}\label{firstcohomologicaltheta}
\begin{aligned}
E_\th^{j+P-2}(u,\v)
&-D \overline{K}_\th^{j-1} (u) \overline{R}^N_u(u)
- \overline{R}_\v^{j+P-2}(u)-\widetilde{R}_\v^{j+P-2}(u,\v) 
 \\
  &+\widetilde{K}_\th^{j+P-2}(u,\v) - \widetilde{K}_\th^{j+P-2}(u,\v+\omega)
=\OO(\|u\|^{j+P-1}).        
\end{aligned}
\end{equation}

Now, we explain how to deal with equations~\eqref{firstcohomologicalx}, \eqref{firstcohomologicaly} and \eqref{firstcohomologicaltheta} to obtain the terms $\mathcal{K}^{(j)} $ and $\RR^{(j)}$. 
We introduce some notation. Given a function $G(u,\v) = \OO(\|u\|^\ell)$ that can be expressed as sum of homogeneous functions of integer degree, we write
\begin{equation} \label{Gparentesis}
G(u,\v)=\{G\}^{\ell}(u,\v) + \OO(\|u\|^{\ell+1}),
\end{equation}
where $\{G\}^\ell$ is the homogeneous part of $G$ of its lowest degree. For practical purposes we do not assume $\{G\}^{\ell}$ to be different from zero. We also introduce
\begin{align*}
\mathcal{G}\big [K^{(j)}_{y,\th},R_\v^{(j)}\big ]=&\partial_y \fN(u,0,\v) \overline{K}_y^j(u)
+\partial_\th \fN(u,0,\v) \overline{K}_\th^{j-1}(u) 
+ \partial_\th \fN(u,0,\v) \widetilde{K}_\th^{j+P-2}(u,\v) \\
& 
-\partial_\v\widetilde{K}_x^N (u,\v) (\overline{R}_\v^{j+P-2}(u)+\widetilde{R}_\v^{j+P-2}(u,\v))
\end{align*}
that satisfies $\mathcal{G}\big [K_{u,\th}^{(j)},R_\v^{(j)}\big ] = \OO (\|u\|^{j+N-1})$ since $P\geq 1$ and $\partial_{\v}\widetilde{K}_x^N = \OO (\|u\|^N)$.

Therefore, using that $\overline{R}^N_u(u)= \overline{f}^N(u,0)$, equations~\eqref{firstcohomologicalx}, \eqref{firstcohomologicaly} and \eqref{firstcohomologicaltheta} 
decouple into the triangular system:
\begin{align}
E_y^{j+M-1}(u,\v) +& \left \{
-D \overline{K}_y^j (u) \overline{f}^N(u,0) \right \}^{j+M-1}
+\partial_y \gN(u,0,\v) \overline{K}_y^j(u)
 \notag \\
 +&\widetilde{K}_y^{j+M-1}(u,\v) - \widetilde{K}_y^{j+M-1}(u,\v+\omega)=0,  \label{secondcohomologicaly}\\
E_\th^{j+P-2}(u,\v) +& \left \{
-D \overline{K}_\th^{j-1} (u) \overline{f}^N(u,0)
 \right \}^{j+P-2} - \overline{R}_\v^{j+P-2}(u)-\widetilde{R}_\v^{j+P-2}(u,\v) \notag\\ 
  +&\widetilde{K}_\th^{j+P-2}(u,\v) - \widetilde{K}_\th^{j+P-2}(u,\v+\omega)=0, \label{secondcohomologicaltheta} \\ 
  E_x^{j+N-1}(u,\v) &
-D \overline{K}_x^j (u) \overline{f}^N(u,0)+ \partial_x \fN(u,0,\v) \overline{K}_x^j(u)
  \notag \\& - \overline{R}_u^{j+N-1}(u) - \widetilde{R}_u^{j+N-1}(u,\v)
\notag \\
&+ \widetilde{K}_x^{j+N-1}(u,\v) -\widetilde{K}_x^{j+N-1} (u,\v+\omega) + \left \{\mathcal{G}[K^{(j)}_{y,\th},R_\v^{(j)}]\right \}^{j+N-1}=0. \label{secondcohomologicalx}
\end{align}
These are the so-called cohomological equations. 
To solve these equations we deal separately with the average and the oscillatory parts.
We first deal with~\eqref{secondcohomologicaly}. We distinguish the cases $M<N$ and  $M=N$.
\begin{itemize}
    \item Case $M<N$. Averaging~\eqref{secondcohomologicaly} we obtain 
    $
    \partial_y \overline{g}^M(u,0) \overline{K}_y^j(u)=-\overline{E}^{j+M-1}(u)
    $
    and therefore, since by the hypotheses of Theorem~\ref{thm:approximationmaps},  
    $\partial_y \overline{g}^M(u,0)$ is invertible, 
    \begin{equation}\label{defKyj}
    \overline{K}^{j}_y (u)= -\big (\partial_y \overline{g}^M(u,0) \big )^{-1} \overline{E}_y^{j+M-1}(u).
    \end{equation}
    \item Case $M=N$, the average part of  equation~\eqref{secondcohomologicaly} is 
   \begin{equation} \label{eq:DKy}
D\overline{K}^j_y(u) \overline{f}^N(u,0) - \partial_y \overline{g}^N(u,0) \overline{K}^j_y(u) =  \overline{E}_y^{j+N-1}(u).       
\end{equation}
    This equation is of the form \eqref{modellinearequation}, therefore
we apply Theorem~\ref{thm:solutionlinearequation} with $\Qa(u)=\partial_y \overline{g}^N(u,0)$ and $\pa(u)=\overline{f}^N(u,0)$ with the associated constants 
$a_\pa=a_f$, $b_\pa=b_f$, $A_\pa=A_f$ and 
$\BB_\Qa=B_g$ defined in~\eqref{defa}, \eqref{defbCf} and~\eqref{defABg}, respectively.
Note that, by (iv) the domain with respect to $u$ of $\overline{f}^N(u,0) $ and $\partial_y \overline{g}^N(u,0) $ can be extended to $\R^n $ by homogeneity.

    By condition~\eqref{Ej-1} and Theorem~\ref{thm:solutionlinearequation} with $ \m= j-1$ the solution of \eqref{eq:DKy} is
    $$
    \overline{K}^{j}_y = \mathcal{H}_{\pa,\Qa} \left [ \overline{E}_y^{j+N-1}\right ],\qquad \text{with } \quad \pa= \overline {f}^N(u,0),  \quad \Qa= \partial_y \overline{g}^N(u,0),
    $$
    where $\mathcal{H}_{\pa,\Qa}$ is defined in \eqref{defh0}.
\end{itemize}
In both cases the oscillatory part of~\eqref{secondcohomologicaly} is solved as a small divisors equation, using Theorem \ref{thm:smalldivisors} to an extension of the involved functions to a complex neighbourhood of their domain. 
We have
    \begin{equation}\label{defKtildey}
    \widetilde{K}_y^{j+M-1} = \mathcal{D} \big [ \widetilde{E}_y^{j+M-1}  + \partial_y \widetilde{g}^M \overline{K}^{j}_y \big ], 
    \end{equation}
where $\mathcal{D}$ is introduced in Section \ref{sect:Diophan}.
\begin{remark}  
A remarkable fact is that, once $\overline{K}^j_y$ and  $\widetilde{K}^{j+M-1}_y$ are found,  equations~\eqref{secondcohomologicaltheta} and~\eqref{secondcohomologicalx} always have solution. For instance we can choose   
\begin{equation}\label{solnotrivial1}
\overline{K}_\th^{j-1}, \; \overline{K}_x^j =0, \qquad \widetilde{R}_\v^{j+P-2}, \; \widetilde{R}_u^{j+N-1}=0, 
\end{equation}
\begin{equation}\label{solnotrivial2}
\begin{aligned}
\overline{R}^{j+N-1}_u &= \overline{E}_x^{j+N-1}+ \left \{\overline{\mathcal{G}}[K^{(j)}_{y,\th},R_\v^{(j)}]\right \}^{j+N-1} , \\ 
\overline{R}^{j+P-2}_{\v} &= \overline{E}^{j+P-2}_\th 
 \end{aligned}
\end{equation}
and
\begin{equation*}
 \widetilde{K}_\th^{j+P-2} = \mathcal{D}[\widetilde{E}_\th^{j+P-2}],\qquad  \widetilde{K}_x^{j+N-1} = \mathcal{D}[\widetilde{E}_x ^{j+N-1}
+ \left \{\widetilde{\mathcal{G}}[K^{(j)}_{y,\th},R_\v^{(j)}]\right \}^{j+N-1}].
\end{equation*}
We notice that with this choice all the involved functions keep the same regularity as $F$, $\overline{K}_y^j$, and  $ \widetilde{K}_y^{j+M-1}$. 

However, we want to go further and keep $R$ as simple as possible. That is, we want to take, whenever possible, $R^{j+P-2}_\v$ and $R_u^{j+N-1}$ equal to $0$.
\end{remark}

Before starting solving~\eqref{secondcohomologicaltheta} and~\eqref{secondcohomologicalx} let us  say some words about the regularity of $K^{(j)}_{x,\theta}$ and $R^{(j)}_{u,\v}$.
\begin{remark}
In  Theorem~\ref{thm:solutionlinearequation}, if instead of condition (b) we have $A_\pa< b_\pa$, we cannot conclude that the solution of equation~\eqref{modellinearequation} has the same regularity as $\pa$ and $\Qa$. This is an optimal general condition as it was shown in Section 6 of~\cite{BFM2020b},  where some examples showing the loss of regularity were provided. 

However, when $M<N$, the functions $\overline{K}_y^j$ and  $ \widetilde{K}_y^{j+M-1}$ defined in~\eqref{defKyj} and~\eqref{defKtildey} are analytic. Therefore, in this case, when solving~\eqref{secondcohomologicaltheta} and~\eqref{secondcohomologicalx}, if $A_f\leq b_f$, 
to have analytic solutions of~\eqref{secondcohomologicaltheta} and~\eqref{secondcohomologicalx}
we  use the expressions \eqref{solnotrivial1} and~\eqref{solnotrivial2},
\end{remark}
After this remark we continue with the assumption that $A_f>b_f$.

The following analysis discusses how to obtain solutions with the simplest possible $R$. 
We solve first~\eqref{secondcohomologicaltheta}. We take
$$
\widetilde{K}_\th^{j+P-2} =  \mathcal{D}\big [\widetilde{E}_\th^{j+P-2} \big], \qquad \widetilde{R}_\v^{j+P-2}= 0.
$$
Then,  equation~\eqref{secondcohomologicaltheta} becomes
$$
\overline{E}_\th^{j+P-2}(u) = \left \{
D \overline{K}_\th^{j-1} (u) \overline{f}^N(u,0)
 \right \}^{j+P-2} + \overline{R}_\v^{j+P-2}(u).
$$
We distinguish the cases $P<N$ and $P=N$. 
\begin{itemize}
    \item Case $P<N$. The expression  $\left \{
D \overline{K}_\th^{j-1} (u) \overline{f}^N(u,0)
 \right \}^{j+P-2} =0$ and we take
$$
\overline{R}_\v^{j+P-2} = \overline{E}_\th^{j+P-2}, \qquad \overline{K}^{j-1}_\th \text{ free}.
$$
\item Case $P=N$. We have that $\overline{K}_\th^{j-1}$ and $\overline{R}_\v^{j+N-2}$ must satisfy
$$
D\overline{K}_\th^{j-1}(u) \overline{f}^N(u,0)+ \overline{R}_\v^{j+N-2}(u)
=\overline{E}_\th^{j+N-2}(u). 
$$
We take $\Qa= 0$ and $\pa(u)=\overline{f}^N(u,0)$ in Theorem~\ref{thm:solutionlinearequation} and the corresponding constants $B_\Qa=0, a_\pa=a_f, A_\pa=A_f$ and $ b_\pa=b_f$ defined in~\eqref{defABg}, \eqref{defa} and~\eqref{defbCf}. 
We take 
    $$
    \overline{K}_\th^{j-1}= \mathcal{H}_{\overline{f}^N,0} \big [\overline{E}_\th^{j+P-2}\big ], \qquad \overline{R}_\v^{j+P-2}= 0.
    $$
\end{itemize}

In both cases, the solution of the oscillatory part of \eqref{secondcohomologicaltheta} can be given by 
$$
 \widetilde{K}_\th^{j+P-2}= \mathcal{D} [\widetilde{E}_\th^{j+P-2}], \qquad \widetilde{R}_\v^{j+P-2}= 0.
$$

We finally solve equation~\eqref{secondcohomologicalx}. We notice that after having solved~\eqref{secondcohomologicaly} and~\eqref{secondcohomologicaltheta}, the function $\mathcal{G}[K_{y,\th}^{(j)},R_\v^{(j)}]$ is already a known function. To simplify the notation, we introduce
$$
G^{j+N-1}:=\left \{ \mathcal{G}[K_{y,\th}^{(j)},R_\v^{(j)}] \right \}^{j+N-1},
$$
where the notation $\{\cdot \}^k$ has been introduced in
\eqref{Gparentesis}.
We first deal with the average part of~\eqref{secondcohomologicalx} which is
$$
D\overline{K}_x^{j}(u) \overline{f}^N(u,0) - \partial_x \overline{f}^N(u,0) \overline{K}_x^j(u)  +\overline{R}^{j+N-1}_u =  \overline{E}^{j+N-1}_x+ \overline{G}^{j+N-1}.
$$ 
We use again  Theorem~\ref{thm:solutionlinearequation}, now taking $\Qa(u)=\partial_x \overline{f}^{N}(u,0)$ and $\pa(u)= \overline{f}^N(u,0)$. 
Let $\BB_{\partial_x \overline{f}^N}=D_f$ be the corresponding constant (see~\eqref{defconstantspa} and~\eqref{defABg}) and $j^*_u =\left [-\frac{D_f}{a_{f}}\right ]$ if $D_f<0$ and $j^*_u =1$ if $D_f\ge0$ as defined in \eqref{defju}. 
Condition~\eqref{hipgraus} in Theorem~\ref{thm:solutionlinearequation} 
is satisfied when $j+ \frac{D_f}{a_f}>0$. Therefore, 
\begin{itemize}  
    \item When $j \leq j^*_u$
we take $\overline{K}_x^j$ free and  
$$
 \overline{R}_u^{j+N-1}(u)= \overline{E}^{j+N-1}_x(u)+ \overline{G}^{j+N-1}(u) -
 D\overline{K}_x^{j}(u) \overline{f}^N(u,0) + \partial_x \overline{f}^N(u,0) \overline{K}_x^j(u).
$$
\item When $j>j^*_u$ we apply Theorem~\ref{thm:solutionlinearequation} and we take
$$
\overline{K}^j_x = \mathcal{H}_{\overline{f}^N, \partial_x \overline{f}^N} \big [\overline{E}^{j+N-1}_x+ \overline{G}^{j+N-1}\big ],\qquad \overline{R}_u^{j+N-1}= 0. 
$$
\end{itemize}

The oscillatory part of~\eqref{secondcohomologicalx} is then solved by setting
$$
\widetilde{K}_x^{j+N-1} =  \mathcal{D}[\widetilde E_x^{j+N-1}+\widetilde{G}^{j+N-1} + \partial_x \widetilde{f}^N \overline{K}_x^{j}\big ], \qquad \widetilde{R}^{j+N-1}_u = 0. 
$$

This arguments show we can take ${R}^{(j)}=0$ if $j> j^*_{u}$.
We emphasize that when $n=1$, then $\overline{f}^N(x,0) = -a_fx^N$ and therefore $b_f=a_f$, $A_f= Na_f,\BB_{\partial_x \overline{f}^N} = -Na_f$. As a consequence,  $j^*_u= N$ and 
$(\overline{R}_u^{j+N-1}, \widetilde{R}_u^{j+N-1})=0$ 
if $j>N$.   


\subsubsection{End of the proof of Theorem~\ref{thm:approximationmaps}}\label{sec:endoftheprooapproximation}

As we pointed out in Remark~\ref{rem:jNMP}, with the procedure described in the previous section, we have obtained that there exist $K^{(j)}$ and $R^{(j)}$ satisfying
$$
E^{(j)}=F\circ K^{(j)} - K^{(j)} \circ R^{(j)}= \big (\OO(\|u\|^{j+N}), \OO(\|u\|^{j+M}), \OO(\|u\|^{j+P-1})
$$
instead of the stated result $E^{(j)}=\OO(\|u\|^{j+N})$. We need then to work further.  

When $M<N$, we look for $K^{(l)}$, $l=j+1,\cdots,j+N-M$ of the form 
$K^{(l)}=  K^{(l-1)} + \mathcal{K}^{(l)}$, 
with 
$$
\mathcal{K}_{x,\th}^{(l)}=0, \qquad \mathcal{K}_y^{(l)}(u,\v) =
\overline{K}_y^{l}(u) + \widetilde{K}_y^{l+M-1}(u,\v), 
$$
and we keep $\RR^{(l)}=0$.
Assume that, for $j+1\leq l \leq j+N-M$
$$
E^{(l-1)}=F\circ K^{(l-1)} - K^{(l-1)} \circ R^{(l-1)}=\big (\OO(\|u\|^{l+N-1}),\OO(\|u\|^{l+M-1}), \OO(\|u\|^{l+P-2})\big ).
$$
Since the map $F$ can be expressed as a sum of homogeneous functions up to degree $j+N\leq q-1$, we can apply the procedure described before setting $K^{(l)}_{x,\th}= 0$. The equation corresponding to~\eqref{secondcohomologicaly} is 
$$ 
E_y^{l+M-1} (u,\v)+ \partial_y g^M(u,0,\v) \overline{K}_y^{l}(u) + \widetilde{K}_y^{l+M-1}(u,\v) - 
\widetilde{K}_y^{l+M-1}(u,\v+\omega) =0,
$$
which can be solved as described in the previous section. In addition, since $K^{(l)}_{x,\th}=0$ and $ R^{(l)}= 0$, then $E^{(l)}_x=\OO(\|u\|^{j+N})$ and $E^{(l)}_\th = \OO(\|u\|^{j+P-1})$ (see equations~\eqref{firstcohomologicalx} and~\eqref{firstcohomologicaltheta}).
 
We repeat this procedure until $l=j+N-M$ and  we obtain that 
$$
E^{(j+N-M)}_{x,y} = \OO(\|u\|^{j+N}), \qquad E^{(j+N-M)}_{\th} = \OO(\|u\|^{j+P-1}). 
$$

Finally, we look for 
$K^{(l+N-M)}, R^{(l+N-M)}$ for $l=j+1, \cdots, j+N-P+1$ of the form 
$K^{(l+N-M)}=  K^{(l+N-M-1)} + \mathcal{K}^{(l+N-M)}$, 
with 
$$
\mathcal{K}_{x,y}^{(l+N-M)}=0, \qquad \mathcal{K}_\th^{(l+N-M)}(u,\v) =
\overline{K}_\th^{l-1}(u) + \widetilde{K}_\th^{l+P-2}(u,\v)
$$
and $R^{(l+N-M)}=R^{(l+N-M-1)} + \mathcal{R}^{(l+N-M-1)}$ with 
$$
\mathcal{R}_u^{(l+N-M-1)}=0, \qquad \mathcal{R}_\v=\overline{R}_\v^{l+P-2}+\widetilde{R}_\v^{l+P-2}.
$$
Assume that, for $j+1\leq l \leq j+N-P+1$
\begin{align*}
E^{(l+N-M-1)}&=F\circ K^{(l+N-M-1)} - K^{(l+N-M-1)} \circ R^{(l+N-M-1)} \\ & =\big (\OO(\|u\|^{j+N},\OO(\|u\|^{j+N}), \OO(\|u\|^{l+P-1})\big ).
\end{align*}
Similarly as before,  now the equation corresponding to~\eqref{secondcohomologicaltheta}  is
\begin{align*}
    E_\th^{l+P-2}(u,\v) +& \left \{
-D \overline{K}_\th^{l-1} (u) \overline{f}^N(u,0)
 \right \}^{l+P-2} - \overline{R}_\v^{l+P-2}(u)-\widetilde{R}_\v^{l+P-2}(u,\v) \notag\\ 
  &+\widetilde{K}_\th^{l+P-2}(u,\v) - \widetilde{K}_\th^{l+P-2}(u,\v+\omega)=0
\end{align*}
and it is solved as described previously. 
Note that we can always take $\widetilde{R}_\v^{l+P-2}=0$ and, if $P=N$ we can also take 
$\overline{R}_\v^{l+P-2}=0$. 
Looking at equations~\eqref{firstcohomologicalx} and~\eqref{firstcohomologicaly}, it can be easily deduced that $E^{(l+N-M)}_{x,y}=\OO(\|u\|^{j+N})$. In the last step of this new induction procedure we obtain that the corresponding remainder $E^{(j+2N-M-P)} = \OO(\|u\|^{j+N})$.


\subsection{Approximation of the invariant manifolds for differential equations}\label{sec:cohomologicalflow} 

Let $X$ be a vector field of the form~\eqref{systemflow} depending quasiperiodically on time with time frequency $\nu$. We briefly describe the procedure which is analogous to the one for maps explained in detail in Section~\ref{sec:inductionalgorithm}. Indeed, first we set 
$$
K^{(1)}(u,\vf,t)=(u+\widetilde{K}_x^N(u,\vf,t),0,\vf), \qquad Y^{(1)}(u,\vf,t)=(\overline{f}^N(u,0),  \vf+\omega)
$$
and we check that $E^{(1)}$ defined by~\eqref{def:Ejflow} satisfies
$$
E^{(1)}=\big (\OO(\|u\|^{1+N}), \OO(\|u\|^{1+M}), \OO(\|u\|^{P}) \big ).
$$
Then,  we define $K^{(j)}=K^{(j-1)} + \mathcal{K}^{(j)}$, $Y^{(j)} = Y^{(j-1)} + \mathcal{Y}^{(j)}$ with 
$$
\K^{(j)}_{x}(u,\vf,t)=\overline{K}_{x }^j(u)
+\widetilde{K}_{x }^{j+N-1}(u,\vf,t),
\qquad \K^{(j)}_{y}(u,\vf,t)=\overline{K}_{y}^j(u)
+\widetilde{K}_{y }^{j+M-1}(u,\vf,t),
$$
$$
\K^{(j)}_{\th}(u,\vf,t)=\overline{K}_{\th}^{j-1}(u)
+\widetilde{K}_{\th}^{j+P-2}(u,\vf,t)
$$
and $\mathcal{Y}^{(j)} $ as:
$$
\mathcal{Y}^{(j)}_u(u,\vf,t)=\overline{Y}^{j+N-1}_u(u) +
\widetilde{Y}^{j+N-1}_u(u,\vf,t),
\qquad
\mathcal{Y}^{(j)}_{\vf}(u,\v)=\overline{Y}^{j+P-2}_{\vf}(u)
+\widetilde{Y}^{j+P-2}_{\vf}(u,\vf,t).
$$

We prove by induction, reproducing the same arguments as the ones in Section~\ref{sec:induction}, that if $E^{(j-1)}$ 
defined by~\eqref{def:Ejflow} is such that 
\begin{align*}
E^{(j-1)}(u,\vf,t)  = & (E^{j+N-1}_x(u,\vf,t), E^{j+M-1}_y(u,\vf,t), E^{j+P-2}_\theta(u,\vf,t)) \\ 
& + \big (\OO(\|u\|^{j+N}), \OO(\|u\|^{j+M}), \OO(\|u\|^{j+P-1}) \big )
\end{align*}
with $E^l_{x,y,\theta}$ homogeneous functions of degree $l$, then $E^{(j)}$ satisfies 
$$
E^{(j)}=(\OO(\|u\|^{j+N}), \OO(\|u\|^{j+M}), \OO(\|u\|^{j+P-1})), 
$$
if $\K^{(j)},\mathcal{Y}^{(j)}$ are solutions of the cohomological equations
\begin{align}
E_y^{j+M-1}(u,\vf,t) =& \left \{
D \overline{K}_y^j (u) \overline{f}^N(u,0) \right \}^{j+M-1}
+\partial_y \gN(u,0,\vf,t) \overline{K}_y^j(u)
 \notag \\
& - \partial_{\vf} \widetilde{K}_y^{j+M-1}(u,\vf,t) \omega - \partial_t \widetilde{K}_y^{j+M-1}(u,\vf,t),  \label{secondcohomologicalyflow}\\
E_\theta^{j+P-2}(u,\vf,t) =& \left \{
D \overline{K}_\theta^{j-1} (u) \overline{f}^N(u,0)
 \right \}^{j+P-2} + \overline{Y}_{\vf}^{j+P-2}(u)+\widetilde{Y}_{\vf}^{j+P-2}(u,\vf,t) \notag\\ 
  &- \partial_{\vf}\widetilde{K}_\theta^{j+P-2}(u,\vf,t) \omega - \partial_t \widetilde{K}_\th^{j+P-2}(u,\vf,t), \label{secondcohomologicalthetaflow} \\ 
  E_x^{j+N-1}(u,\vf,t) =&
D \overline{K}_x^j (u) \overline{f}^N(u,0)- \partial_x \fN(u,0,\v) \overline{K}_x^j(u)
  \notag \\& + \overline{Y}_u^{j+N-1}(u) + \widetilde{Y}_u^{j+N-1}(u,\vf,t)
\label{secondcohomologicalxflow} \\
&-\partial_{\vf} \widetilde{K}_x^{j+N-1}(u,\vf,t) \omega -\partial_t \widetilde{K}_x^{j+N-1} (u,\vf,t) - \left \{\mathcal{G}[K^{(j)}_{y,\th},R_{\vf}^{(j)}]\right \}^{j+N-1}. \notag 
\end{align}

Equations~\eqref{secondcohomologicalyflow}, 
\eqref{secondcohomologicalthetaflow} and~\eqref{secondcohomologicalxflow} are the corresponding ones to equations~\eqref{secondcohomologicaly}, 
\eqref{secondcohomologicaltheta} and~\eqref{secondcohomologicalx} for the case of maps. As expected, the difference between them is that the difference term in the map setting 
$$
\widetilde{K}(u,\v+\omega) - \widetilde{K}(u,\v)
$$
now becomes the  term 
$$
\partial_{\vf} \widetilde{K}(u,\vf,t) \omega + \partial_t \widetilde{K}(u,\vf,t)  .
$$
Here, to solve the corresponding equations 
\begin{equation}\label{small_divisors_formal}
\partial_{\vf} \widetilde{K}(u,\vf,t) \omega + \partial_t \widetilde{K}(u,\vf,t) = \widetilde{h}(u,\vf,t)
\end{equation}
with $\widetilde{h}$ a known function with zero average, 
we use the small divisors theorem (Theorem~\ref{thm:smalldivisors}) for differential equations instead of the one for maps. Indeed, consider $\widehat{h}(u,\vf,\tau)$ be such that 
$\widetilde{h}(u,\vf,t)= \widehat{h}(u,\vf, \nu t)$ (as explained in Section~\ref{sec:notacio}) and 
the small divisor equation
$$
\partial_{\vf} \widehat{K}(u,\vf,\tau) \omega + \partial_\tau \widehat{K}(u,\vf,\tau) \nu = \widehat{h}(u,\vf,\tau).
$$
Let $\widehat{K}:=\mathcal{D}[\widehat{h}]$ be 
its unique solution with zero average (we recall that we use the same notation, $\mathcal{D}$, for both settings: flows and maps).  It is then clear that $\widetilde{K}(u,\vf,t)= \mathcal{D}[\widehat{h}](u,\vf,\nu t)$ is the solution of~\eqref{small_divisors_formal}. Then, with this interpretation, the algorithm described in  Section~\ref{sec:induction} applies in the same way.   


\section{Double parabolic orbits to infinity in the $n+2$-body problem}
\label{sec:problemadenmes2cossos}

\subsection{The $n+2$-body problem and Jacobi coordinates}
\label{sec:coordenadesdeJacobi}

We consider $n+2$ point masses, $m_i$, $i=0,\dots,n+1$, evolving in the plane under their mutual Newtonian gravitational attraction. We denote by $q_i \in \R^2$, $i=0,\dots,n+1$, the coordinates of the $i$-th mass in an inertial frame of reference. Their motion is described by the Hamiltonian
\begin{equation}
\label{def:Hamiltoniancartesian}
H(q,p) = T(p)-U(q),
\end{equation}
where $p = (p_0,\dots, p_{n+1}) \in \R^{2(n+2)}$ are the conjugate momenta and
\begin{equation*}
\begin{aligned}
T(p_0,\dots,p_{n+1}) & = \sum_{j=0}^{n+1} \frac{1}{2m_j} \|p_j\|^2, \\
U(q_0,\dots,q_{n+1}) & = \sum_{0\le i < j \le n+1} \frac{m_i m_j}{\|q_i-q_j\|}.
\end{aligned}
\end{equation*}
Well known first integrals of this system, besides the energy, are the total linear momentum, $\sum_{j=0}^{n+1}p_j$,
and the total angular momentum, $\sum_{j=0}^{n+1} \det(q_j,p_j)$.

We devote next sections to prove Theorem~\ref{thm:solucionsdoblementparaboliqueainfinitv1}. It will be an immediate consequence of Theorems~\ref{thm:posterioriresultflow},
\ref{thm:approximationflows}, once the Hamiltonian~\eqref{def:Hamiltoniancartesian} is written in the appropriate variables.

We want to show that there are solutions in which the first $n$ bodies evolve in a bounded motion while the last two arrive to infinity as time goes to infinity.
For this reason, we use the classical Jacobi coordinates, in which the position of the $j$-th body is measured with respect the center of mass of the  bodies $0$ to $j-1$, for $1\le j \le n+1$. More concretely, we consider the new set of coordinates $(\wt  q_0, \dots, \wt q_{n+1})$ defined by
\begin{equation*}
\begin{aligned}
\wt q_0 & = q_0, \\
\wt q_j & = q_j - \frac{1}{M_j} \sum_{0\le \ell \le j-1} m_\ell q_\ell, \qquad j=1,\dots, n+1,
\end{aligned}
\end{equation*}
where $M_j = \sum_{\ell=0}^{j-1} m_\ell$, $j\ge 1$.
 The inverse change is given by
\begin{equation}
\label{def:inverseJacobi}
\begin{aligned}
q_0 & = \wt q_0, \\
q_j & = \wt q_j +  \sum_{0\le \ell \le j-1} \frac{m_\ell}{M_{\ell+1}} \wt q_\ell, \qquad j=1,\dots, n+1.
\end{aligned}
\end{equation}
Denoting  by $A$ the matrix such that $\wt q = A q$, the change in the momenta given by $\wt p = A^{-\top} p$ makes the whole transformation symplectic. Let
\begin{equation*}
\wt H(\wt q,\wt p) = T(A^\top \wt p)-U(A^{-1} \wt q)
\end{equation*}
be the Hamiltonian $H$ in the new variables.
Notice that, in the $(\wt q,\wt p)$
variables, the total linear momentum is simply $\wt p_0$. In particular, this implies that $\wt H$ does not depend on $\wt q_0$. We can also assume that $\wt p_0 = 0$. 
Then,  $\wt H$ does not depend on $(\wt q_0,\wt p_0)$.
With this choice and defining
$\MM = \textrm{diag}\,(m_0^{-1}, \dots, m_{n+1}^{-1})$, a computation gives
\[
T(A^\top \wt p) = \frac{1}{2} \wt p^\top A \MM A^\top \wt p =
\sum_{j=1}^{n+1} \frac{1}{2\mu_j} \|\wt p_j\|^2 ,
\]
where $\mu_j^{-1} = M_j^{-1}+m_j^{-1}$.
Also, in view of~\eqref{def:inverseJacobi}, we have that
\[
\begin{aligned}
q_j -q_0 & = \wt q_j + \sum_{1 \le \ell \le j-1} \frac{m_\ell}{M_{\ell+1}} \wt q_\ell, &     1 \le j \le n+1,\\
q_j-q_i &  = \wt q_j -\wt q_i + \sum_{i \le \ell \le j-1} \frac{m_\ell}{M_{\ell+1}} \wt q_\ell
= \wt q_j -\frac{M_i}{M_{i+1}}\wt q_i + \sum_{i+1 \le \ell \le j-1} \frac{m_\ell}{M_{\ell+1}} \wt q_\ell, &    1 \le i < j \le n+1.
\end{aligned}
\]
Then,
\[
\begin{aligned}
U(A^{-1} \wt q)
=& 
\sum_{0\le i < j \le n-1} \frac{ m_j m_i}{\|q_j-q_i\|}+\sum_{0\le i \le n-1} \frac{m_n m_i}{\|q_n-q_i\|}
+\sum_{0\le i \le n-1} \frac{ m_{n+1} m_i}{\|q_{n+1}-q_i\|}+\frac{m_{n+1} m_n}{\|q_{n+1}-q_n\|} \\
= & \sum_{1\le  j \le n-1} \frac{ m_j m_0}{\left\|\wt q_j + \sum_{1 \le \ell \le j-1} \frac{m_\ell}{M_{\ell+1}} \wt q_\ell\right\|}+
\sum_{1\le i < j \le n-1} \frac{m_j m_i}{\left\|\wt q_j- \wt q_i + \sum_{i\le \ell \le j-1} \frac{m_\ell}{M_{\ell+1}} \wt q_{\ell}\right\|} \\
& + \frac{m_n m_0}{\left\|\wt q_n + \sum_{1 \le \ell \le n-1} \frac{m_\ell}{M_{\ell+1}} \wt q_\ell\right\|} +\sum_{1\le i \le n-1} \frac{m_n m_i}{\left\|\wt q_n -\wt q_i + \sum_{i \le \ell \le n-1} \frac{m_\ell}{M_{\ell+1}} \wt q_\ell\right\|} \\
& + \frac{m_{n+1} m_0}{\left\|\wt q_{n+1} + \sum_{1 \le \ell \le n} \frac{m_\ell}{M_{\ell+1}} \wt q_\ell\right\|} +\sum_{1\le i \le n-1} \frac{m_{n+1} m_i}{\left\|\wt q_{n+1} -\wt q_i + \sum_{i \le \ell \le n} \frac{m_\ell}{M_{\ell+1}} \wt q_\ell\right\|}\\
&+
\frac{m_{n+1} m_n}{\|\wt q_{n+1} -\wt q_n +  \frac{m_n}{M_{n+1}} \wt q_n\|} ,
\end{aligned}
\]
where in the first line of the formula $q= A^{-1} \wt q$.
Now we introduce symplectic polar coordinates in each subspace generated by 
$(\wt q_j,\wt p_j)$:
\[
\left\{
\begin{aligned}
\wt q_j & = r_j e^{i\theta_j}, \\
\wt p_j & = y_j e^{i\theta_j} + i \frac{G_j}{r_j} e^{i\theta_j},
\end{aligned}
\right. \qquad j = 1, \dots, n+1,
\]
and denote by $H(\wh r,\wh y, \wh \theta, \wh G, r_n, y_n, r_{n+1},y_{n+1}, \theta_n, G_n, \theta_{n+1}, G_{n+1})$ the Hamiltonian in these new variables, where
$\wh r = (r_1,\dots,r_{n-1})$ and, analogously, the same notation applies to $\wh y$, $\wh \theta$, $\wh G$.

We will be interested in the region of the phase space where $r_{n+1} , r_n \gg r_i$, $i=1,\dots, n-1$. However, since the final motions we are looking for are parabolic, it will happen that $r_n/r_{n+1}$ will be of order $1$. Hence, we will be able to expand several magnitudes in $r_i/r_n$, $r_i/r_{n+1}$, $i=1,\dots,n-1$, but not in $r_n/r_{n+1}$.

In the new variables the potential  is
\begin{equation*}
U(r,\theta)
= \wh U (\wh r, \wh \theta)
 + U_n (\wh r, r_n, \theta_1-\theta_n, \dots, \theta_{n-1}-\theta_n) + U_{n+1} (\wh r, r_n, r_{n+1},\theta_1-\theta_{n+1}, \dots, \theta_{n}-\theta_{n+1}),
\end{equation*}
where
\begin{equation}
\label{def:whU}
\begin{aligned}
\wh U (\wh r, \wh \theta)  =& \sum_{1\le  j \le n-1} \frac{ m_j m_0}{\left|r_j + \sum_{1 \le \ell \le j-1} \frac{m_\ell}{M_{\ell+1}} r_\ell e^{i(\theta_\ell-\theta_j)}\right|} \\& +
\sum_{1\le i < j \le n-1} \frac{m_j m_i}{\left|r_j- r_ie^{i(\theta_i-\theta_j)} + \sum_{i\le \ell \le j-1} \frac{m_\ell}{M_{\ell+1}} r_{\ell}e^{i(\theta_\ell-\theta_j)}\right|}, \\
U_n (\wh r, r_n, \phi_1, \dots, \phi_{n-1})  = & \frac{m_n m_0}{r_n \left|1
  + \sum_{1 \le \ell \le n-1} \frac{m_\ell}{M_{\ell+1}} \frac{r_\ell}{r_n} e^{i \phi_\ell} \right|} \\
  & +\sum_{1\le j \le n-1} \frac{m_n m_j}{r_n\left|1 -\frac{r_j}{r_n}e^{i\phi_j} + \sum_{j \le \ell \le n-1} \frac{m_\ell}{M_{\ell+1}} \frac{r_\ell}{r_n} e^{i\phi_\ell}\right|}, \\
U_{n+1} (\wh r, r_n, r_{n+1},\phi_1, \dots, \phi_n)  =
& \frac{m_{n+1} m_0}{r_{n+1} \left|1 + \sum_{1 \le \ell \le n} \frac{m_\ell}{M_{\ell+1}} \frac{r_\ell}{r_{n+1}} e^{i\phi_\ell} \right|} \\
 & +\sum_{1\le j \le n} \frac{m_{n+1} m_j}{r_{n+1}\left|1 -\frac{r_j}{r_{n+1}}e^{i\phi_j}
 + \sum_{j \le \ell \le n} \frac{m_\ell}{M_{\ell+1}} \frac{r_\ell}{r_{n+1}} e^{i\phi_\ell}\right|}.
\end{aligned}
\end{equation}

\begin{proposition}
\label{prop:descomposicio_i_expansio_del_potencial}
Let $m_0,\dots, m_{n-1} \in \R^+$ be fixed. The functions $U_n$ and $U_{n+1}$ can be written as
\begin{equation}
\label{def:UntildeUn}
U_n (\wh r, r_n, \phi_1, \dots, \phi_{n-1})
= \frac{m_n M_n}{r_n} +  \frac{m_n}{r_n}\wt U_n (\wh r, r_n, \phi_1, \dots, \phi_{n-1})
\end{equation}
and
\begin{equation}
\label{def:UnmesutildeUmesu}
U_{n+1} (\wh r, r_n, r_{n+1},\phi_1, \dots, \phi_n)  = \frac{m_{n+1} M_{n+1}}{r_{n+1}} + \frac{m_{n+1}}{r_{n+1}}
\wt U_{n+1} (\wh r, r_n, r_{n+1},\phi_1, \dots, \phi_n) ,
\end{equation}
with
\begin{equation}
\label{def:AnAnmesu}
\begin{aligned}
\wt U_n (\wh r, r_n, \phi_1, \dots, \phi_{n-1}) & = \A_n \left( \frac{r_1}{r_n} e^{i\phi_1}, \frac{r_1}{r_n} e^{-i\phi_1}\dots,
\frac{r_{n-1}}{r_n} e^{i\phi_{n-1}},\frac{r_{n-1}}{r_n} e^{-i\phi_{n-1}}\right), \\
\wt U_{n+1}(\wh r, r_n, r_{n+1},\phi_1, \dots, \phi_n) & = \A_{n+1} \left( \frac{r_1}{r_{n+1}} e^{i\phi_1},\frac{r_1}{r_{n+1}} e^{-i\phi_1}, \dots, \frac{r_{n}}{r_{n+1}} e^{i\phi_n}, \frac{r_{n}}{r_{n+1}} e^{-i\phi_n},m_n \right),
\end{aligned}
\end{equation}
where
\begin{enumerate}
\item[(1)]
$\A_n(z_1,\overline{z_1},\dots,z_{n-1},\overline{z_{n-1}})$ is analytic with respect to its arguments in a neighborhood of $0$ and satisfies
\begin{equation}
\label{eq:wtUnordre}
\A_n(z_1,\overline{z_1},\dots,z_{n-1},\overline{z_{n-1}})  = \OO_2(z_1,\overline{z_1},\dots,z_{n-1},\overline{z_{n-1}}),
\end{equation}
\item[(2)]
for any ${\Kc}>1$, there exists $m>0$ such that $\A_{n+1}(z_1,\overline{z_1},\dots,z_n,\overline{z_{n}},m_n)$ is analytic in
\begin{align*}
D_{{\Kc},m} = & \{|z_j|,|\overline z_j| < {\Kc}^{-1}, \; j=1,\dots, n-1, \ |z_n|,
\ |\overline z_n| < {\Kc}, 
\\
& \quad |z_n -1|, \ |\overline z_n -1| > {\Kc}^{-1},\;
|m_n| <m\}
\end{align*}
and, defining
\begin{align}\label{eq:expansiodeUnmes1}
\wh \A_{n+1}(z,\overline{z},m_n)  = & \A_{n+1} (0,\dots,0,z,\overline{z},m_n)\\
  = &
\frac{M_n}{\left|1+\frac{m_n}{M_{n+1}} z \right|}+
\frac{m_n}{\left|1-\frac{M_n}{M_{n+1}} z \right|} - M_{n+1} \notag \\
  = &
\frac{M_n}{\left(1+\frac{m_n}{M_{n+1}} (z+\overline{z}) +  \frac{m_n^2}{M_{n+1}^2} z\overline{z} \right)^{1/2}} 
\frac{m_n}{\left(1-\frac{M_n}{M_{n+1}} (z+\overline{z})
	+ \frac{M_n^2}{M_{n+1}^2}z\overline{z} \right)^{1/2}} - M_{n+1} \notag,
\end{align}
one has
\begin{equation}
\label{eq:Anmesu_cancellacio}
\A_{n+1} (z_1,\overline{z_1},\dots,z_n,\overline{z_{n}},m_n) - \wh \A_{n+1}(z_n,\overline{z_n},m_n) = \OO_2(z_1,\overline{z_1},\dots,z_{n-1},\overline{z_{n-1}}),
\end{equation}
uniformly in $D_{K, m}\rho$.
Finally, 
\[
T(r,y,G) = \sum_{j=1}^{n+1} \frac{1}{2\mu_j} \left (y_j^2 + \frac{G_j^2}{r_j^2} \right ) .
\]
\end{enumerate}
\end{proposition}

\begin{proof}
In view of~\eqref{def:whU}, we clearly have that
\[
\A_n(z_1,\overline{z_1},\dots,z_{n-1},\overline{z_{n-1}}) = \frac{m_0}{\left|1
  + \sum_{1 \le \ell \le n-1} \frac{m_\ell}{M_{\ell+1}} z_\ell\right|}
  +\sum_{1\le j \le n-1} \frac{m_j}{\left|1 -z_j + \sum_{j \le \ell \le n-1} \frac{m_\ell}{M_{\ell+1}} z_\ell\right|}-M_n
\]
and
\[
\A_{n+1} (z_1,\overline{z_1},\dots,z_n,\overline{z_{n}},m_n)
=
\frac{m_0}{\left|1 + \sum_{1 \le \ell \le n} \frac{m_\ell}{M_{\ell+1}} z_\ell \right|}
 +\sum_{1\le j \le n} \frac{ m_j}{\left|1 -z_j
 + \sum_{j \le \ell \le n} \frac{m_\ell}{M_{\ell+1}} z_\ell\right|}-M_{n+1}.
\]
The claim is then a straightforward computation. Formulas~\eqref{eq:wtUnordre} and~\eqref{eq:Anmesu_cancellacio} are obtained by expanding
in powers of $z_1,\overline{z_1},\dots,z_{n-1},\overline{z_{n-1}}$. The first order terms cancel out identically.
\end{proof}

Now we reduce the number of equations by the total angular momentum. To do so, we consider the symplectic change of variables
\begin{equation}
\label{def:reductionangularmomentum}
\begin{aligned}
\wt r_i & = r_i, &  \wt y_i & = y_i,  & i&=1,\dots,n+1 \\
\wt G_i & = G_i, &  \wt \theta_i & = \theta_i-\theta_{n+1},  \quad &i&=1,\dots,n \\
\wt G_{n+1} & = G_1+\cdots+G_{n+1}, & \qquad \wt \theta_{n+1} & = \theta_{n+1}.
\end{aligned}
\end{equation}
Since the total angular momentum $\Theta = G_1+\cdots+G_{n+1} = \wt G_{n+1}$ is a conserved quantity,
the Hamiltonian in the new variables does not depend on $\wt \theta_{n+1}$.

We remark that, since the potential $\wh U$ in~\eqref{def:whU} only depends on the angles through $\theta_i-\theta_j$, with $1\le i,j \le n-1$,
in the new variables~\eqref{def:reductionangularmomentum} it has the same expression. We will use it with the same name.
The same happens to $U_n$ but not to $U_{n+1}$. Dropping the tildes from the variables, the potential $U$ in the new variables --- which we denote again with the same letter although now does not depend on $\theta_{n+1}$ --- is
\begin{equation}
\label{def:potencialmomentangularreduit}
U ( \wh r, r_n, \wh \theta, \theta_n) = \wh U( \wh r, \wh \theta) + U_n (\wh r, r_n, \theta_1-\theta_n, \dots, \theta_{n-1}-\theta_n)
+ U_{n+1} (\wh r, r_n, r_{n+1},\wh \theta, \theta_n).
\end{equation}
The Hamiltonian in the new variables is
\begin{equation}
\label{def:Hamiltonianafterreductionoftheangularmomentum}
H(\wh r,\wh y, \wh \theta, \wh G, r_n, y_n, r_{n+1}, y_{n+1}, \theta_n, G_n) = \wh H (\wh r,\wh y, \wh \theta, \wh G)
+ \H (\wh r,\wh y, \wh \theta, \wh G, r_n, y_n, r_{n+1}, y_{n+1}, \theta_n, G_n),
\end{equation}
where
\begin{equation*}
\wh H (\wh r,\wh y, \wh \theta, \wh G) = \wh T(\wh r,\wh y,  \wh G) - \wh U (\wh r, \wh \theta),
\end{equation*}
with
\begin{equation}
\label{def:whT}
\wh T(\wh r,\wh y,  \wh G) =  \sum_{j=1}^{n-1} \frac{1}{2\mu_j} \left( y_j^2+ \frac{G_j^2}{r_j^2} \right),
\end{equation}
the potential $\wh U$ was introduced in~\eqref{def:whU} and
\begin{equation}
\label{hamiltoniaenvariablestilde}
\H (\wh r,\wh y, \wh \theta, \wh G, r_n, y_n, r_{n+1}, y_{n+1}, \theta_n, G_n) = \TTT(r_n, y_n, r_{n+1}, y_{n+1}, \wh G, G_n)  - \left(U(r,\theta) - \wh U (\wh r, \wh \theta)\right),
\end{equation}
with
\begin{equation*}
\TTT(r_n, y_n, r_{n+1}, y_{n+1}, \wh G, G_n) = \frac{1}{2\mu_n} \left( y_n^2+ \frac{G_n^2}{r_n^2} \right)
+ \frac{1}{2\mu_{n+1}} \left( y_{n+1}^2+ \frac{(\Theta-G_1\cdots -G_n)^2}{r_{n+1}^2} \right).
\end{equation*}

\subsection{A torus in the $n$-body problem}
\label{sec:tordiofantic}
The Hamiltonian $\wh H = \wh T - \wh U$, with $\wh T$ and $\wh U$ defined in~\eqref{def:whT} and~\eqref{def:whU}, respectively, is the Hamiltonian of a planar $n$-body problem in Jacobi coordinates. As such, it possesses $2(n-1)$-dimensional KAM invariant tori. Let $\omega \in \R^{2(n-1)}$ be a Diophantine frequency for which a KAM tori of $\wh H$ exists. There exists a symplectic with respect to the standard $2$-form
$d\wh r \wedge d \wh y + d \wh \theta \wedge d \wh G$, analytic change of variables $(\wh r, \wh y, \wh \theta, \wh G) = \wh \Phi(\varphi,\rho)$, $(\varphi,\rho) \in \T^{2(n-1)}\times B$, where $B\subset \R^{2(n-1)} $ is some ball,  such that
\begin{equation*}
\wh H_\omega (\varphi,\rho) = \wh H \circ \wh \Phi(\varphi, \rho) = \langle \omega, \rho \rangle + \OO_2(\rho).
\end{equation*}

Let 
\[
\Phi (\varphi,\rho,r_n, y_n, r_{n+1}, y_{n+1}, \theta_n, G_n) = (\wh \Phi(\varphi,\rho), r_n, y_n, r_{n+1}, y_{n+1}, \theta_n, G_n).
\] It is canonical in the sense that transforms  the standard $2$-form into
\begin{equation}
\label{def:formvarphirho}
d\varphi \wedge d \rho   + d r_n \wedge d y_n +
dr_{n+1} \wedge d y_{n+1}+ d\theta_n \wedge d G_n.
\end{equation}

We define
\begin{equation}\label{def:HT}
\begin{aligned}
H_\omega(\varphi,\rho,r_n, y_n, r_{n+1}, y_{n+1}, \theta_n, G_n)  & = H \circ \Phi(\varphi,\rho,r_n, y_n, r_{n+1}, y_{n+1}, \theta_n, G_n) \\
& = \wh H_\omega (\varphi,\rho) + \wt \H \circ \Phi(\varphi,\rho,r_n, y_n, r_{n+1}, y_{n+1}, \theta_n, G_n),
\end{aligned}
\end{equation}
the Hamiltonian in the new variables. 

We define the function
\begin{equation}
\label{def:wtTheta}
\wt \Theta (\varphi,\rho) = \Theta - (G_1+\cdots +G_{n-1}) \circ \wh \Phi(\varphi, \rho).
\end{equation}
Since, for $\rho=0$, $G_1+\cdots+G_{n-1}$ is a conserved quantity of $\wh H_\omega$, we have that
\begin{equation}
\label{def:wtTheta00}
\wt \Theta_0^0 =  \Theta - (G_1 + \cdots +G_{n-1})\circ \wh \Phi(\varphi,0)
\end{equation}
does not depend on $\varphi$ and it is the average with respect to $\varphi$ of $\wt \Theta(\varphi,0)$.

Theorem~\ref{thm:solucionsdoblementparaboliqueainfinitv1} is a consequence of the following result.

\begin{theorem}
\label{thm:movimentsdoblementparabolicsv2}
If $m_n,m_{n+1} >0 $ are small enough, then Hamiltonian~\eqref{def:HT} satisfies the following.
\begin{itemize}
\item \emph{Collinear case}.
There exist $A = 1+\OO(m_n,m_{n+1})$, depending on $m_n,m_{n+1}$, and $G_n^0$, depending on $m_n,m_{n+1}$ and $\wt \Theta_0^0$, and two $2+2(n-1)$-dimensional analytic invariant manifolds, $W_{\mathrm{Col}}^\mp$, invariant by the flow generated by~\eqref{def:HT} such that, for any solution
\[
(\varphi,\rho,r_n, y_n, r_{n+1}, y_{n+1}, \theta_n, G_n)(t) \in W_{\mathrm{Col}}^\mp,
\]
there exists $\varphi_\pm^0 \in \T^{2(n-1)}$ such that
\[
\begin{aligned}
\lim_{t\to \pm \infty} &  r_n(t) = \lim_{t\to \pm \infty} r_{n+1}(t) = \infty, \quad  
& \lim_{t\to \pm \infty} &  \theta_n(t) = \pi, \\
\lim_{t\to \pm \infty} &  y_n(t) = \lim_{t\to \pm \infty} y_{n+1}(t) = 0, &
\lim_{t\to \pm \infty} &  G_n(t) = G_n^0, \\
\lim_{t\to \pm \infty} &  \rho(t) =  0, &
\lim_{t\to \pm \infty} & [\varphi(t) - \omega t] = \varphi_\pm^0
\end{aligned}
\]
and
\[
\lim_{t\to \pm \infty} \frac{r_{n+1}(t)}{r_n(t)} = A.
\]
\item \emph{Equilateral case}.
There exist  $\wt \theta_0 = \pi/3+ \OO(m_n,m_{n+1})$ and $A = 1+\OO(m_n,m_{n+1})$, depending on $m_n,m_{n+1}$, and $G_n^0$, depending on $m_n,m_{n+1}$ and $\wt \Theta_0^0$, and two $3+2(n-1)$-dimensional analytic invariant manifold, $W_{\mathrm{Eq}}^\mp$, invariant by the flow generated by~\eqref{def:HT} such that, for any solution $(\varphi,\rho,r_n, y_n, r_{n+1}, y_{n+1}, \theta_n, G_n)(t) \in W_{\mathrm{Eq}}^\mp$, there exists $\varphi_\pm^0 \in \T^{2(n-1)}$ such that
\[
\begin{aligned}
\lim_{t\to \pm\infty} &  r_n(t) = \lim_{t\to \pm\infty} r_{n+1}(t) = \infty, & \quad \lim_{t\to \pm\infty} &  \theta_n(t) = \wt \theta_0, \\
\lim_{t\to \pm\infty} &  y_n(t) = \lim_{t\to \pm\infty} y_{n+1}(t) = 0, &
\lim_{t\to \pm\infty} &  G_n(t) = G_n^0, \\
\lim_{t\to \pm\infty} &  \rho(t) =  0, &
\lim_{t\to \pm\infty} & [\varphi(t) - \omega t] = \varphi_\pm^0
\end{aligned}
\]
and
\[
\lim_{t\to \pm\infty} \frac{r_{n+1}(t)}{r_n(t)} = A.
\]
\end{itemize}
\end{theorem}

We devote the rest of the section to the proof of the theorem. The collinear case is a immediate consequence of Proposition~\ref{prop:cascolinial}
and, the equilateral one, of Proposition~\ref{prop:casequilater}, below.

\subsection{Local behaviour at infinity: double McGehee coordinates}

In order to study the behaviour of the system when $r_{n+1} , r_n \gg r_i$, $i=1,\dots, n-1$, we introduce the \emph{double} McGehee coordinates
$x_n\, x_{n+1}, \wt y_n, \wt y_{n+1}$ through
\begin{equation}
\label{def:McGeheedoble}
r_n = \frac{2\alpha_n}{x_n^2}, \qquad y_n = \beta_n \wt y_n, \qquad r_{n+1} = \frac{2\alpha_{n+1}}{x_{n+1}^2}, \qquad y_{n+1} = \beta_{n+1} \wt y_{n+1},
\end{equation}
where $\alpha_n$, $\beta_n$, $\alpha_{n+1}$ and $\beta_{n+1}$ are constants, depending on $m_n$,
$m_{n+1}$, such that
\begin{equation}
\label{def:alphanbetanetcimplicites}
\begin{aligned}
\frac{\beta_n}{4 \mu_n \alpha_n} & = \frac{m_n M_n}{4 \alpha_n^2 \beta_n} = 1, \\
\frac{\beta_{n+1}}{4 \mu_{n+1} \alpha_{n+1}} & = \frac{m_{n+1} M_{n+1}}{4 \alpha_{n+1}^2 \beta_{n+1}} = 1,
\end{aligned}
\end{equation}
that is,
\begin{equation}
\label{def:alphanbetanetcexplicites} 
\begin{aligned}
\alpha_n & = \frac{1}{2^{4/3}}M_{n+1}^{1/3}, & \beta_n & = 2^{2/3} \frac{M_n m_n}{M_{n+1}^{2/3}}, \\
\alpha_{n+1} & = \frac{1}{2^{4/3}}M_{n+2}^{1/3}, & \beta_{n+1} & = 2^{2/3} \frac{M_{n+1} m_{n+1}}{M_{n+2}^{2/3}}.
\end{aligned}
\end{equation}
We are interested in the case where $m_0+\dots+m_{n-1}$ is of order $1$ while $m_n$ and $m_{n+1}$ are small. In particular,
the constants $\alpha_n$ and $\alpha_{n+1}$ are of order $1$ while $\beta_n$ and $\beta_{n+1}$ are small. Furthermore,
we have that
\begin{equation}
\label{bound:quocientalphanalphanmes1}
\frac{\alpha_n}{\alpha_{n+1}} = 1+ \OO\left(\frac{m_{n+1}}{M_{n+1}}\right).
\end{equation}

The change~\eqref{def:McGeheedoble} is not symplectic. It transforms the form~\eqref{def:formvarphirho} into
\begin{equation}
\label{def:forminMcGehee}
d\varphi \wedge d \rho - \frac{4 \alpha_n \beta_n }{x_n^3} d x_n \wedge d \wt y_n
- \frac{4 \alpha_{n+1}\beta_{n+1}}{x_{n+1}^3} d x_{n+1} \wedge d \wt y_{n+1}+ d\theta_n \wedge d G_n.
\end{equation}
We denote $\wt \H = \wt \TTT - \wt U$ the Hamiltonian $\H$ in \eqref{hamiltoniaenvariablestilde} and $\wt H = \wh H + \wt \H$ in \eqref{def:Hamiltonianafterreductionoftheangularmomentum} both expressed in these new variables. We drop the tildes on the $y$ variables.

Taking into account~\eqref{def:potencialmomentangularreduit}, \eqref{def:whU}, \eqref{def:UntildeUn} and~\eqref{def:UnmesutildeUmesu},
the potential $U-\wh U$ (see \eqref{def:potencialmomentangularreduit}) is transformed into
\begin{align}
\label{def:wtU}
\wt U_\omega (\varphi, \rho, x_n, x_{n+1}, \theta_n)
 =& \frac{m_n M_n}{2\alpha_n }x_n^2 + \frac{m_{n+1} M_{n+1}}{2\alpha_{n+1} }x_{n+1}^2
\\
&+ m_n \frac{x_n^2}{2\alpha_n} \wt U_{n,\omega} (\varphi,\rho,x_n,\theta_n)
  +  m_{n+1} \frac{x_{n+1}^2}{2\alpha_{n+1}} \wt U_{n+1,\omega} (\varphi,\rho,x_n,x_{n+1},\theta_n), \notag
\end{align}
where
\begin{align}
\label{def:wtUnomega}
\wt U_{n,\omega} (\varphi,\rho,x_n,\theta_n) &  =
\wt U_n \left(\wh r(\varphi,\rho), \frac{2\alpha_n}{x_n^2}, \theta_1(\varphi,\rho)-\theta_n, \dots,
\theta_{n-1}(\varphi,\rho)-\theta_n\right), \\
\wt U_{n+1,\omega} (\varphi,\rho,x_n,x_{n+1},\theta_n) & = \wt U_{n+1} \left(\wh r(\varphi,\rho), \frac{2\alpha_n}{x_n^2}, \frac{2\alpha_{n+1}}{x_{n+1}^2}, \theta_1(\varphi,\rho), \dots,
\theta_{n-1}(\varphi,\rho),\theta_n\right), \notag
\end{align}
with $(r_1(\varphi,\rho),\dots,r_{n-1}(\varphi,\rho),\theta_1(\varphi,\rho),\dots,\theta_{n-1}(\varphi,\rho)) = (\wh r, \wh \theta) \circ \wh \Phi (\varphi,\rho)$,
while the kinetic energy part of $\H$ becomes
\begin{equation}
\label{def:wtTTT}
\wt \TTT_\omega(\varphi,\rho,x_n, y_n, x_{n+1}, y_{n+1},  G_n) = \frac{\beta_n^2}{2\mu_n}  y_n^2+ \frac{\beta_{n+1}^2}{2\mu_{n+1}}  y_{n+1}^2
+ \frac{G_n^2 x_n^4}{4 \alpha_n^2\mu_n} + \frac{(\wt \Theta(\varphi,\rho)-G_n)^2 x_{n+1}^4}{4\alpha_{n+1}^2\mu_{n+1}} ,
\end{equation}
where $\wt \Theta$ was introduced in~\eqref{def:wtTheta}.
\begin{proposition}
\label{prop:expansionsdetildeUnomegaitildeUnmesuomega}
Let $m_0,\dots,m_{n-1} \in \R^+$ be fixed.
\begin{enumerate}
\item[(1)]
$\wt U_{n,\omega}$ is analytic with respect to its arguments
in a neighborhood of $(\rho,x_n) = 0$ and admits an expansion of the form
\begin{equation*}
\wt U_{n,\omega} (\varphi,\rho,x_n,\theta_n) = \sum_{j\ge 2, \ell \ge 0} c_{j,\ell} (\varphi,\theta_n) x_n^{2j}\rho^\ell.
\end{equation*}
\item[(2)]
The function $\wt U_{n+1,\omega}$ can be written as
\begin{multline*}
\wt U_{n+1,\omega} (\varphi,\rho,x_n,x_{n+1},\theta_n) = u_0\left( \frac{\alpha_n}{\alpha_{n+1}}\frac{x_{n+1}^2}{x_n^2} e^{i \theta_{n}},\frac{\alpha_n}{\alpha_{n+1}}\frac{x_{n+1}^2}{x_n^2} e^{-i \theta_{n}}\right) \\ + \sum_{k\ge 2} u_k \left(\varphi,\rho, \frac{\alpha_n}{\alpha_{n+1}}\frac{x_{n+1}^2}{x_n^2} e^{i \theta_{n}},\frac{\alpha_n}{\alpha_{n+1}}\frac{x_{n+1}^2}{x_n^2} e^{-i \theta_{n}}\right) x_{n+1}^{2k},
\end{multline*}
where, given ${\Kc} >1$,   $u_j(\varphi,\rho,z,\overline{z})$ is analytic in a neighborhood of $\rho = 0$, $|z|$, $|\overline{z}| < {\Kc}$, $|1-z|$, $|1-\overline{z}| \ge \Kc^{-1}$ and
\begin{equation}
\label{def:u0}
u_0(z,\overline{z}) = \wh \A_{n+1}(z,\overline{z},m_n)
\end{equation}
where $\wh \A_{n+1}$ was introduced in~\eqref{eq:expansiodeUnmes1}.
For $k\ge 2$, we introduce the expansion
\begin{equation*}
u_k(\varphi,\rho,z,\overline{z}) = \sum_{j\ge 0} u_{k,j} (\varphi,z,\overline{z}) \rho^j.
\end{equation*}
\end{enumerate}
\end{proposition}

\begin{proof}
The claim for $\wt U_{n,\omega}$ follows immediately from~\eqref{def:wtUnomega}, \eqref{def:UntildeUn} and item (1) of Proposition~\ref{prop:descomposicio_i_expansio_del_potencial}.

As for $\wt U_{n+1,\omega} $, in view of~\eqref{def:AnAnmesu}, we have that
\begin{multline*}
\wt U_{n+1,\omega} (\varphi,\rho,x_n,x_{n+1},\theta_n) =
\A_{n+1}  \left( r_1(\varphi,\rho) \frac{x_{n+1}^2}{2\alpha_{n+1}} e^{i\theta_1(\varphi,\rho)},
 r_1(\varphi,\rho) \frac{x_{n+1}^2}{2\alpha_{n+1}} e^{-i\theta_1(\varphi,\rho)}, \right.\\
\left.  \dots, r_{n-1}(\varphi,\rho) \frac{x_{n+1}^2}{2\alpha_{n+1}} e^{i\theta_{n-1}(\varphi,\rho)},
r_{n-1}(\varphi,\rho) \frac{x_{n+1}^2}{2\alpha_{n+1}} e^{-i\theta_{n-1}(\varphi,\rho)},\frac{\alpha_n}{\alpha_{n+1}}\frac{x_{n+1}^2}{x_n^2} e^{i \theta_{n}},\frac{\alpha_n}{\alpha_{n+1}}\frac{x_{n+1}^2}{x_n^2} e^{-i \theta_{n}} \right).
\end{multline*}
The claim follows immediately from item (2) of Proposition~\ref{prop:descomposicio_i_expansio_del_potencial}.
\end{proof}

Let 
\[ 
V_0(\alpha,\theta) = u_0(\alpha e^{i\theta},\alpha e^{-i\theta}).
\]
The following lemma summarizes the properties of the functions $u_0$, $V_0$  that will be need.
\begin{lemma}
\label{lem:fitestermeacoblament}
There exist $\delta, K, m >0$ such that, for all $0\le m_n \le m$, $(\alpha,\theta) \in [1-\delta,1+\delta] \times [\pi-\delta,\pi+\delta] \cup [1-\delta,1+\delta] \times [\pi/3-\delta,\pi/3+\delta]$ where $V_0$ is analytic and
\[
 \left|\frac{\partial^j V_0}{\partial \alpha^j}(\alpha,\theta)\right| \le K m_n, \qquad j=0,1,2.
\]
Moreover,
\[
\begin{aligned}
\frac{\partial V_0}{\partial \theta}(\alpha,\theta) & = -\frac{7}{8} m_n ( \theta -\pi)(1+ \OO(\alpha-1,m_n,\theta -\pi)), \\
\frac{\partial V_0}{\partial \theta}(\alpha,\theta) & = \frac{9}{4} m_n \left( \theta-\frac{\pi}{3}+ \OO(\alpha-1,m_n)
+ \OO_2(\alpha-1,m_n,\theta-\pi/3)\right).
\end{aligned}
\]
In particular, for each $(\alpha,m_n) \in [1-\delta,1+\delta]\times [0,m]$, the equation
\[
\frac{\partial V_0}{\partial \theta}(\alpha,\theta) = 0
\]
has the solutions $\theta = \pi$ and the unique analytic solution in $[\pi/3-\delta,\pi/3+\delta]$, $\widehat{\theta}^0(\alpha,m_n)$, satisfying
\[
\widehat{\theta}^0(\alpha,m_n) = \frac{\pi}{3} + \OO(\alpha-1,m_n).
\]
\end{lemma}

\begin{proof}
In view of~\eqref{eq:expansiodeUnmes1} and recalling that $M_{n+1} = M_n+m_n$,
\[
V_0(\alpha,\theta) =
\frac{M_n}{\left(1+2 \frac{m_n}{M_{n+1}} \alpha \cos \theta +  \frac{m_n^2}{M_{n+1}^2} \alpha^2\right)^{1/2}}+
\frac{m_n}{\left(1-2\frac{M_n}{M_{n+1}} \alpha \cos \theta + \frac{M_n^2}{M_{n+1}^2} \alpha^2\right)^{1/2}} - M_{n+1}
\]
is clearly analytic in neighborhoods of $(\alpha,\theta,m_n) = (1,\pi,0)$ and $(\alpha,\theta,m_n) = (1,\pi/3,0)$, since, then,
$M_n/M_{n+1} = 1$, and ${V_0}_{\mid m_n =0} = 0$. This implies the first claim.

The second claim is a straightforward computation. The third one is an immediate consequence of the second.
\end{proof}

\subsection{The constants $A$, $B$ and $G_n^0$}

Next lemma provides constants that will be needed later.

\begin{lemma}
\label{lem:constantsAiB}
Let $M_n = \sum_{j=0}^{n-1} m_j$ be fixed. Consider the equations for the constants $A$ and $B$
\[
\left\{
\begin{aligned}
A^3 B & = A  , \\
\Bigg( 1+\frac{m_{n+1}}{4 \alpha_{n+1}^2 \beta_{n+1}} V_0\left(\frac{\alpha_n}{\alpha_{n+1}} A^2,\theta\right) + & \frac{m_{n+1}\alpha_n}{4\alpha_{n+1}^3 \beta_{n+1}}
\frac{\partial V_0}{\partial \alpha}\left(\frac{\alpha_n}{\alpha_{n+1}}  A^2,\theta\right)A^2 \Bigg) A^4 \\
 & =  \left( 1- \frac{m_{n+1}}{4 \alpha_{n+1}^2   \beta_n}
\frac{\partial V_0}{\partial \alpha}\left(\frac{\alpha_n}{\alpha_{n+1}} A^2,\theta\right)A^4 \right)B,
\end{aligned}
\right.
\]
with $\theta = \pi$ or $\theta = \theta^0(A,m_n):=\widehat{\theta}^0 \left (\frac{\alpha_n}{\alpha_{n+1}} A^2,m_n \right )$, where $\widehat{\theta}^0$ is the function introduced in Lemma~\ref{lem:fitestermeacoblament}. Then, if $m_n$ and $m_{n+1}$ are small enough, they
admit two pairs of solutions, $A,B$, corresponding to $\theta=\pi$ and $\theta= \theta^0(A,m_n)$, 
\begin{equation*}
A  = 1+\OO(m_n,m_{n+1}), \qquad B = 1+\OO(m_n,m_{n+1}).
\end{equation*}
As a consequence $\theta^0(A,m_n)=\frac{\pi}{3} + \OO(m_n,m_{n+1})$.
\end{lemma}
\begin{proof}
We emphasize that, for $z\in \mathbb{C}$, $z\neq 1$, 
$\wh \A_{n+1}(z,\overline{z},0)= M_n - M_{n+1}=0$. Then, when $m_n=0$, 
$V_0(\alpha,\theta) = \wh \A_{n+1}(\alpha e^{i\theta}, \alpha e^{-i\theta},0)=0$ for all $\alpha,\theta$ such that $\alpha e^{i\theta}\neq 1$. Using this, the claim simply follows by applying the standard implicit function theorem at the value $(A,B,m_n,m_{n+1}) = (1,1,0,0)$, taking into account the definitions of $\alpha_n$, $\beta_n$, $\alpha_{n+1}$ and $\beta_{n+1}$ in~\eqref{def:alphanbetanetcexplicites} and Lemma~\ref{lem:fitestermeacoblament}.
\end{proof}

We expand $\wt \Theta$, introduced in~\eqref{def:wtTheta}, as
\begin{equation*}
\wt \Theta (\varphi,\rho) = \sum_{k\ge 0} \wt \Theta_k (\varphi) \rho^k.
\end{equation*}
We also introduce
\begin{equation}
\label{def:Gn0}
G_n^0 = \frac{\wt \Theta_0^0 A^4}{\alpha_{n+1}^2 \mu_{n+1}} \left(\frac{1}{\alpha_n^2\mu_n} + \frac{A^4}{\alpha_{n+1}^2 \mu_{n+1}}\right)^{-1},
\end{equation}
where $A$ is given by Lemma~\ref{lem:constantsAiB} and $\wt \Theta_0^0$ was introduced in~\eqref{def:wtTheta00}.
Observe that $G_n^0$ can take two different values, one for $\theta = \pi$ and another one for $\theta =\theta^0(A,m_n)$ in the definition of the constants $A$ and $B$. We use the same letter to denote both quantities.

We use $G_n^0$ to introduce a new variable $g_n$ through $G_n = G_n^0+ g_n$. This change, which preserves the $2$-form~\eqref{def:forminMcGehee},  only affects the kinetic energy part of the Hamiltonian, in~\eqref{def:wtTTT}, which now becomes
\begin{equation*}
\begin{aligned}
\wt \TTT_\omega(\varphi,\rho,x_n, y_n, x_{n+1}, y_{n+1},  g_n) =& \frac{\beta_n^2}{2\mu_n}  y_n^2+ \frac{\beta_{n+1}^2}{2\mu_{n+1}}  y_{n+1}^2 \\
&+ \frac{(G_n^0+g_n)^2 x_n^4}{4 \alpha_n^2\mu_n} + \frac{(\wt \Theta(\varphi,\rho)-G_n^0-g_n)^2 x_{n+1}^4}{4\alpha_{n+1}^2\mu_{n+1}}.
\end{aligned}
\end{equation*}

\subsection{Some steps of normal form}
In order to apply Theorems~\ref{thm:approximationflows} and~\ref{th:existenceflow} we will need  some coefficients of the expansions of $\wt \TTT_\omega$ and $\wt U_\omega$ in powers of $x_n$, $x_{n+1}$, $x_{n+1}/x_n$, $\theta_n$  and $\rho$
to be independent of $\varphi$. To accomplish this, we perform several steps of normal form, as is done in~\cite{BFM20}. We use the following immediate fact.
Given the generating function $S(\wt \varphi, \rho, y_n, \wt x_n, y_{n+1}, \wt x_{n+1}, \wt \theta_n, g_n)$, if the equations
\begin{equation}
\label{def:funciogeneradora}
\left\{
\begin{aligned}
\varphi & = \wt \varphi + \frac{\partial S}{\partial \rho}, & \wt \rho & = \rho + \frac{\partial S}{\partial \wt \varphi} \\
\frac{2\alpha_k\beta_k}{x_k^2} & = \frac{2\alpha_k\beta_k}{\wt x_k^2} + \frac{\partial S}{\partial y_k}, \quad & \frac{4\alpha_k\beta_k}{\wt x_k^3} \wt y_k & =
\frac{4\alpha_k\beta_k}{\wt x_k^3}  y_k - \frac{\partial S}{\partial \wt x_k}, \quad k=n,n+1 \\
\theta_n & = \wt \theta_n + \frac{\partial S}{\partial g_n}, & \wt g_n & = g_n + \frac{\partial S}{\partial \wt \theta_n},
\end{aligned}
\right.
\end{equation}
define a close to the identity map 
\[
T:(\varphi, \rho, x_n, y_n,  x_{n+1},  y_{n+1},    \theta_n, g_n) \mapsto
(\wt \varphi, \wt \rho, \wt x_n, \wt y_n,  \wt x_{n+1}, \wt y_{n+1},  \wt \theta_n, \wt g_n),
\] 
then $T$ preserves the $2$-form~\eqref{def:forminMcGehee}.

\begin{proposition}
\label{prop:Hamiltoniapromitjat}
Choose $\theta= \pi$ or $\theta = \theta^0(A,m_n)$ in Lemma~\ref{lem:constantsAiB} and $\Kc$ as in Proposition~\ref{prop:expansionsdetildeUnomegaitildeUnmesuomega}. Then, after an averaging procedure, Hamiltonian~\eqref{def:HT} becomes
\begin{multline*}
H_\omega(\varphi,\rho, \theta_n, g_n,x_n, y_n, x_{n+1}, y_{n+1}) \\
\begin{aligned}
& = \langle \omega, \rho \rangle  + \frac{\beta_n^2}{2\mu_n}  y_n^2+ \frac{\beta_{n+1}^2}{2\mu_{n+1}}  y_{n+1}^2 - \frac{m_n M_n}{2\alpha_n }x_n^2 - \frac{m_{n+1} M_{n+1}}{2\alpha_{n+1} }x_{n+1}^2  \\ & -  m_{n+1} \frac{x_{n+1}^2}{2\alpha_{n+1}} u_0\left(\frac{\alpha_n}{\alpha_{n+1}}\frac{x_{n+1}^2}{x_n^2} e^{i \theta_{n}},\frac{\alpha_n}{\alpha_{n+1}}\frac{x_{n+1}^2}{x_n^2} e^{-i \theta_{n}}\right)  + \frac{1}{4\alpha_n^2 \mu_n} (G_n^0+g_n)^2 x_n^4
  + \wh \Theta_0 x_{n+1}^4
  \\
 & - \frac{1}{2\alpha_{n+1}^2\mu_{n+1}} (\wt \Theta_0^0 - G_n^0) x_{n+1}^4 g_n
 + \frac{1}{4\alpha_{n+1}^2 \mu_{n+1}} x_{n+1}^4 g_n^2
 + R(\varphi,\rho, \theta_n, g_n,x_n, y_n, x_{n+1}, y_{n+1}),
\end{aligned}
\end{multline*}
where
\begin{enumerate}
\item[(1)] the function $u_0(z,\bar z)$ was introduced in~\eqref{def:u0} (see also~\eqref{eq:expansiodeUnmes1}),
\item[(2)] $G_n^0$ and $\wt \Theta_0^0$ were introduced in~\eqref{def:Gn0} and~\eqref{def:wtTheta00}, respectively, and depend on the choice of
$\theta$ in Lemma~\ref{lem:constantsAiB},
\item[(3)] $\wh \Theta_0 = [(\wt \Theta_0(\varphi)-G_n^0)^2]/(4\alpha_{n+1}^2\mu_{n+1})$ is a constant,
\item[(4)] the remainder has the form
\begin{multline*}
R(\varphi,\rho, \theta_n, g_n,x_n, y_n, x_{n+1}, y_{n+1}) \\
 = \sum_{\scriptsize \begin{array}{c} k,j,m,l,r,s \ge 0,\\ k+j  \ge 2\end{array}} u_{k,j,m,l,r,s}\left(\varphi,\frac{x_{n+1}^2}{x_n^2}e^{i \theta_n},\frac{x_{n+1}^2}{x_n^2}e^{-i \theta_n}\right) x_n^{2k} x_{n+1}^{2j} y_n^m y_{n+1}^l g_n^r \rho^s,
\end{multline*}
and there exists $\bar{\varrho}$ depending on $\varrho$, such that $u_{k,j,m,l,r,s}$ are analytic in its arguments when 
$\varphi\in \T^{2(n-1)}$, 
$|1-z_1|, |1-\overline{z_1}| >\bar{\varrho}^{-1}$ 
with $z_1= \frac{x_{n+1}^2}{x_{n}^2} e^{i\theta_n}$. In addition $R$ satisfies
\[
\begin{aligned}
\frac{\partial R}{\partial x_n} & = \OO_5(x_n,x_{n+1}),&
\frac{\partial R}{\partial x_{n+1}}  = & \OO_5(x_n,x_{n+1}),\\
\frac{\partial R}{\partial y_n} & = \OO_6(x_n,x_{n+1}),& \frac{\partial R}{\partial y_{n+1}}  =&  \OO_6(x_n,x_{n+1}),\\
\frac{\partial R}{\partial \varphi} & = (\OO(\rho)+\OO_2(y_{n+1},g_n,x_n,x_{n+1})) \OO_6(x_n,x_{n+1}),& \frac{\partial R}{\partial \rho} = &  \OO_4(x_n,x_{n+1}),\\
\frac{\partial R}{\partial \theta_n} & = \OO(\rho)\OO_4(x_n,x_{n+1}) + \OO_6(x_n,x_{n+1}),& \frac{\partial R}{\partial g_n}  = & \OO(\rho)\OO_4(x_n,x_{n+1}) \\
& & & \ + \OO_6(x_n,x_{n+1}).
\end{aligned}
\]
\end{enumerate}
\end{proposition}
\begin{proof}
Using Proposition~\ref{prop:expansionsdetildeUnomegaitildeUnmesuomega} for $\wt{U}_{n,\omega}$ and $\wt{U}_{n+1,\omega}$ we write (with the notation $z= \frac{\alpha_n x_{n+1}^2}{\alpha_{n+1} x_{n}^2} e^{i\theta_n} $)
\begin{align*}
\wt \TTT_\omega(\varphi,\rho,x_n,& y_n, x_{n+1}, y_{n+1},  g_n) = \frac{\beta_n^2}{2\mu_n}  y_n^2+ \frac{\beta_{n+1}^2}{2\mu_{n+1}}  y_{n+1}^2 + \frac{(G_n^0+g_n)^2 x_n^4}{4 \alpha_n^2\mu_n} + \frac{(G_n^0+g_n)^2 x_n^4}{4 \alpha_n^2\mu_n} \\ &+ \frac{ x_{n+1}^4 g_n^2 }{4\alpha_{n+1}^2 \mu_{n+1}}+ \frac{(\wt \Theta(\varphi,\rho)-G_n^0)^2 x_{n+1}^4}{4\alpha_{n+1}^2\mu_{n+1}} - \frac{1}{2 \alpha_{n+1}^2 \mu_{n+1}} (\wt \Theta(\varphi,\rho) - G_n^0 ) g_n x_{n+1}^4, \\
\wt U_\omega (\varphi, \rho, x_n, &x_{n+1}, \theta_n)
 =  \frac{m_n M_n}{2\alpha_n }x_n^2 + \frac{m_{n+1} M_{n+1}}{2\alpha_{n+1} }x_{n+1}^2
+ m_n \frac{x_n^6}{2\alpha_n} c_{2,0}(\varphi, \theta_n) 
\\ &+  m_{n+1} \frac{x_{n+1}^2}{2\alpha_{n+1}} u_0 \left ( z,\overline{z} \right ) + m_{n+1} \frac{x_{n+1}^6}{2\alpha_{n+1}} u_{2,0}(\varphi, z,\overline{z}) + R_0(\varphi,\rho,x_n,\theta_n)
\end{align*}
 with $u_{2,0}(\varphi,z,\bar{z})= u_{2}(\varphi,\rho,z,\bar{z})$ and $R_0$ satisfying the properties stated for $R$ in the Proposition. Indeed, the problematic terms are the ones of the form $u_k(\varphi, \rho, z, \bar{z}) x_{n+1}^{2k+2}$, $k\geq 3,$ with $u_k$ analytic. For those terms
$$
\partial_{x_n} (\wt u_k(\varphi,\rho,z,\bar{z}) x_{n+1}^{2k+2} ),\, 
\partial_{x_{n+1}} (\wt u_k(\varphi,\rho,z,\bar{z}) x_{n+1}^{2k+2} ) 
= \OO(x_{n}^{2k+1}),
$$
provided $|z-1|, |\bar{z}-1| > \varrho^{-1}$ and $m_{n},m_{n+1}$ are small enough, 
and its is immediate to check that these terms satisfy the other properties stated for $R$.  
 
 Therefore, the terms on the Hamiltonian we need to average out are the following:
\begin{enumerate}
\item[(1)] $x_{n+1}^4$ in $\wt \TTT_\omega$,
\item[(2)] $x_{n+1}^4 g_n$ in $\wt \TTT_\omega$,
\item[(3)] $x_n^6$ in $x_n^2 \wt U_{n,\omega}$,
\item[(4)] $x_{n+1}^6 u(\varphi,\rho, \frac{\alpha_n x_{n+1}^2}{\alpha_{n+1} x_n^2} e^{i\theta_n},\frac{\alpha_n x_{n+1}^2}{\alpha_{n+1} x_n^2} e^{-i\theta_n})$ that comes
 from the term $u_{2}$ in $\wt U_{\omega}$ and a contribution from the averaging step (2),  and
\item[(5)] $\beta_{n+1}b_1(\varphi) x_{n+1}^6 y_{n+1}/(4\alpha_{n+1} \mu_{n+1})$. This term appears after the averaging step (1).
 \end{enumerate}

We average them out with a sequence of transformations defined through~\eqref{def:funciogeneradora} with suitable generating functions $S$.
We drop the tildes in the variables after each step. Along the proof, after performing each step of averaging, we take care about the new terms that can not be considered as a remainder. The tedious but immediate substitution of the sequence of transformations is left to the reader.

We recall that, given a function $f$ depending on some angles $\wt \varphi$, $[f]$ denotes its average with respect $\wt \varphi$. 

We start with (1). We consider in~\eqref{def:funciogeneradora} the generating function $S(\wt \varphi, \rho, \wt x_{n+1}) = b_1(\wt \varphi,\rho) \wt x_{n+1}^4$, where $\langle \omega, \nabla_{\wt \varphi} b_1 \rangle  =  (4\alpha_{n+1}^2 \mu_{n+1})^{-1}\Big(\big (\wt{\Theta}(\tilde{\varphi},\rho) - G_n^0 \big )^2- [ \big (\wt{\Theta}(\tilde{\varphi},\rho) - G_n^0\big )^2]\Big)$. We recall that $\omega$ is Diophantine, and then the existence and analyticity of $b_1$ is guaranteed by Theorem~\ref{thm:smalldivisors} (see also~\cite{Russmann75}). In addition, we can select it with zero mean. Therefore,
\begin{equation*}
\begin{aligned}
\rho & = \wt \rho - \nabla_{\wt \varphi} b_1(\wt \varphi,\rho)\wt x_{n+1}^4, \qquad 
 &\varphi&=\wt \varphi + \nabla_\rho b_1 (\wt \varphi,\rho) \wt x_{n+1}^4, \\ 
 y_{n+1} & = \wt y_{n+1} +\frac{1}{\alpha_{n+1}\beta_{n+1}} b_1(\wt \varphi,\rho) \wt x_{n+1}^6. \qquad & &
 \end{aligned}
\end{equation*}
After this change, the term with $x_{n+1}^4$ in the kinetic energy becomes 
$$[ \big (\wt{\Theta}(\tilde{\varphi},\rho) - G_n^0\big )^2] \frac{1}{4 \alpha_{n+1}^2 \mu_{n+1}} = \wh \Theta_0 + \OO(\rho), $$
with $\OO(\rho)$ satisfying the conditions for the remainder $R$,  to which we add  
the term (besides some other terms considered as a remainder) 
$$
 b_{1}(\tilde{\varphi},\tilde{\rho})\frac{\beta_{n+1}}{\alpha_{n+1} \mu_{n+1}} \wt x_{n+1}^6 \wt y_{n+1}.
$$
We will average out this term in step (5).   

As for (2), we consider $b_2(\tilde{\varphi}, \rho)$ satisfying 
$\langle \omega, \nabla_{\wt \varphi} b_2 \rangle =  - (2 \alpha_{n+1}^2 \mu_{n+1})^{-1} (\wt \Theta(\wt \varphi, \rho) - G_n^0) - [\wt \Theta (\wt \varphi, \rho) - G_n^0]
$ and the generating function 
$S(\wt \varphi,\rho, \wt x_{n+1}, g_n)= b_2 (\wt \varphi,\rho) \wt x_{n+1}^4 g_n $. 
Again, since $\omega$ is Diophantine, this equation can be solved. It defines the change 
\begin{equation*}
\begin{aligned}
\rho & = \wt \rho -  \nabla_{\wt \varphi}  b_2 (\wt \varphi,\rho) \wt x_{n+1}^4 \wt g_n, \qquad &\varphi &= \wt \varphi + \nabla_\rho b_2 (\wt \varphi, \rho) \wt x_{n+1}^4 \wt g_n, \\ 
y_{n+1} & = \wt y_{n+1} + \frac{2}{\alpha_{n+1}\beta_{n+1}}  b_2(\wt \varphi,\rho ) \wt x_{n+1}^6 \wt g_n, \qquad 
&\theta_n & = \wt \theta_n + b_2(\wt \varphi,\rho) \wt x_{n+1}^4 
.
\end{aligned}
\end{equation*}
We emphasize that after this change the coefficient of $x_{n+1}^4 g_n$ becomes  
$$ 
-x_{n+1}^4 g_n \frac{([\wt \Theta (\wt \varphi, \rho)] - G_n^0)}{2\alpha_{n+1}^2 \mu_{n+1}} = 
-x_{n+1}^4 g_n \frac{(\wt \Theta_0^0 - G_n^0) }{2 \alpha_{n+1}^2 \mu_{n+1}} + \OO(x_{n+1}^4 g_n \rho)
$$
with $\OO(x_{n+1}^4 g_n \rho)$, independent on $\varphi$, satisfying the remainder conditions. Moreover, this change of variables produces a new term in the Hamiltonian of the form
\begin{equation}\label{def:tildeu}
\sum_{j\geq 1} \tilde{u}_j (z,\bar{z})  x_{n+1}^{2+4j} (b_{2} (\varphi,\rho))^j  = \tilde{u}_1 (z,\bar{z}) x_{n+1}^6 b_2(\varphi,0) + \OO(\rho x_{n+1}^6) + \OO(x_{n+1}^{10}),
\end{equation}
where $\wt u_j $ are analytic with respect their arguments, provided $|z-1|,|\bar{z}-1|> \varrho^{-1}$, see Proposition~\ref{prop:expansionsdetildeUnomegaitildeUnmesuomega}.

The coefficient $\wt u_1$ is averaged out in step (4). The rest of the terms go to the remainder.

Now we deal with (3).  
We consider $S(\wt \varphi, \wt \theta_n,\wt x_n) = b_{3}(\wt \varphi,\wt \theta_n) \wt x_n^6$, where $\langle \omega, \nabla_{\wt \varphi} b_{3} \rangle  = c_{2,0} - [c_{2,0}]$, 
and is straightforwardly checked that, after the change of variables induced by the generating function $S$, the new coefficient of $x_n^6$ is $[c_{2,0}]$ and that the remainder satisfies the required properties. 

To deal with (4), we consider a generating function of the form
\[
S(\wt \varphi, \wt x_n, \wt x_{n+1}) =   
\wh S\left(\wt \varphi,\frac{\alpha_{n}}{\alpha_{n+1}}\frac{\wt x_{n+1}^2}{\wt x_n^2}e^{i\wt \theta_n}, \frac{\alpha_{n}}{\alpha_{n+1}}\frac{\wt x_{n+1}^2}{\wt x_n^2}e^{-i\wt \theta_n}\right)\wt x_{n+1}^6,
\]
with, $\wh S$ satisfying
$$
\langle \omega, \nabla_{\wt \varphi}  \wh S\rangle = \frac{m_{n+1}}{2 \alpha_{n+1}} (u_{2,0}- [u_{2,0}]) + \wt u_{1} b_2 ,
$$
where $\wt u_{1}b_2 $, introduced in~\eqref{def:tildeu}, has zero mean.
In this case,  through~\eqref{def:funciogeneradora} 
$S$ defines the change,  
\begin{equation*}
\begin{aligned}
\rho & = \wt \rho -  \nabla_{\wt \varphi} \wh S\left(\wt \varphi,\rho, \wt z , \overline{\wt z} \right) \wt x_{n+1}^6, \qquad 
&g_n & = \wt g_n +  F_3\left(\wt \varphi,\rho, \wt z , \overline{\wt z} \right) \wt x_{n+1}^6 \\
y_{n} & = \wt y_{n} +  F_1\left(\wt \varphi,\rho, \wt z , \overline{\wt z} \right)\wt x_n^2 \wt x_{n+1}^6, \qquad 
&y_{n+1} & = \wt y_{n+1} +  F_2\left(\wt \varphi,\rho, \wt z , \overline{\wt z} \right) \wt x_{n+1}^8,
\end{aligned}
\end{equation*}
where $\tilde{z} =\frac{\alpha_{n}}{\alpha_{n+1}}\frac{\wt x_{n+1}^2}{\wt x_n^2} $,
and $F_i$, $i=1,2,3$, are analytic functions of their arguments.

Finally, in (5), we consider
\[
S(\wt \varphi, y_{n+1}, \wt x_{n+1}) = b_3(\wt \varphi)\wt x_{n+1}^6 y_{n+1},
\]
where 
$\langle \omega, \nabla_{\wt \varphi}  b_3 \rangle = {\beta_{n+1}b_1/(\alpha_{n+1} \mu_{n+1})}$.
Equations~\eqref{def:funciogeneradora} define the change
\begin{equation*}
\begin{aligned}
x_{n+1} & = \tilde{x}_{n+1} (1+ b_3 (\tilde{\varphi}) x_{n+1}^8)^{-1/2} = \wt x_{n+1}+ \wt S_1 \left(\wt \varphi,\wt x_{n+1}^8 \right)\wt x_{n+1}^9, \\
y_{n+1} & = \wt y_{n+1} (1 - 6 b_3(\tilde{\varphi})\wt x_{n+1}^2)^{-1/2} = \wt y_{n+1}+ \wt S_2 \left(\wt \varphi,\wt x_{n+1}^8\right) \wt x_{n+1}^8 \wt y_{n+1}, \\
\rho &=\wt \rho - \nabla_{\wt \varphi} b_3(\tilde\varphi) \wt x_{n+1}^6 y_{n+1}= \wt \rho + \wt S_3(\wt \varphi, \wt x_{n+1}^8) \wt x_{n+1}^6 \wt y_{n+1},
\end{aligned}
\end{equation*}
where $\wt S_i$, $i=1,2,3$, are analytic in their arguments.
\end{proof}

\subsection{Regularization of infinity}

In what follows, $\wt \theta_0$ will be either $-\pi$ or $\theta^0(A,m_n) = \frac{\pi}{3} + \OO(m_n,m_{n+1})$ in Lemma~\ref{lem:constantsAiB}. Recalling that
\begin{equation*}
u_0\left(\frac{\alpha_n}{\alpha_{n+1}}\frac{x_{n+1}^2}{x_n^2} e^{i \theta_{n}},\frac{\alpha_n}{\alpha_{n+1}}\frac{x_{n+1}^2}{x_n^2} e^{-i \theta_{n}}\right) = V_0\left(\frac{\alpha_n}{\alpha_{n+1}}\frac{x_{n+1}^2}{x_n^2},\theta_n \right)
\end{equation*}
where $V_0$ was introduced in Lemma~\ref{lem:fitestermeacoblament}, we define
\begin{equation}
\label{def:wtVj0} 
\wt v_{i,j} = \frac{\partial^{i+j} V_0}{\partial \alpha^i \partial \theta^j} \left(\frac{\alpha_n}{\alpha_{n+1}} A^2, \wt \theta_0 \right), \qquad i,j \ge 0.
\end{equation}
By Lemma~\ref{lem:fitestermeacoblament}, $\wt v_{i,j} = \OO(m_n)$.

For future purposes, we introduce the constants
\begin{equation} \label{def:nuGamman}
\nu = \sqrt{1- \frac{m_{n+1}}{4\alpha_{n+1}^2 \beta_n}A^4 \wt v_{1,0}} = 1+ \OO(m_{n+1}),\qquad
\Gamma_n = \frac{1}{2}\left(\frac{1}{\alpha_{n}^2 \mu_{n}} +\frac{A^4}{\alpha_{n+1}^2 \mu_{n+1}}\right),
\end{equation}
where $A$ and $B$ were introduced in Lemma~\ref{lem:constantsAiB}, whose value depends  on the choice of $\wt \theta_0$. We notice that, since
\[
\frac{\mu_{n}}{\mu_{n+1}} = \frac{M_{n+2} M_n m_n}{M_{n+1}^2 m_{n+1}} = \frac{m_n}{m_{n+1}} (1+ \OO(m_n,m_{n+1})),
\]
$A=1 + \OO(m_n,m_{n+1})$ and the conditions~\eqref{def:alphanbetanetcimplicites} and~\eqref{bound:quocientalphanalphanmes1}, we have that 
\begin{equation}\label{asymptoticGamma_n}
\Gamma_n = \frac{1}{2\alpha^2_n \mu_n} \left ( 1 + \frac{1}{m_{n+1}} (m_n + \OO_2(m_n,m_{n+1}) )\right ) =
\frac{8 \alpha_n}{ M_n m_n m_{n+1}} (m_{n+1}+ m_{n} + \OO_2 (m_n,m_{n+1})).
\end{equation}

The regularization will be obtained as a sequence of
simple changes of variables and blow-ups that are summarized in the following technical result.
\begin{proposition}
Consider the blow-ups given by 
\begin{equation*}
\begin{aligned}
x_{n+1}& =x_n(A+ \xi_{n+1}), \qquad y_{n+1}=y_{n}(B+\eta_{n+1}), \qquad  y_{n}=x_{n}(\nu + \zeta_n), \\
\xi_{n+1}&=x_n \wt \xi_{n+1}, \quad  \eta_{n+1}  = x_n \tilde\eta_{n+1},\quad 
\theta_n = \wt \theta_0 + x_n \wt \theta_n, \quad g_n = \Gamma_n^{-1}\wt g_n ,\quad  \rho = x_n^3 \wt \rho.
\end{aligned}
\end{equation*}
Then, denoting $Z=(\zeta_n,\wt \xi_{n+1}, \wt \eta_{n+1}, \wt \theta_n, \wt g_n, \wt \rho_n)$ there exists a linear change of variables $\wt Z= \mathbf{C}Z$, where
\[
C = \begin{pmatrix}
1 & 0 & 0 & 0 & 0 & 0 \\
0 & 1+\delta_{2,2} & 1+\delta_{2,3} & \delta_{2,4} & \delta_{2,5} & 0 \\
0 & -4+\delta_{3,2} & 1+\delta_{3,3} & \delta_{3,4} & \delta_{3,5} & 0 \\
0 & \delta_{4,2} & \delta_{4,3} & 1+\delta_{4,4} & 1+\delta_{4,5} & 0 \\
0 & \delta_{5,2} & \delta_{5,3} & \delta_{5,4} & -1+\delta_{5,5} & 0 \\
0 & 0 & 0 & 0 & 0 &  \Id
\end{pmatrix},
\]
and
\[
\delta_{i,j} = \OO(m_n,m_{n+1}),
\]
such that in these variables the Hamiltonian system with Hamiltonian $H_\omega$ has the equations
\begin{equation}
\label{eq:sistemacompletredressat}
\left\{
\begin{aligned}
\dot x_n & = -\nu x_n^4 +  x_n^4 \OO_1(\zeta_n)+ \OO_9(x_n),\\
\dot{\wt Z} & = x_n^3 \mathbf{M} \wt Z + x_n^3 \OO_2(x_n,\wt Z), \\
\dot \varphi & = \omega + x_n^3\OO_1(x_n, \wt Z),
\end{aligned}
\right.
\end{equation}
where 
\[
\mathbf{M} = \begin{pmatrix}
2+\eps_{1,1} & 0 & 0 & 0 & 0 & 0 \\
0 & 3+\eps_{2,2} & 0 & 0 & 0 & 0 \\
0 & 0 & -2+\eps_{3,3} & 0 & 0 & 0 \\
0 & 0 & 0 & 1+\eps_{4,4} & 0 & 0 \\
0 & 0 & 0 & 0 & -\gamma_2+\eps_{5,5} & 0 \\
0 & 0 & 0 & 0 & 0 & (1+\eps_{6,6}) \Id
\end{pmatrix},
\]
with
\[
\left\{
\begin{aligned} & \eps_{i,i}  = \OO(m_n,m_{n+1}), \qquad \text{if $\ \ i \neq 5$},\\
& \eps_{5,5}  = \OO_2(m_n,m_{n+1}).
\end{aligned}
\right.
\]
\end{proposition}
\begin{remark} 
Notice that, since the hypotheses of the existence result, Theorem~\ref{th:existenceflow}, only depend on the dominant terms, there is no need to control the dependence on $m_n,m_{n+1}$ of the non dominant terms. 
\end{remark}
\begin{proof}
We perform the blow ups in three steps. The first one corresponds to $(\xi_{n+1}, \eta_{n+1})$:
\begin{equation*} 
x_{n+1}  = x_n (A+ \xi_{n+1}), \qquad 
y_{n+1}  = y_n (B+\eta_{n+1}).
\end{equation*}

For any choice of $\wt \theta_0$ we have $\wt v_{0,1} = 0$ (see definition~\eqref{def:wtVj0} of $\wt v_{0,1}$ and Lemma~\ref{lem:fitestermeacoblament}).
We recall that the equations of motion associated to the Hamiltonian $H_\omega$, in Proposition~\ref{prop:Hamiltoniapromitjat}, are obtained using the $2$-form~\eqref{def:forminMcGehee} taking into account the choice of the constants $\alpha_n$, $\alpha_{n+1}$, $\beta_n$, and $\beta_{n+1}$ in~\eqref{def:alphanbetanetcimplicites}. Then, also using Lemma\ref{lem:constantsAiB} we have that
\[
\begin{aligned}
\dot \xi_{n+1} & = \frac{1}{x_n} \dot x_{n+1} - \frac{x_{n+1}}{x_n^2} \dot x_n  
  = - \frac{1}{4 \alpha_{n+1} \beta_{n+1}} \frac{x_{n+1}^3}{x_n} \frac{\partial H_\omega}{\partial y_{n+1}} +
\frac{1}{4 \alpha_{n} \beta_{n}}x_n x_{n+1} \frac{\partial H_\omega}{\partial y_{n}} \\
& = - \frac{1}{4 \alpha_{n+1} \beta_{n+1}} \frac{x_{n+1}^3}{x_n} \left( \frac{\beta_{n+1}^2}{\mu_{n+1}} y_{n+1}
+ \frac{\partial R}{\partial y_{n+1}}\right) +
\frac{1}{4 \alpha_{n} \beta_{n}}x_n x_{n+1} \left( \frac{\beta_n^2}{\mu_n} y_n + \frac{\partial R}{\partial y_n}\right) \\
& = - x_n^2 y_n (A+\xi_{n+1})^3(B+\eta_{n+1}) + x_n^2 y_n (A+\xi_{n+1}) + \OO_8(x_n) \\
& = x_n^2 y_n (-A^3 B +A)-\left((3A^2 B-1) \xi_{n+1} + A^3 \eta_{n+1} \right) x_n^2 y_n +\OO_2(\xi_{n+1},\eta_{n+1}) x_n^2 y_n
+ \OO_8(x_n) \\
&=  \left [-(2+ \OO(m_n,m_{n+1}) \xi_{n+1}   - (1+ \OO(m_n,m_{n+1}) \eta_{n+1}    + \OO_2(\xi_{n+1},\eta_{n+1}) \right ] x_n^2 y_n
+ \OO_8(x_n).   
\end{aligned}
\] 
To avoid cumbersome notation, $V_0$ (and its derivatives) means $V_0$ evaluated at $\left(\frac{\alpha_n}{\alpha_{n+1}}\frac{x_{n+1}^2}{x_n^2},\theta \right)$. 
Similar computations, recalling that $x_{n+1}=x_n(A+ \xi_{n+1})$ and $y_{n+1}=y_n(B+\eta_{n+1})$, and using again Lemma~\ref{lem:constantsAiB}, lead us to :
\[
\begin{aligned}
\dot \eta_{n+1}  = &\frac{1}{y_n} \dot y_{n+1} - \frac{y_{n+1}}{y_n^2} \dot y_n  =  \frac{1}{4 \alpha_{n+1} \beta_{n+1}} \frac{x_{n+1}^3}{y_n}\frac{\partial H_\omega}{\partial x_{n+1}} -
\frac{1}{4 \alpha_{n} \beta_{n}} \frac{x_{n}^3 y_{n+1}}{y_n^2} \frac{\partial H_\omega}{\partial x_{n}} \\
 =  &\frac{1}{4 \alpha_{n+1} \beta_{n+1}} \frac{x_{n+1}^4}{y_n}
 \left(-\frac{m_{n+1} M_{n+1}}{\alpha_{n+1}}    - \frac{m_{n+1}}{\alpha_{n+1}}   V_0 - \frac{m_{n+1}}{\alpha_{n+1}} \frac{\alpha_n}{\alpha_{n+1}} \frac{x_{n+1}^2}{x_n^2} \frac{\partial V_0}{\partial \alpha}+ x_{n+1}^{-1}\OO_5(x_n,x_{n+1})
 \right)\\
& -
\frac{1}{4 \alpha_{n} \beta_{n}} \frac{x_{n}^4 y_{n+1}}{y_n^2}
\left(-\frac{m_n M_n}{\alpha_n}   +\frac{m_{n+1}}{\alpha_{n+1}} \frac{\alpha_n}{\alpha_{n+1}} \frac{x_{n+1}^4}{x_n^4} \frac{\partial V_0}{\partial \alpha}+ x_{n}^{-1}\OO_5(x_n,x_{n+1})
 \right)
  \\
   =& \frac{x_n^4}{y_n} \Bigg(-A^4\left(1+\frac{1}{M_{n+1}} \wt v_{0,0}+\frac{1}{M_{n+1}}\frac{\alpha_n}{\alpha_{n+1}} \wt v_{1,0} A^2\right)+ B \left(1 - \frac{m_{n+1}}{4 \alpha_{n+1}^2\beta_n} \wt v_{1,0} A^4\right) \\
   & + L_1 \xi_{n+1} + L_2 \eta_{n+1} + L_3(\theta-\wt \theta_0)+ \OO_2 (\xi_{n+1},\eta_{n+1},\theta-\wt \theta_0) + \OO_4(x_n) \Bigg) \\
  = & \frac{x_n^4}{y_n} \left (L_1 \xi_{n+1} + L_2 \eta_{n+1} + L_3(\theta-\wt \theta_0)+ \OO_2 (\xi_{n+1},\eta_{n+1},\theta-\wt \theta_0) + \OO_4(x_n)  \right ),
\end{aligned}
\]
where, taking into account~\eqref{def:wtVj0} and ~\eqref{def:alphanbetanetcimplicites},
$$
L_1 
= -4 + \OO(m_n,m_{n+1}), \qquad 
L_2  =  
1+ \OO(m_n,m_{n+1}), \qquad 
L_3  = 
\OO(m_n,m_{n+1}).
$$

We emphasize that the non-explicit error terms are now analytic functions in their variables, the only non-regular factor being the quotient $x_n^4/y_n$.

The rest of the equations can be obtained immediately from the Hamiltonian structure and Proposition~\ref{prop:Hamiltoniapromitjat}. Concerning $x_n$ and $y_n$, using~\eqref{def:alphanbetanetcimplicites}, we have that
\begin{equation}
\label{equacionsperxniyn}
\begin{aligned}
\dot x_n & = - \frac{x_n^3}{4\alpha_n \beta_{n}}\frac{\partial H_\omega}{\partial y_n} = - x_n^3 y_n + \OO_9(x_n),\\
\dot y_n & = \frac{x_n^3}{4\alpha_n \beta_{n}}\frac{\partial H_\omega}{\partial x_n}= -\left( 1- \frac{m_{n+1}}{4\alpha_{n+1}^2 \beta_n}A^4 \wt v_{1,0}\right)x_n^4 + x_n^4  \OO_1(\xi_{n+1},\theta_n-\wt \theta_0)+ \OO_8(x_n).
\end{aligned}
\end{equation}
In the case of $\theta_n$ and $g_n$, by the choice of $G_n^0$ in~\eqref{def:Gn0} and Lemma~\ref{lem:fitestermeacoblament} and using that, by Lemma~\ref{lem:fitestermeacoblament} and the choice of $\wt \theta_0$, $\wt v_{0,1} = 0$, the equations are
\[
\begin{aligned}
\dot \theta_n  = & \frac{\partial H_\omega}{\partial g_n} = \frac{1}{2}\left(\frac{1}{\alpha_{n}^2 \mu_{n}} +\frac{A^4}{\alpha_{n+1}^2 \mu_{n+1}}\right)x_{n}^4 g_n + x_n^4 \OO_1(\xi_{n+1}) + \OO(\rho) \OO_4(x_n) + \OO_6(x_n),\\
\dot g_n  = & -\frac{\partial H_\omega}{\partial \theta_n} = \frac{m_{n+1}}{2\alpha_{n+1}} x_{n+1}^2 \frac{\partial V_0}{\partial \theta_n}
\left(\frac{\alpha_n}{\alpha_{n+1}}\frac{x_{n+1}^2}{x_n^2},\theta_n\right) +  \OO(\rho)\OO_4(x_n,x_{n+1}) + \OO_6(x_n,x_{n+1})\\
 = & \frac{m_{n+1}}{2\alpha_{n+1}} x_{n}^2 (A+\xi_{n+1})^2\frac{\partial V_0}{\partial \theta_n}
\left(\frac{\alpha_n}{\alpha_{n+1}}(A+\xi_{n+1})^2,\theta_n\right) +  \OO(\rho)\OO_4(x_n) + \OO_6(x_n) \\
 = & \frac{\alpha_n}{\alpha_{n+1}^2} m_{n+1}A^3 \wt v_{1,1} x_n^2 \xi_{n+1} + \frac{1}{2  \alpha_{n+1}} m_{n+1}A^2 { {\wt v_{0,2}}} x_n^2 (\theta_n-
\wt \theta_0) + x_n^2 \OO_2(\xi_{n+1},\theta-\wt \theta_0)\\
& +  \OO(\rho)\OO_4(x_n) + \OO_6(x_n).
\end{aligned}
\]
In view of Lemma~\ref{lem:fitestermeacoblament}, if $\wt \theta_0 = 0$, then $\wt v_{1,1} = 0$ but, if $\wt \theta_0 = \theta^0(A,m)$,
then $\wt v_{1,1} \neq 0$. The coefficient $\wt v_{0,2} $ is different from 0 for both choices of $\wt \theta_0$.

Finally, the equations for $\varphi$ and $\rho$ become
\begin{equation}
\label{eq:varphirho}
\begin{aligned}
\dot \varphi & = \frac{\partial H_\omega}{\partial \rho}=  \omega  + \OO_4(x_n) ,\\
\dot \rho & = -\frac{\partial H_\omega}{\partial \varphi} = (\OO(\rho)+\OO_2(y_n,g_n,x_n)) \OO_6(x_n).
\end{aligned}
\end{equation}
The change $y_n=x_n(\nu + \zeta_n$)  regularizes the term $x_n^4/y_n$ in the equation for $\eta_{n+1}$. Indeed, with this change,  
\begin{equation*}
\begin{aligned}
\dot \xi_{n+1}  = & \nu x_n^3 (1+ \nu^{-1} \zeta_n)\big[-(2+\OO(m_n,m_{n+1})) \xi_{n+1} - (1+\OO(m_n,m_{n+1})) \eta_{n+1}    +\OO_2(\xi_{n+1},\eta_{n+1} )\big] \\ & + \OO_8(x_n),\\
\dot \eta_{n+1}  = & \nu^{-1} x_n^3 (1 + \nu^{-1} \zeta_n)^{-1}\big[ -(4+\OO(m_n,m_{n+1})) \xi_{n+1} + (1+\OO(m_n,m_{n+1})) \eta_{n+1} \\
&  + \OO(m_n,m_{n+1})(\theta_n-\wt \theta_0)+\OO_2(\xi_{n+1},\eta_{n+1},\theta_n-\wt \theta_0 ) +\OO_4(x_n) \big],
\end{aligned}
\end{equation*}
while equations~\eqref{equacionsperxniyn} are transformed into
\begin{equation*}
\begin{aligned}
\dot x_n & = - \nu x_n^4  (1+ \nu^{-1} \zeta_n)+\OO_9(x_n),\\
\dot \zeta_n & = 2 \nu x_n^3 \zeta_n +  x_n^3  \OO_1(\xi_{n+1},\theta_n-\wt \theta_0)+x_n^3\zeta_n^2+\OO_5(x_n).
\end{aligned}
\end{equation*}
Equations~\eqref{eq:varphirho} become
\begin{equation*}
\begin{aligned}
\dot \varphi & = \omega + \OO_4(x_n),\\
\dot \rho &  = (\OO(\rho)+\OO_2(x_n,g_n)) \OO_6(x_n).
\end{aligned}
\end{equation*}
The equations for $\theta_{n},g_n$ remain unchanged (the higher order terms $\OO_l$ can change their explicit expression but they keep the same order). 

After this change, the vector field is analytic in its arguments in a neighborhood of
\[
\{\varphi\in \T, \;\rho= 0,\; x_n=0,\; \zeta_n = 0,\; \xi_{n+1}= 0,\; \eta_{n+1}= 0,\; \theta_n\in \T ,g_n = 0\}.
\]

Now we deal with the last blow-up:
\begin{equation*}
\xi_{n+1}  = x_n \wt \xi_{n+1}, \quad \theta_n   = \wt \theta_0 + x_n \wt \theta_n,  \quad \rho   = x_n^3 \wt \rho,\quad 
\eta_{n+1}  = x_n \wt \eta_{n+1},    \quad g_n   = \Gamma_n^{-1}\wt g_n.
\end{equation*}
 
Proceeding as before, it is immediate to check that  
\begin{equation}
\label{equacionsperxniynterceraversio}
\begin{aligned}
\dot x_n & = - \nu x_n^4 + x_n^4 \OO_1(\zeta_n)+ \OO_9(x_n),\\
\dot \zeta_n & = 2 \nu x_n^3 \zeta_n +  x_n^4  \OO_1(\wt \xi_{n+1},\wt \theta_n)+x_n^3\zeta_n^2+\OO_5(x_n).
\end{aligned}
\end{equation}
Also, for $(\wt \xi_{n+1},\wt \eta_{n+1})$,
\begin{align}
\dot{\tilde{ \xi}}_{n+1}  = & \nu x_n^3 \big[-(1+\OO(m_n,m_{n+1})) \wt \xi_{n+1} - (1+\OO(m_n,m_{n+1})) \wt \eta_{n+1}   + \OO_2(\wt \xi_{n+1},\wt \eta_{n+1},\zeta )\big] + \OO_7(x_n), \notag \\
\dot{\tilde{\eta}}_{n+1}  = & \nu^{-1} x_n^3\big[ -(4+\OO(m_n,m_{n+1}))\wt \xi_{n+1} + (1+ \nu^2+\OO(m_n,m_{n+1}))  \tilde\eta_{n+1} 
\label{eq:xinmes1etanmes1terceraversio}\\
&  + \OO(m_n,m_{n+1})\wt \theta + x_n\OO_2(\tilde\xi_{n+1},\tilde\eta_{n+1},\tilde\theta)+\OO_3(x_n)\big], \notag
\end{align}
and for $(\wt \theta_n, g_n)$,
\begin{equation}
\label{eq:thetangnterceraversio}
\begin{aligned}
\dot{\wt \theta}_n & = \nu x_n^3 \wt \theta_n + x_{n}^3 \wt g_n + x_n^4 \OO_1(\wt \xi_{n+1})
+ x_n^3 \OO_2(x_n, \wt \xi_{n+1},\wt \theta_n, \zeta_n),\\
\dot{\tilde{g}}_n &  = \gamma_1 x_n^3 \wt \xi_{n+1} + \gamma_2 x_n^3 \wt \theta_n + x_n^3 \OO_2(\wt \xi_{n+1},\wt \theta_n)
+ \OO_6(x_n), 
\end{aligned}
\end{equation}
where, using~\eqref{asymptoticGamma_n}, 
\begin{equation}
\label{def:coeficientsdedottildegn}
\begin{aligned}
\gamma_1 & =  \frac{\alpha_n}{\alpha_{n+1}^2} m_{n+1}A^3 \wt v_{1,1}\Gamma_n = \frac{8}{M_n}(m_n+m_{n+1}+\OO_2(m_n,m_{n+1})) \frac{\wt v_{1,1}}{m_n},\\
\gamma_2 & =  \frac{1}{2\alpha_{n+1}} m_{n+1}A^2 \wt v_{0,2} \Gamma_n= \frac{4}{M_n}(m_n+m_{n+1}+\OO_2(m_n,m_{n+1})) \frac{\wt v_{0,2}}{m_n},
\end{aligned}
\end{equation}
and, finally, for $(\varphi,\wt \rho)$,
\begin{equation}
\label{eq:varphirhoterceraversio}
\begin{aligned}
\dot \varphi & = \omega+   \OO_4(x_n),\\
\dot{ \tilde{ \rho}} &  = 3\nu x_n^3 \wt \rho+ \OO(\wt \rho)\OO_6(x_n) +\OO_2(x_n,g_n)\OO_3(x_n) + \tilde{\rho} x_n^3 \OO_1(\zeta_n).
\end{aligned}
\end{equation}

To finish the proof of the proposition, the last change is simply a linear change of variables to distinguish between the contracting and the expanding variables. It only involves the variables $(\wt \xi_{n+1},\wt \eta_{n+1}, \wt \theta_n, \wt g_n)$. Denoting $Z = (\zeta_n, \wt \xi_{n+1},\wt \eta_{n+1},\wt \theta_n,\wt g_n,\wt \rho)^\top$, equations~\eqref{equacionsperxniynterceraversio}, \eqref{eq:xinmes1etanmes1terceraversio}, \eqref{eq:thetangnterceraversio} and~\eqref{eq:varphirhoterceraversio} can be written as
\begin{equation*}
\left\{
\begin{aligned}
\dot x_n & = -\nu x_n^4 + x_n^4 \OO_1(\zeta_n)+ \OO_9(x_n),\\
\dot Z & = x_n^3 \overline{\mathbf{M}} Z + x_n^3\OO_2(x_n,Z), \\
\dot \varphi & = \omega + x_n^3\OO_1(x_n,Z),
\end{aligned}
\right.
\end{equation*}
with,
\begin{equation*}
M = \begin{pmatrix}
2+\eps_{1,1} & 0 & 0 & 0 & 0 & 0 \\
0 & -1+\eps_{2,2} & -1+\eps_{2,3} & \eps_{2,4} & 0 & 0 \\
0 & -4+\eps_{3,2} & 2+\eps_{3,3} & \eps_{3,4} & 0 & 0 \\
0 & 0 & 0 & 1+\eps_{4,4} & 1 & 0 \\
0 & \gamma_1 & 0 & \gamma_2 & 0 & 0 \\
0 & 0 & 0 & 0 & 0 & (3+\eps_{6,6}) \Id
\end{pmatrix},
\end{equation*}
where, using that, by~\eqref{def:nuGamman}, $\nu = 1+\OO(m_n,m_{n+1})$, Lemma~\ref{lem:fitestermeacoblament} and~\eqref{def:coeficientsdedottildegn},
\begin{equation*}
\eps_{i,j}  = \OO(m_n,m_{n+1}).
\end{equation*}
Taking into account the definition of $\wt v_{1,1}$ and $\wt v_{0,2}$ in~\eqref{def:wtVj0} and Lemma~\ref{lem:fitestermeacoblament}, we have that
\begin{equation}
\label{midagamma1gamma2}
\gamma_i = \OO(m_n,m_{n+1}), \qquad i = 1,2,
\end{equation}
and
\begin{equation}
\label{coef:signecoefm65}
\begin{cases}
\gamma_2 <0, \qquad \text{if \ \ $\wt \theta_0 = \pi$} \\
\gamma_2 > 0,\qquad \text{if \ \ $\wt \theta_0 = \wt \theta(A,m_n,m_{n+1})$}.
\end{cases}
\end{equation}

Next, we need to diagonalize the submatrix $\mathbf{M}$. 

We notice that the most part of the matrix $\textbf{M}$ is already in diagonal form so that it is only necessary to diagonalize the submatrix
\begin{equation*}
\wt{\mathbf{M}}
=
\begin{pmatrix}
-1+\eps_{2,2} & -1+\eps_{2,3} & \eps_{2,4} & 0  \\
-4+\eps_{3,2} & 2+\eps_{3,3} & \eps_{3,4} & 0  \\
 0 & 0 & 1+\eps_{4,4} & 1  \\
 \gamma_1 & 0 & \gamma_2 & 0
\end{pmatrix} = \begin{pmatrix}
\mathbf{M}_{1,1} & \mathbf{M}_{1,2} \\
\mathbf{M}_{2,1} & \mathbf{M}_{2,2}
\end{pmatrix}, 
\end{equation*}
where $\mathbf{M}_{i,j}$ are the $2\times 2$ blocks of $\wt{\mathbf{M}}$.

We observe that the eigenvalues of $\mathbf{M}_{1,1}$ are $3+\OO(m_n,m_{n+1})$ and $-2+\OO(m_n,m_{n+1})$ and, using~\eqref{midagamma1gamma2} and~\eqref{coef:signecoefm65}, the eigenvalues of $\mathbf{M}_{2,2}$ are $1 + \OO_1(m_n,m_{n+1})$ and $-\gamma_2 +\OO_2(m_n,m_{n+1})$. The corresponding eigenvectors are, respectively,
$v_1= (1,-4)^\top+ \OO(m_n,m_{n+1})$, $v_2=(1,1)^\top+ \OO(m_n,m_{n+1})$, $v_3=(1,0)^\top+ \OO_2(m_n,m_{n+1})$ and $v_4=(1,-1)^\top+ \OO(m_n,m_{n+1})$. Let $\mathbf{B}_{1,1}$ and $\mathbf{B}_{2,2}$ be the matrices with columns $v_1$, $v_2$ and $v_3$, $v_4$, respectively, and
\[
\mathbf{B} = \begin{pmatrix}
\mathbf{B}_{1,1} & 0 \\
0 & \mathbf{B}_{2,2}
\end{pmatrix}.
\]
Clearly, the matrix
\begin{equation*}
\wh{\mathbf{M}} = \mathbf{B}^{-1} \wt{\mathbf{M}} \mathbf{B} = \begin{pmatrix}
\wh{\mathbf{M}}_{1,1} & \wh{\mathbf{M}}_{1,2} \\
\wh{\mathbf{M}}_{2,1} & \wh{\mathbf{M}}_{2,2}
\end{pmatrix}
\end{equation*}
satisfies
\begin{equation*}
\begin{aligned}
\wh{\mathbf{M}}_{1,1} &= \begin{pmatrix}
3+\OO(m_n,m_{n+1}) & 0 \\
0 & -2+\OO(m_n,m_{n+1})
\end{pmatrix}, \\ \wh{\mathbf{M}}_{2,2} &= \begin{pmatrix}
1+\OO(m_n,m_{n+1}) & 0 \\
0 & -\gamma_2 +\OO_2(m_n,m_{n+1})
\end{pmatrix}
\end{aligned}
\end{equation*}
while
\begin{equation*}
\wh{\mathbf{M}}_{1,2}, \wh{\mathbf{M}}_{2,1} = \OO(m_n,m_{n+1}).
\end{equation*}
It remains to prove that there exists
\[
\mathbf{A} = \Id_{4\times4} + \OO_2(m_n,m_{n+1})
\]
such that
\begin{equation}\label{AMA_new}
\mathbf{A}^{-1} \wh{\mathbf{M}} \mathbf{A} = \begin{pmatrix}
\wt{\mathbf{M}}_{1,1}  & 0 \\
0 & \wt{\mathbf{M}}_{2,2}
\end{pmatrix},
\end{equation}
with
\[
\wt{\mathbf{M}}_{1,1} = \wh{\mathbf{M}}_{1,1} + \OO_2(m_n,m_{n+1}), \qquad \wt{\mathbf{M}}_{2,2} = \wh{\mathbf{M}}_{2,2} + \OO_2(m_n,m_{n+1})
\]
being diagonal matrices. We notice that, taking $\wt{\mathbf{C}} = \mathbf{B} \mathbf{A}$ and
\[
\mathbf{C} = \begin{pmatrix}
\Id & 0 & 0 \\
0 & \wt{\mathbf{C}} & 0 \\
0 & 0 & \Id
\end{pmatrix}, 
\]
the proposition follows. In order to prove~\eqref{AMA_new}, we first look for $\mathbf{A}_{1,2}$ such that the matrix
\[
\wt{\mathbf{A}} = \begin{pmatrix}
\Id & \mathbf{A}_{1,2} \\
0 & \Id
\end{pmatrix}
\]
satisfies
\begin{equation}
\label{primerpasdiagonalitzacio}
\wt{\mathbf{A}}^{-1} \wh{\mathbf{M}} \wt{\mathbf{A}}= \begin{pmatrix}
\wh{\mathbf{M}}_{1,1} + \OO_2(m_n,m_{n+1}) & 0 \\
\wh{\mathbf{M}}_{2,1} & \wh{\mathbf{M}}_{2,2}+ \OO_2(m_n,m_{n+1})
\end{pmatrix}.
\end{equation}
Since
\[
\wt{\mathbf{A}}^{-1} \wh{\mathbf{M}} \wt{\mathbf{A}} = \begin{pmatrix}
\wh{\mathbf{M}}_{1,1} + \mathbf{A}_{1,2} \wh{\mathbf{M}}_{2,1} & \wh{\mathbf{M}}_{1,1} \mathbf{A}_{1,2}- \mathbf{A}_{1,2} \wh{\mathbf{M}}_{2,2} + \wh{\mathbf{M}}_{1,2} - \mathbf{A}_{1,2} \wh{\mathbf{M}}_{2,1} \mathbf{A}_{1,2} \\
\wh{\mathbf{M}}_{2,1} & \wh{\mathbf{M}}_{2,2}+ \wh{\mathbf{M}}_{2,1} \mathbf{A}_{1,2}
\end{pmatrix},
\]
equation \eqref{primerpasdiagonalitzacio} is equivalent to find a solution $\mathbf{A}_{1,2} = \OO(m_n,m_{n+1})$ of
\begin{equation}
\label{def:equacioperA21}
\LL \mathbf{A}_{1,2} = -\wh{\mathbf{M}}_{1,2} + \mathbf{A}_{1,2} \wh{\mathbf{M}}_{2,1} \mathbf{A}_{1,2},
\end{equation}
where
\begin{equation*}
\LL \mathbf{A}_{1,2} = \wh{\mathbf{M}}_{1,1} \mathbf{A}_{1,2}- \mathbf{A}_{1,2} \wh{\mathbf{M}}_{2,2}.
\end{equation*}

One can easily check that $\LL$ is invertible and then
we rewrite equation~\eqref{def:equacioperA21} as the fixed point equation
\[
\mathbf{A}_{1,2} = \F \mathbf{A}_{1,2}  := - \L^{-1} \wh M_{1,2} + \L^{-1} \mathbf{A}_{1,2} \wh{\mathbf{M}}_{2,1} \mathbf{A}_{1,2}.
\]
We have that $\|\F 0\| = \|\LL^{-1} \wh M_{1,2}\| \le \|\LL^{-1}\| \|\wh{\mathbf{M}}_{1,2}\| = \OO(m_n,m_{n+1})$. Defining $\rho = 2 \|\F 0\|$, $\|\F \mathbf{A}_{1,2} - \F \wt{\mathbf{A}}_{1,2} \| \le 2\rho \|\L^{-1}\| \|\mathbf{A}_{2,1} - \wt{\mathbf{A}}_{2,1} \|$ if $\mathbf{A}_{1,2}, \wt{\mathbf{A}}_{1,2}$ satisfy $\|\mathbf{A}_{1,2}\|,\|\wt{\mathbf{A}}_{1,2}\| \le \rho$. Consequently, $\F$ is a contraction in the ball of radius $\rho$, if $m_n$ and $m_{n+1}$ are small enough, which proves the existence of $\mathbf{A}_{1,2} = \OO(m_n,m_{n+1})$.

Next, let
\[
\wt{\mathbf{B}} = \begin{pmatrix} \Id + \OO_2(m_n,m_{n+1}) & 0 \\
0 & \Id + \OO_2(m_n,m_{n+1})
\end{pmatrix}
\]
such that the diagonal blocks of
\[
\wt{\mathbf{B}}^{-1} \wt{\mathbf{A}}^{-1} \wh{\mathbf{M}} \wt{\mathbf{A}} \wt{\mathbf{B}} = \begin{pmatrix}
\mathbf{N}_{1,1} & 0 \\
\mathbf{N}_{2,1} & \mathbf{N}_{2,2}
\end{pmatrix}
\]
are in diagonal form. Such matrix $\wt{\mathbf{B}}$ exists because the diagonal blocks of $\mathbf{A}^{-1} \wh{\mathbf{M}} \mathbf{A}$ are already in diagonal form
up to errors of size $\OO_2(m_n,m_{n+1})$. We observe that
\[
\mathbf{N}_{1,1} = \wh{\mathbf{M}}_{1,1}+ \OO_2(m_n,m_{n+1}), \quad \mathbf{N}_{2,2} = \wh{\mathbf{M}}_{2,2}+ \OO_2(m_n,m_{n+1}), \quad \mathbf{N}_{2,1} = \wh{\mathbf{M}}_{2,1} +  \OO_2(m_n,m_{n+1}).
\]

Next, let $\mathbf{A}_{2,1}$ be such that
\[
\mathbf{N}_{2,2} \mathbf{A}_{2,1}- \mathbf{A}_{2,1} \mathbf{N}_{1,1} = - \mathbf{N}_{2,1}.
\]
Such matrix exists, since, as the operator $\LL$ above, the operator $\mathbf{A}_{2,1} \mapsto \mathbf{N}_{2,2} \mathbf{A}_{2,1}- \mathbf{A}_{2,1} \mathbf{N}_{1,1}$ is invertible. Let
\[
\wh{\mathbf{A}} = \begin{pmatrix}
\Id & 0 \\
\mathbf{A}_{2,1} & \Id
\end{pmatrix}.
\]
It is immediate to check that $\wh{\mathbf{A}}^{-1} \wt{\mathbf{B}}^{-1} \wt{\mathbf{A}}^{-1} \wt M \wt{\mathbf{A}}\wt{\mathbf{B}} \wh{\mathbf{A}}$ is block diagonal and, in fact, diagonal provided $\mathbf{N}_{1,1}, \mathbf{N}_{2,2}$ are diagonal matrices.
\end{proof}

\subsection{Applying Theorems~\ref{thm:approximationflows} and~\ref{th:existenceflow}. Collinear case}
\label{sec:provateoremacascolinial}
We need to distinguish the cases $\wt \theta_0 = \pi$ and $\wt \theta_0 = \theta_0 (A,m_n)$ since the corresponding stable invariant manifolds have different dimension (see Theorem~\ref{thm:movimentsdoblementparabolicsv2}). In this section we consider the case $\wt \theta_0 = \pi$, that corresponds to the collinear configuration. In this case, the constant $\gamma_2$ in the matrix $\mathbf{M}$ in~\eqref{eq:sistemacompletredressat} is negative.  
Following the notation of Section~\ref{sec:setupmap}, we introduce $x = (x_n,\wh \eta_{n+1})^\top$, $y = (\zeta_n, \wh \xi_{n+1},\wh \chi_n, \wh \upsilon_n,\wt \rho)^\top$ with $x \in \R^2$, $y\in \R^{4+2(n-1)}$ and $\varphi \in \T^{2(n-1)}$. Then, equations~\eqref{eq:sistemacompletredressat} become
\begin{equation}
\label{eq:sistemaperaplicarteorema1}
\begin{aligned}
\dot x & = f(x,y) + \OO_{5}(x,y),\\
\dot y & = g(x,y) + \OO_{5}(x,y), \\
\dot \varphi & = \omega + \OO_{4}(x,y),
\end{aligned}
\end{equation}
where
\begin{equation}
\label{figcolineal}
\begin{aligned}
f(x,y) & =  x_n^{3} S x, \\
g(x,y) & =  x_n^{3} U y,
\end{aligned}
\end{equation}
and
\begin{equation}
\label{SiUcolineal}
S = \begin{pmatrix}
-\lambda_1 & 0  \\
0 & -\lambda_2
\end{pmatrix}, \qquad
U =
\begin{pmatrix}
\wt \lambda_1 & 0 & 0 & 0 & 0 \\
0 & \wt \lambda_2 & 0 & 0 & 0 \\
0 & 0 & \wt \lambda_3 & 0 & 0 \\
0 & 0 & 0 &  \wt \lambda_4 & 0 \\
0 & 0 & 0 & 0 & \wt \lambda_5 \Id
\end{pmatrix}
\end{equation}
with
\begin{equation}
\label{def:lambdai1}
\begin{aligned}
\lambda_1 & = \nu, & \quad \lambda_2 & = 2 + \eps_{2,2}, &  & & \\
\wt \lambda_1 & = 2 + \wt \eps_{1,1}, & \quad \wt \lambda_2 & = 3 + \wt \eps_{2,2}, & \quad \wt \lambda_3 & = 1+\wt \eps_{3,3},
& \quad \wt \lambda_4 = -\gamma_2 +\wt \eps_{4,4}, 
& \quad \wt \lambda_5 = 1+\wt \eps_{5,5},
\end{aligned}
\end{equation}
and
\[
\eps_{i,i} = \OO(m_n,m_{n+1}),  \quad \tilde\eps_{j,j} = \OO(m_n,m_{n+1}), \; j \neq 4 , \quad \wt \eps_{4,4} = \OO_2(m_n,m_{n+1}).
\]

For $\delta, \kappa >0$, we introduce the cone in $\R^2$
\begin{equation*}
V_{\delta,\kappa} = \{x = (\wh x_n,\wh \eta_{n+1}) \in \R^2 \mid \; 0 < \wh x_n < \delta, \;
|\wh \eta_{n+1}| \le \kappa \wh x_n\}.
\end{equation*}
For all $x = (\wh x_n,\wh \eta_{n+1}) \in V_{\delta,\kappa}$ we have that
\begin{equation}
\label{propietatdecon1}
\wh x_n ,|\wh \eta_{n+1}| \le \|x\| \le (1+\kappa^2)^{1/2} \wh x_n,
\end{equation}
where $\|\cdot\|$ denotes the standard Euclidean norm in $\R^2$.

The next proposition guarantees that we can apply Theorems~\ref{thm:approximationflows} and~\ref{th:existenceflow} to Equation~\eqref{eq:sistemaperaplicarteorema1}.

\begin{proposition}
\label{prop:cascolinial}
The vector field corresponding to equation~\eqref{eq:sistemaperaplicarteorema1} has the form~\eqref{systemflow} with $N= M = P = 4$.
If $m_n$, $m_{n+1}$ are small enough, for $\delta$ small enough, it satisfies hypothesis~\ref{hipotesisv} in Section~\ref{sec:setupmap} in the domain $V_{\delta,\kappa}$ with
\[
a_V = \frac{1}{(1+\kappa^2)^{1/2}} \min \left\{\frac{1}{(1+\kappa^2)^{1/2}} \left(1+\OO(m_n,m_{n+1})  \right),  1+\OO(m_n,m_{n+1}) \right\} >0.
\]
The constants $a_f$, in~\eqref{defa}, $b_f, A_f$, in~\eqref{defbCf} and $B_g$, in~\eqref{defABg}, in the domain $V_{\delta,\kappa}$
have the following values:
\[
\begin{aligned}
a_f & \ge  \nu+ \OO(\delta^{3},\kappa^2), & \qquad b_f & \le  1+\OO(m_n,M_{n+1})+\OO(\kappa^2),\\
A_f & \ge  2+\OO(m_n,m_{n+1}) + \OO(\kappa)+ \OO(\wh \delta^{3},\kappa^2)  , & \qquad B_g  & \ge  -\gamma_2+\OO_2(m_n,m_{n+1}).
\end{aligned}
\]
Hence, if $m_n$, $m_{n+1}$ are small enough so that $-\gamma_2 + \OO_2(m_n,m_{n+1}) >0$, then, for $\kappa$ and $\delta$ small enough,
\[
a_f >0, \qquad A_f > b_f \max \{ 1, N-P\}, \qquad B_g >0.
\]
Consequently, Equation~\eqref{eq:sistemaperaplicarteorema1} satisfies the hypotheses of Theorems~\ref{thm:approximationflows} and~\ref{th:existenceflow}. The origin possesses a $2+2(n-1)$ analytic stable invariant manifold.
\end{proposition}

\begin{proof}
We will use the standard Euclidean norm and its induced matrix norm to compute all the constants.
We start by computing $a_V$. Clearly, if $(a,b) \in V_{\delta, \kappa}$,
\[
d((a,b), V_{\delta,\kappa}^c) = \min\left\{ \frac{1}{(1+\kappa^2)^{1/2}}\left(\kappa a - |b|\right), \delta-a \right\}.
\]
Then, if $x\in V_{\kappa,\delta}$, denoting $x^* = (\wh x_n^*,\wh \eta_{n+1}^*) = x+ f(x,0)$, since
\[
| \wh \eta_{n+1}^* | = |\wh \eta_{n+1}|(1-\lambda_2 \wh x_n^{3}) < \kappa \wh x_n (1-\lambda_2 \wh x_n^{3}),
\]
we have that, for $\delta$ small enough,
\[
\begin{aligned}
\frac{1}{(1+\kappa^2)^{1/2}}\left(\kappa \wh x_n^* - |\wh \eta_{n+1}^*|\right) & = \frac{1}{(1+\kappa^2)^{1/2}}\left(\kappa \left(\wh x_n - \lambda_1 \wh x_n^{4}\right) - |\wh \eta_{n+1}^2(1-\lambda_2\wh x_n^{3})|\right) \\
& \ge \frac{1}{(1+\kappa^2)^{1/2}}\left(\kappa \left(\wh x_n - \lambda_1 \wh x_n^{4}\right) - \kappa x_n |1-\lambda_2\wh x_n^{3}|\right) \\
& = \frac{1}{(1+\kappa^2)^{1/2}} \left(\lambda_2-\lambda_1\right) \wh x_n^{4}.
\end{aligned}
\]
Also, for $x\in V_{\kappa,\delta}$,
\[
\delta - \wh x_n^* = \delta - \wh x_n + \lambda_1 \wh x_n^{4} \ge \lambda_1 \wh x_n^{4}.
\]
Hence, using~\eqref{propietatdecon1}, for $x\in V_{\kappa,\delta}$, we have that
\[
d(x^*,V_{\kappa,\delta}^c) \ge \frac{1}{(1+\kappa^2)^{1/2}} \min \left\{\frac{1}{(1+\kappa^2)^{1/2}} \left(\lambda_2-\lambda_1  \right),  \lambda_1 \right\} \|x\|^{4}.
\]
The claim for $a_V$ follows 
combining this last inequality with~\eqref{def:lambdai1} and taking into account that $\nu = 1 + \OO(m_n,m_{n+1})$, .

Now we compute $a_f$. Using~\eqref{figcolineal}, \eqref{SiUcolineal} and~\eqref{propietatdecon1}, since 
\[
\begin{aligned}
\|x+ f(x,0)\| & = \sqrt{\wh x_n^2 (1-\lambda_1\wh x_n^{3})^2 + \wh \eta_{n+1}^2 (1-\lambda_2 \wh x_n^{3})^2} \\
& = \|x\|\left( 1- 2\lambda_1 \frac{\wh x_n^{5}}{\|x\|^2} - 2\lambda_2 \frac{\wh x_n^{3}\wh \eta_{n+1}^2}{\|x\|^2} + \OO(\|x\|^{6})\right)^{1/2} \\
& \le \|x\| - (\lambda_1 + \lambda_2 \kappa^2 )\|x\|^{4} + \OO(\|x\|^{7})
\end{aligned}
\]
we have that
\[
a_f=-\sup_{x\in V_{\delta,\kappa} }\frac{\|x+ f(x,0)\| - \|x\|}{\|x\|^{4}} \ge \frac{\lambda_1+ \lambda_2 \kappa^2  + \OO(\delta^{3})}{(1+\kappa^2)^{3/2}}.
\]
By~\eqref{def:lambdai}, the claim follows.

Next, we compute $b_f$. Since, in view of~\eqref{figcolineal}, \eqref{SiUcolineal} and~\eqref{propietatdecon1},
\[
\|f(x,0)\| = \wh x_n^{3} \|S x\| \le x_n^{4}\sqrt{\lambda_1^2+\kappa^2\gamma_2^2 }
\]
we have that, using~\eqref{def:lambdai},
\[
b_f= \sup_{x\in V_{\delta,\kappa}}\frac{\|f(x,0)\|}{\|x\|^{4}} \le \sqrt{\nu+ 4\kappa^2(4+\OO(m_n,m_{n+}))}.
\]
The claim on $b_f$ follows then from~\eqref{def:lambdai1}.

Now we compute
\[
A_f=-\sup_{x\in V_{\delta,\kappa}} \frac{\|\Id+ D_x f(x,0)\| - 1}{\|x\|^{3}}.
\]
We bound the spectral radius of $(\Id+ D_x f(x,0))^\top(\Id+ D_x f(x,0))$. Since
\[
\Id+ D_x f(x,0) = \begin{pmatrix} 1-4 \lambda_1 \wh x_n^{3} & 0  \\
-3  \lambda_2 \wh x_n^{2} \wh \eta_{n+1} & 1 -  \lambda_2 \wh x_n^{3}
\end{pmatrix},
\]
we have that
\[
(\Id+ D_x f(x,0))^\top (\Id+ D_x f(x,0))  =
\begin{pmatrix}
1-8 \lambda_1 \wh x_n^{3} + \OO(\wh x_n^{6})& -(1-\lambda_2 \wh x_n^{3}) 3  \lambda_2 \wh x_n^{2} \wh \eta_{n+1}  \\
 -(1-\lambda_2 \wh x_n^{3}) 3  \lambda_2 \wh x_n^{2} \wh \eta_{n+1} &
 1-2 \lambda_2 \wh x_n^{3} + \OO(\wh x_n^{6})
\end{pmatrix}.
\]
Hence,  since ~\eqref{def:lambdai1} implies that
\[
8 \lambda_1 = 8 + \OO(m_n,m_{n+1}), \qquad 2\lambda_2 = 4 + \OO(m_n,m_{n+1}),
\]
applying Gershgorin circle theorem,
\[
\|\Id+ D_x f(x,0)\| \le 1- \left( 2+\OO(m_n,m_{n+1}) + \OO(\kappa)+ \OO(\wh x_n^{3})\right)\wh x_n^{3}.
\]
Hence,
\[
A_f\ge  \frac{2+\OO(m_n,m_{n+1}) + \OO(\kappa)+ \OO(\wh \delta^{3})}{(1+\kappa^2)^{3/2}}.
\]
We finally compute
\[
B_g=-\sup_{x\in V_{\kappa,\delta}} \frac{\|\Id- D_y g(x,0)\| - 1}{\|x\|^{3}}.
\]
By~\eqref{figcolineal} and~\eqref{SiUcolineal} it follows that $D_y g(x,0) = \wh x_n^{3} U$. Then, using~\eqref{def:lambdai1} we get
\[
\|\Id - D_y g(x,0)\| \le 1- (-\gamma_2+\OO_2(m_n,m_{n+1}))\wh x_n^{3},
\]
from which the claim for the stable manifold follows. In order to obtain the unstable one we apply the same procedure to the time reversed system.
\end{proof}

\subsection{Applying Theorems~\ref{thm:approximationflows} and~\ref{th:existenceflow}. Equilateral case}

Now we deal with the case $\wt \theta_0 = \theta_0 (A,m_n) = \pi/3+\OO(m_n,m_{n+1})$. Unlike the previous one, we will see that the invariant manifolds are $3+2(n-1)$-dimensional, because in this case $\wh \upsilon_n$ is a ``stable'' direction.

However, since $\wh \upsilon_n$ is very slow, it is easy to check that equation~\eqref{eq:sistemacompletredressat} does not readily satisfy the hypotheses in Theorems~\ref{thm:approximationflows} and~\ref{th:existenceflow}. To apply these theorems, we introduce a new set of variables in the next proposition.
We recall that $\gamma_2 = \OO(m_n,m_{n+1})$ and $\nu = 1 +\OO(m_n,m_{n+1})$.

\begin{proposition}
Let $m_n$, $m_{n+1} >  0$ be fixed but small enough. Take $\wt \theta_0 = \theta_0 (A,m_n)$ in equation~\eqref{eq:sistemacompletredressat}, that corresponds to $\gamma_2 >0$. Let $\ell\in \N$ and define $\wh x_n$ through $x_n = \wh x_n^\ell$, while maintaining the other variables the same. Equation~\eqref{eq:sistemacompletredressat} becomes
\begin{equation}
\label{eq:sistemacompletredressatcasequilater}
\left\{
\begin{aligned}
\dot{\wh x}_n & = -\frac{\nu}{\ell} \wh x_n^{3\ell+1} + \OO_{8\ell+1}(\wh x_n),\\
\dot{\wt Z} & = \wh x_n^{3\ell} C^{-1} M C \wt Z + \wh x_n^{3\ell} \OO_2(\wh x_n^{\ell},\wt Z), \\
\dot \varphi & = \omega + \wh x_n^{3\ell}\OO_1(\wh x_n^{\ell}, \wt Z).
\end{aligned}
\right.
\end{equation}
\end{proposition}

\begin{proof} It is a straightforward computation. Indeed, using~\eqref{eq:sistemacompletredressat},
\[
\dot{\wh x}_n = \frac{1}{\ell \wh x^{\ell-1}} \dot x_n = \frac{1}{\ell \wh x^{\ell-1}} \left( -\nu \wh x_n^{4 \ell} + \OO_{9\ell}(\wh x_n)\right),
\]
from which the claim follows immediately.
\end{proof}

\begin{remark}
Later, in Proposition~\ref{prop:casequilater}, we will fix $\ell \ge 1$ such that $\frac{\nu}{\ell} < \gamma_2$. Since $\nu = 1+\OO(m_n,m_{n+1})$ and $\gamma_2 = \OO(m_n,m_{n+1})$, $\ell$ will be large
but fixed.
\end{remark}

We use the same notation as in Section~\ref{sec:provateoremacascolinial}. We introduce $x = (\wh x_n,\wh \eta_{n+1}, \wh \upsilon_n)^\top$, $y = (\zeta_n, \wh \xi_{n+1},\wh \chi_n,\wt \rho)^\top$, that is, $x \in \R^3$, $y\in \R^{3+2(n-1)}$ and $\varphi \in \T^{2(n-1)}$. Then, equation~\eqref{eq:sistemacompletredressatcasequilater} becomes
\begin{equation}
\label{eq:sistemaperaplicarteorema}
\begin{aligned}
\dot x & = f(x,y) + \OO_{3\ell+2}(x,y),\\
\dot y & = g(x,y) + \OO_{3\ell+2}(x,y), \\
\dot \varphi & = \omega + \OO_{3\ell+1}(x,y),
\end{aligned}
\end{equation}
where
\begin{equation}
\label{def:figdefinitives}
\begin{aligned}
f(x,y) & = \wh x_n^{3\ell} S x, \\
g(x,y) & = \wh x_n^{3\ell} U y
\end{aligned}
\end{equation}
and
\begin{equation}
\label{def:matriusSU}
S = \begin{pmatrix}
-\lambda_1 & 0 & 0 \\
0 & -\lambda_2 & 0 \\
0 & 0 & -\lambda_3
\end{pmatrix}, \qquad
U =
\begin{pmatrix}
\wt \lambda_1 & 0 & 0 & 0 \\
0 & \wt \lambda_2 & 0 & 0 \\
0 & 0 & \wt \lambda_3 & 0 \\
0 & 0 & 0 &  \wt \lambda_4 \Id
\end{pmatrix}
\end{equation}
with
\begin{equation}
\label{def:lambdai}
\begin{aligned}
\lambda_1 & = \frac{\nu}{\ell}, & \quad \lambda_2 & = 2 + \eps_{2,2}, & \quad \lambda_3 & = \gamma_2 + \eps_{3,3}, & & \\
\wt \lambda_1 & = 2 + \wt \eps_{1,1}, & \quad \wt \lambda_2 & = 3 + \wt \eps_{2,2}, & \quad \wt \lambda_3 & = 1+\wt \eps_{3,3},
, & \quad \wt \lambda_4 = 1+\wt \eps_{4,4},
\end{aligned}
\end{equation}
and
\[
\eps_{i,i} = \OO(m_n,m_{n+1}), \quad i \neq 3, \quad \eps_{3,3} = \OO_2(m_n,m_{n+1}), \quad \tilde\eps_{j,j} = \OO(m_n,m_{n+1}).
\]

For $\delta, \kappa >0$, we introduce the following cone in $\R^3$
\begin{equation*}
V_{\delta,\kappa} = \{x = (\wh x_n,\wh \eta_{n+1}, \wh \upsilon_n) \in \R^3 \mid \; 0 < \wh x_n < \delta, \;
\wh \eta_{n+1}^2+ \wh \upsilon_n^2 \le \kappa^2 \wh x_n^2\}.
\end{equation*}
For all $x = (\wh x_n,\wh \eta_{n+1}, \wh \upsilon_n) \in V_{\delta,\kappa}$ we have that
\begin{equation}
\label{propietatdecon}
\wh x_n ,|\wh \eta_{n+1}|, |\wh \upsilon_n| \le \|x\| \le (1+\kappa^2)^{1/2} \wh x_n,
\end{equation}
where $\|\cdot\|$ denotes the standard Euclidean norm in $\R^3$.

Next proposition is analogous to Proposition~\ref{prop:cascolinial} in this case.

\begin{proposition}
\label{prop:casequilater}
The vector field corresponding to equation~\eqref{eq:sistemaperaplicarteorema} has the form~\eqref{systemflow} with $N= M = P = 3\ell+1$.
If $m_n$, $m_{n+1}$ are small, choosing $\ell$ large enough, for $\delta$ small, hypothesis~\ref{hipotesisv} in Section~\ref{sec:setupmap}  is satisfied in 
the domain $V_{\delta,\kappa}$ with
\[
a_V = \frac{1}{(1+\kappa^2)^{1/2}} \min \left\{\frac{1}{(1+\kappa^2)^{1/2}} \left(2+\gamma_2-\frac{\nu}{\ell}+\OO(m_n,m_{n+1}) + \OO(\delta^{3\ell}) \right),  \lambda_1 \right\} >0.
\]
For the constants $a_f$, in~\eqref{defa}, $b_f, A_f$, in~\eqref{defbCf} and $B_g$, in~\eqref{defABg}, in the domain $V_{\delta,\kappa}$
have the following estimates:
\[
\begin{aligned}
a_f & \ge  \frac{\nu}{\ell}+ \OO(\delta^{3\ell},\kappa^2), & \qquad b_f & \le  \sqrt{\frac{\nu^2}{\ell^2}+\OO(\kappa^2)},\\
A_f & \ge  \gamma_2+\OO_2(m_n,m_{n+1}) + \OO(\kappa)+ \OO(\wh \delta^{3\ell},\kappa^2)  , & \qquad B_g  & \ge 1+\OO(m_n,m_{n+1}).
\end{aligned}
\]
Hence, if $m_n$, $m_{n+1}$ are small enough such that $\gamma_2 + \OO_2(m_n,m_{n+1}) >0$, taking $\ell$ sufficiently large so that $\nu/\ell < \gamma_2 + \OO_2(m_n,m_{n+1})$, then, for $\kappa$ and $\delta$ small,
\[
a_f <0, \qquad A_f > b_f \max \{ 1, N-P\}, \qquad B_g >0.
\]
Consequently, equation~\eqref{eq:sistemaperaplicarteorema} satisfies the hypotheses of Theorems~\ref{thm:approximationflows} and~\ref{th:existenceflow}. The origin possesses a $3+2(n-1)$ analytic stable invariant manifold.
\end{proposition}

\begin{proof}
 We assume $m_n$ and $m_{n+1}$ small enough so that $\lambda_3 =  \gamma_2 +\OO_2(m_n,m_{n+1})>0$ and choose $\ell$ such that $ \lambda_1 = \nu/\ell = (1+\OO(m_n,m_{n+1}))/\ell < \lambda_3$.

We will use the standard Euclidean norm and its induced matrix norm to compute all the constants.
We start by computing $a_V$. Clearly, if $(a,b,c) \in V_{\delta, \kappa}$,
\[
d((a,b,c), V_{\delta,\kappa}^c) = \min\left\{ \frac{1}{(1+\kappa^2)^{1/2}}\left(\kappa a - \sqrt{b^2+c^2}\right), \delta-a \right\}.
\]
Then, if $x\in V_{\kappa,\delta}$, denoting $x^* = (\wh x_n^*,\wh \eta_{n+1}^*, \wh \upsilon_n^*) = x+ f(x,0)$, since
\[
\begin{aligned}
| \wh \eta_{n+1}^* | &= |\wh \eta_{n+1}|(1-\lambda_2 \wh x_n^{3\ell}) < \kappa \wh x_n (1-\lambda_2 \wh x_n^{3\ell}), \\
| \wh \upsilon_n^* | &= |\wh \upsilon_n|(1-\lambda_3 \wh x_n^{3\ell}) < \kappa \wh x_n (1-\lambda_3 \wh x_n^{3\ell}),
\end{aligned}
\]
we have that
\begin{multline*}
\frac{1}{(1+\kappa^2)^{1/2}}\left(\kappa \wh x_n^* - \sqrt{(\wh \eta_{n+1}^*)^2+(\wh \upsilon_n^*)^2}\right)\\
\begin{aligned}
& = \frac{1}{(1+\kappa^2)^{1/2}}\left(\kappa \left(\wh x_n - \frac{\nu}{\ell} \wh x_n^{3\ell+1}\right) - \sqrt{\wh \eta_{n+1}^2(1-\lambda_2\wh x_n^{3\ell})^2+\wh \upsilon_n^2(1-\wt \lambda_3\wh x_n^{3\ell})^2}\right) \\
& \ge \frac{1}{(1+\kappa^2)^{1/2}}\left(\kappa \left(\wh x_n - \lambda_1 \wh x_n^{3\ell+1}\right) - \kappa x_n \sqrt{(1-\lambda_2\wh x_n^{3\ell})^2+(1-\lambda_3 \wh x_n^{3\ell})^2}\right) \\
& = \frac{1}{(1+\kappa^2)^{1/2}} \left(\lambda_2+\lambda_3-\lambda_1 + \OO(\delta^{3\ell}) \right) \wh x_n^{3\ell+1}.
\end{aligned}
\end{multline*}
Also, for $x\in V_{\kappa,\delta}$,
\[
\delta - \wh x_n^* = \delta - \wh x_n + \lambda_1 \wh x_n^{3\ell+1} \ge \lambda_1 \wh x_n^{3\ell+1}.
\]
Hence, using~\eqref{propietatdecon}, for $x\in V_{\kappa,\delta}$,
\[
d(x^*,V_{\kappa,\delta}^c) \ge \frac{1}{(1+\kappa^2)^{1/2}} \min \left\{\frac{1}{(1+\kappa^2)^{1/2}} \left(2+\wt \gamma_2 - \frac{\nu}{\ell} + \OO(\delta^{3\ell}) \right),  \frac{\nu}{\ell} \right\} \|x\|^{3\ell+1}.
\]
Now we compute $a_f$. Using~\eqref{def:figdefinitives}, \eqref{def:matriusSU} and~\eqref{propietatdecon}, we have
\[
\begin{aligned}
\|x+ f(x,0)\| & = \sqrt{\wh x_n^2 (1-\lambda_1\wh x_n^{3\ell})^2 + \wh \eta_{n+1}^2 (1-\lambda_2 \wh x_n^{3\ell})^2+ \wh \upsilon_n^2 (1-\lambda_3 \wh x_n^{3\ell})^2} \\
& = \|x\|\left( 1- 2\lambda_1 \frac{\wh x_n^{3\ell+2}}{\|x\|^2} - 2\lambda_2 \frac{\wh x_n^{3\ell}\wh \eta_{n+1}^2}{\|x\|^2} -
2\lambda_3 \frac{\wh x_n^{3\ell}\wh \upsilon_{n+1}^2}{\|x\|^2}+ \OO(\|x\|^{6\ell})\right)^{1/2} \\
& \le \|x\| - (\lambda_1 + (\lambda_2+\lambda_3) \kappa^2 )\|x\|^{3\ell+1} + \OO(\|x\|^{6\ell+1})
\end{aligned}
\]
and
\[
a_f=-\sup_{x\in V_{\delta,\kappa} }\frac{\|x+ f(x,0)\| - \|x\|}{\|x\|^{3\ell+1}} \ge \frac{(\lambda_1 + (\lambda_2+\lambda_3) \kappa^2 ) + \OO(\delta^{3\ell})}{(1+\kappa^2)^{3\ell/2}}.
\]
By~\eqref{def:lambdai}, the claim follows.

Next, we compute $b_f$. Since, in view of~\eqref{def:figdefinitives}, \eqref{def:matriusSU} and~\eqref{propietatdecon}
\[
\|f(x,0)\| = \wh x_n^{3\ell} \|S x\| \le x_n^{3\ell+1}\sqrt{\lambda_1^2+\kappa^2(\gamma_2^2 + \gamma_3^2)}
\]
Using~\eqref{def:lambdai} we obtain
\[
b_f= \sup_{x\in V_{\delta,\kappa}}\frac{\|f(x,0)\|}{\|x\|^{3\ell+1}} \le \sqrt{\frac{\nu^2}{\ell^2}+\kappa^2(4 + \gamma_2^2+\OO(m_n,m_{n+1}))}.
\]
Now we compute
\[
A_f=-\sup_{x\in V_{\delta,\kappa}} \frac{\|\Id+ D_x f(x,0)\| - 1}{\|x\|^{3\ell}}.
\]
We bound the spectral radius of $(\Id+ D_x f(x,0))^\top(\Id+ D_x f(x,0))$. Since
\[
\Id+ D_x f(x,0) = \begin{pmatrix} 1-(3\ell+1) \lambda_1 \wh x_n^{3\ell} & 0 & 0 \\
-3 \ell \lambda_2 \wh x_n^{3\ell-1} \wh \eta_{n+1} & 1 -  \lambda_2 \wh x_n^{3\ell} & 0 \\
-3 \ell \lambda_2 \wh x_n^{3\ell-1}\wh \upsilon_n & 0 & 1- \lambda_3 \wh x_n^{3\ell}
\end{pmatrix},
\]
we have that
\begin{multline*}
(\Id+ D_x f(x,0))^\top (\Id+ D_x f(x,0)) \\ =
\begin{pmatrix}
1-2 (3\ell+1) \lambda_1 \wh x_n^{3\ell} + \OO(\wh x_n^{6\ell})& -(1-\lambda_2 \wh x_n^{3\ell}) 3 \ell \lambda_2 \wh x_n^{3\ell-1} \wh \eta_{n+1} &
-(1-\lambda_3 \wh x_n^{3\ell}) 3 \ell \lambda_3 \wh x_n^{3\ell-1} \wh \upsilon_{n} \\
 -(1-\lambda_2 \wh x_n^{3\ell}) 3 \ell \lambda_2 \wh x_n^{3\ell-1} \wh \eta_{n+1} &
 1-2 \lambda_2 \wh x_n^{3\ell} + \OO(\wh x_n^{6\ell}) & 0 \\
 -(1-\lambda_3 \wh x_n^{3\ell}) 3 \ell \lambda_3 \wh x_n^{3\ell-1} \wh \upsilon_{n} &  0 &
 1-2 \lambda_3 \wh x_n^{3\ell} + \OO(\wh x_n^{6\ell})
\end{pmatrix}.
\end{multline*}
Hence,  since ~\eqref{def:lambdai} implies that
\[
2 (3\ell+1) \lambda_1 > 6 + \OO(m_n,m_{n+1}), \qquad 2\lambda_2 > 4 + \OO(m_n,m_{n+1}), \qquad 2 \lambda_2 = 2 \gamma_2 + \OO_2(m_n,m_{n+1}),
\]
applying Gershgorin circle theorem,
\[
\|\Id+ D_x f(x,0)\| \le 1- \left( \gamma_2+\OO_2(m_n,m_{n+1}) + \OO(\kappa)+ \OO(\wh x_n^{3\ell})\right)\wh x_n^{3\ell}.
\]
Therefore,
\[
A_f\ge  \frac{\gamma_2+\OO_2(m_n,m_{n+1}) + \OO(\kappa)+ \OO(\wh \delta^{3\ell})}{(1+\kappa^2)^{3\ell/2}}.
\]
We finally compute
\[
B_g=-\sup_{x\in V_{\kappa,\delta}} \frac{\|\Id- D_y g(x,0)\| - 1}{\|x\|^{3\ell}}.
\]
By~\eqref{def:figdefinitives} and~\eqref{def:matriusSU} we have $D_y g(x,0) = \wh x_n^{3\ell} U$. By~\eqref{def:lambdai}, this implies
\[
\|\Id - D_y g(x,0)\| \le 1- (1+\OO(m_n,m_{n+1})\wh x_n^{3\ell},
\]
from which the claim  for the stable manifold follows. As in the collinear case, in order to obtain the unstable one it is only necessary to apply the same procedure to the time reversed system.
\end{proof}


\section{Acknowledgements}

I.B. has been partially supported by the grant PID-2021-122954NB-100, E.F. has been partially supported by the grant PID2021-125535NB-I00, and P.M. has been partially supported by the grant PID2021-123968NB-I00, funded by the
Spanish State Research Agency through the programs MCIN/AEI/10.13039/501100011033
and “ERDF A way of making Europe”. 

Also, all authors have been partially supported by the Spanish State Research Agency, through the Severo Ochoa and Mar\'ia de Maeztu Program for Centers and Units of Excellence in R\&D (CEX2020-001084-M).

\appendix 


\section{Proof of Remark~\ref{prop:formsystem}} \label{app:proofprop}

As in the rest of this work, we do not write the dependence of the different objects with respect to the parameter $\lambda$. 

Assume that a map $\mathcal{F}$ as in the remark satisfies conditions (i)-(iii) and has an invariant manifold tangent to $\{y=0\}$ represented as $y=\mathcal{K}(x,\th)$. It is clear that if $M>N$, we can take $\sum_{j=N}^{M-1} g^{ j}(x,y,\th)=0$, hence   (iv) is satisfied and we are done. 

Now we consider the case $M\leq N$. By Lemma~\ref{lem:Ftilde}, we can remove the dependence on $\th$ of the map $\F$ up to  order $N$. Let $\overline{f}^{\geq N}_*(x,y)$, $\overline{g}^{\geq M}_*(x,y)$ 
and $\overline{h}^{\geq P}_*(x,y)$ be the terms of degree less or equal than $N$ in each component of  $\mathcal{F}-\Id$, respectively, after the dependence on $\th$ has been removed. 
The invariance condition for $\mathcal{K}(x,\th) $ reads 
$$
\mathcal{K} (x,\th ) + \overline{g}_*^{\geq M}(x,\mathcal{K}(x,\th))    =
\mathcal{K} \big(x+ \overline{f}_*^{\geq N} (x,\mathcal{K}(x,\th)),\th + \omega + \overline{h}_*^{\geq P}(x,\mathcal{K} (x,\th))\big ) + \OO(\|x\|^{N+1}).
$$
Differentiating with respect to $\th$ and writing 
$$
\F_{x,\th}^*=\big(x+ \overline{f}_*^{\geq N} (x,\mathcal{K}(x,\th)),\th + \omega + \overline{h}_*^{\geq P}(x,\mathcal{K} (x,\th))\big )
$$
we have
\begin{align}
\partial_\th \mathcal{K}(x,\th) -& \partial_{\th}\mathcal{K} (x,\th + \omega) \notag\\  = & - \partial_y\overline{g}_*^{\geq M}(x,\mathcal{K}(x,\th)) \partial_\th \mathcal{K}(x,\th)+
\partial_x \mathcal{K}(\F^*_{x,\th}) 
\partial_y\overline{f}_*^{\geq N}(x,\KK(x,\th)) 
\partial_\th \mathcal{K}(x,\th) \notag \\ & + \partial_\th \mathcal{K}(\F^*_{x,\th}) 
\partial_y\overline{h}_*^{\geq P}(x,\mathcal{K} (x,\th)) \partial_{\th}\mathcal{K}(x,\th) \notag \\
& +
\partial_\th \mathcal{K}(\F^*_{x,\th})
-\partial_{\th}\mathcal{K} (x,\th + \omega)+ \OO(\|x\|^{N+1}) \label{eq:KKappendix} .
\end{align}
If we assume $\partial_\th \mathcal{K}(x,\th) = \OO(\|x\|^m)$
with $m<N$, then the right hand side of \eqref{eq:KKappendix} has order $\min \{N+1,m+1\}$ with respect to $x$. Since we are assuming $\K$ exists,   
$\partial_\th \mathcal{K}$ has zero average and therefore the right hand side of~\eqref{eq:KKappendix} should have zero average. By Theorem~\ref{thm:smalldivisors}, $\partial_\th \mathcal{K}$ has to have order $m+1$, which is a contradiction. Hence $m\ge N$. 
So we conclude that $\partial_\th \mathcal{K}(x,\th) = \OO(\|x\|^{N+1})$. Therefore, we can write $\mathcal{K}(x,\th) = \KK^\leq (x) + \OO(\|x\|^{N+1})$ and  the invariance condition becomes
$$
 \overline{g}_*^{\geq M}(x,\mathcal{K}^\leq (x)) =
 \int_{0}^1 D \mathcal{K}^{\leq }(x+ s \overline{f}_*^{\geq N}(x,\mathcal{K}^\leq (x))) \overline{f}_*^{\geq N}(x,\mathcal{K}^\leq (x))\, ds   + \OO(\|x\|^{N+1}).
$$
We decompose $\overline{g}_*^{\geq M}(x,y)
= \overline{g}_*^{\geq M}(x,0) +[\overline{g}_*^{\geq M}(x,y)-\overline{g}_*^{\geq M}(x,0)]=:
 g_1(x) + g_2(x,y)y$ and
we denote $M_1$ the order of $g_1$ and  $M_2-1$ the order of $g_2$. If $M_1 > M_2=M$, $\overline{g}_*^{M}(x,y)=g_2(x,y)y$ and satisfies $\overline{g}_*^M(x,0)=0$. In the other case, $M=M_1 \leq M_2 $, we have
$$
g_1(x)  = -g_2(x,\K^\le(x))\K^\le(x) +
D\mathcal{K}^{\leq}(x) \overline{f}_*^{\geq N}(x,\mathcal{K}^\leq (x))
+ \OO(\|x\|^{2N}) +\OO(\|x\|^{N+1})
$$
and this implies that $N\ge M=M_1\geq \min\{M_2+1,N+1,2N \} $ which provides a contradiction that comes from assuming that $M_1\leq M_2$. 

\section{Proof of Corollary~\ref{cor:rootsunity}}\label{proof_corollary}

We first prove that 
\begin{equation}\label{stablesetsFG_1}
\bigcup_{j=0}^{\ell -1} G^j_\lambda(W_{\Vext_{\rho,\beta}}^{\mathrm{s}}(F_\lambda)) \subset W_{\Bext_{\rho,\beta}}^{\mathrm{s}}.
\end{equation}
Take 
$(x,y,\th) \in G_\lambda^j (W_{\Vext_{\rho,\beta}}^{\mathrm{s}}(F_\lambda))$. We have that  
$
(\bar{x},\bar{y},\bar{\th}):=G^{-j}_\lambda(x,y,\th) \in W_{\Vext_{\rho,\beta}}^{\mathrm{s}}(F_\lambda).
$
For all $l\in \N$, there exist $p,q\in \N$, $0\leq p \leq \ell-1$ such that $j+l=q\ell +p$. Then, 
$$
G_{\lambda}^l(x,y,\th) = G_\lambda^{j+l} (\bar{x},\bar{y},\bar{\th}) = G_\lambda^{q \ell + p}(\bar{x},\bar{y},\bar{\th}) =G_\lambda^p \big (F_\lambda^q(\bar{x},\bar{y},\bar{\th} )\big ) \in 
G_\lambda^p (\Vext_{\rho,\beta} ) \subset \Bext_{\rho,\beta}.
$$
Moreover
$$
\| (G^{l}_\lambda)_{x,y}(x,y,\th) \| =  \left \|\left (G_\lambda^p (F_\lambda^q (\bar{x},\bar{y},\bar{\th})\right )_{x,y} \right \| \leq \Mg \left \|(F_\lambda^q)_{x,y} (\bar{x},\bar{y},\bar{\th}) \right \| \to 0 \qquad \text{as   }\;\; q\to \infty.
$$
Therefore, since $q\to \infty$ if and only if $l\to \infty$, $(x,y,\th)\in W^{\mathrm{s}}_{\Bext_{\rho,\beta}}$.

Then, by Theorem~\ref{th:existencemap}
$$
\mathcal{W} \subset \bigcup_{j=0}^{\ell -1} G_\lambda^j (W^{\mathrm{s}}_{\Vext_{\rho,\beta}} (F_\lambda) \subset W_{\Bext_{\rho,\beta}}^{\mathrm{s}}
$$
and the first claim of Corollary~\ref{cor:rootsunity} is proved. 

Assume now that $B_g>0$, then, by Theorem~\ref{th:existencemap} we have the properties in~\eqref{KequalWsmap}:
\begin{equation}\label{KequalWsmap_proof}
K(\wh V_{\rho}\times \T^d,\lambda)  =  W^{\mathrm{s}}_{\wh \Vext_{\rho,\beta}} (F_\lambda) \qquad \mathrm{and} \qquad 
W^{\mathrm{s}}_{\wh \Vext_{\rho,\beta}} (F_\lambda)=  \bigcap_{k\geq 0} F_\lambda^{-k} (\wh \Vext_{\rho,\beta} ),
\end{equation}
where we recall that $\wh \Vext_{\rho,\beta}= \wh V_{\rho,\beta} \times \mathbb{T}^d$ where $\wh V$ is a slightly smaller cone contained in $V$. To avoid cumbersome notations, we skip the symbol \ $\wh{}$ \ in our notation. To prove the last part of the result, by~\eqref{stablesetsFG_1}  and~\eqref{KequalWsmap_proof}, we only need to check that 
\begin{equation}\label{stablesetsFG_2}
W_{\Bext_{\rho,\beta}}^{\mathrm{s}}  \subset  \bigcup_{j=0}^{\ell -1} G^j_\lambda(W_{\Vext_{\rho,\beta}}^{\mathrm{s}}(F_\lambda))
\end{equation}
because, if~\eqref{stablesetsFG_2} holds true, then by~\eqref{stablesetsFG_1} and~\eqref{KequalWsmap_proof}, 
$$
W_{\Bext_{\rho,\beta}}^{\mathrm{s}} =  \bigcup_{j=0}^{\ell -1} G^j_\lambda(W_{\Vext_{\rho,\beta}}^{\mathrm{s}}(F_\lambda)) = \bigcup_{j=0}^{\ell -1} G^j_\lambda \big (K(  V_\rho \times \mathbb{T}^d,\lambda) ) = \mathcal{W}.
$$

Next we  prove~\eqref{stablesetsFG_2}. 
We first observe that, since $G_\lambda, F_\lambda$ are local diffeomorphisms, we have that 
\begin{equation}\label{propunion}
G^l_\lambda(W_{\Vext_{\rho,\beta}}(F_\lambda) )=W_{G^l_\lambda(\Vext_{\rho,\beta}\times \T^d)}^{\mathrm{s}}(F_\lambda), \qquad l\in \mathbb{Z}. 
\end{equation}
Now we notice that, if for some $j\in \{0,\cdots,\ell-1\}$
$$
G^{j}_\lambda (\Vext_{\rho,\beta} ) =\bigcup_{i\neq j} G^{i}_\lambda( \Vext_{\rho,\beta})
$$
then~\eqref{stablesetsFG_2} holds true and the proof is complete in this case. Indeed, in this case $\Bext_{\rho,\beta}=G^{j}_\lambda (\Vext_{\rho,\beta} )$. Therefore, if $(x,y,\theta) \in W_{\Bext_{\rho,\beta}}^{\mathrm{s}}$, then, for all $l\in \mathbb{N}$,
$G_\lambda^{l} (x,y,\theta) \in G^{j}_\lambda (\Vext_{\rho,\beta} )$ and, in particular, $F_\lambda^{l} (x,y,\theta) \in G^{j}_\lambda (\Vext_{\rho,\beta} )$ for all $l\in \mathbb{N}$. From the second identity in~\eqref{KequalWsmap_proof} and~\eqref{propunion}, we conclude that $(x,y,\theta) \in W^{\mathrm{s}}_{G^{j}_\lambda (\Vext_{\rho,\beta} )} (F_\lambda)  = G^j_\lambda \big (W^{\mathrm{s}}_{\Vext_{\rho,\beta}} (F_\lambda) \big )$ and~\eqref{stablesetsFG_2} follows trivially. 

From the previous arguments, we now assume that the set $\Bext_{\rho,\beta}$ can be rewritten as 
$$
\Bext_{\rho,\beta}=\bigcup_{j=0}^{\ell -1 } B_j,\qquad B_j= G^{j}_\lambda (\Vext_{\rho,\beta} ) \backslash \left \{\bigcup_{i\neq j} G^{i}_\lambda( \Vext_{\rho,\beta})
\right \} \neq \emptyset.
$$
We notice that $B_j\cap B_i= \emptyset$ if $i\neq j$.

Let $(x,y,\th) \in W_{\Bext_{\rho,\beta}}^{\mathrm{s}} \cap B_0$. It is clear that $G^l_\lambda(x,y,\th) \in G^{l}_\lambda(B_0)$ if $l\leq \ell -1$ and since the only set $B_j$ with non-empty intersection with $G^l_\lambda (B_0)$ is $B_l$, then $G^l_\lambda(x,y,\th) \in B_l$. In addition, 
$$
G^{\ell }_\lambda (x,y,\th) \in G_\lambda(B_{\ell-1})=
G^{\ell }_\lambda (\Vext_{\rho,\beta} ) \backslash \left \{\bigcup_{i=1 }^{\ell -1} G^{i}_\lambda(\Vext_{\rho,\beta})
\right \}.
$$
Since $B_j\cap B_{i}= \emptyset$ and $G^{\ell}_\lambda(x,y,\th) \notin G^{i}_\lambda(\Vext_{\rho,\beta})$ for $i=1,\cdots,\ell-1$, we conclude that $G^{\ell }_{\lambda} (x,y,\th) \in B_0$. By induction, we prove that if $(x,y,\th) \in B_0   $, $F^q_\lambda(x,y,\th)=G^{q\ell}_\lambda(x,y,\th) \in B_0 $. Therefore, $(x,y,\th) \in W_{B_0}^{\mathrm{s}}(F_\lambda) \subset W_{\Vext_{\rho,\beta}}^{\mathrm{s}}(F_\lambda)$. 

When $(x,y,\th) \in W_{\Bext_{\rho,\beta}}^{\mathrm{s}} \cap B_j$, reasoning in an analogous way as for $j=0$, we conclude that $(x,y,\th) \in W_{B_j}^{\mathrm{s}}(F_\lambda) \subset W_{G^j_\lambda (\Vext_{\rho,\beta})}^{\mathrm{s}}(F_\lambda)$ and by property~\eqref{propunion} the proof of~\eqref{stablesetsFG_2} is complete.

\section{Proof of Lemma~\ref{lem:calR} } \label{apendixA}

We first recall that for $z\in \mathbb{C}^l$ we use the norm $\|z\|=\max (\|\Re z\|, \|\Im z\|)$. In addition, by definition of the complex set $\Omega_{\rho}(\gamma)$, $\|\Im z\| \leq \gamma \| \Re z\|$ and therefore $\|z\|= \| \Re z\|$ if $\gamma \leq 1$. As a consequence, if we consider the definition of the values $a_f,b_f,A_f,D_f $ and $B_g$ in~\eqref{defa}, \eqref{defbCf} and~\eqref{defABg} 
with $x$ belonging to $\Omega_\rho(\gamma) $ instead of $V_\rho$,
they change by a quantity of order $\gamma$, 
provided $\gamma $ is small enough. Since all the conditions on these constants are open conditions we can choose $\gamma$ small enough such that those properties still hold true. 

We also recall that, $\widecheck{R}_v(v)=v+ \overline{f}^N(v,0)+ w^{\ge N+1}(v)$, with $ w^{\ge N+1}(v)= \OO(\|v\|^{N+1})$. 

The two first items in Lemma~\ref{lem:calR} has been proven in previous works~\cite{BFM2020a,BFM2020b,BFM20}. Then,  we sketch a simple proof of them. 
The first item relies on the invariance by $\widecheck{R}_v$ of the set $\Omega_\rho(\gamma)$. To do so,  the following technical lemma, which is a straightforward consequence of Taylor's theorem, is used.  
\begin{lemma}\label{lem:realanalytic} Let $0<\rho,\gamma\leq 1$. 
If $\chi:\Omega_\rho(\gamma) \subset \C^n \to \C^n$ is a real analytic function, satisfying $\chi(v)=\OO(\|v\|^k)$, then
\begin{align*}
\chi(v) = \chi(\Re v) + i\int_{0}^1 D\chi(\Re v + is  \Im v) \Im v \, ds  
=\chi(\Re v) +i D \chi (\Re v) \Im v + \gamma^2 \OO(\|\Re v\|^{k}).
\end{align*}
\end{lemma}
We fix $a,b,A$ satisfying~\eqref{defabACanalytic}, namely $a<a_f$, $b>b_f$ and $A<A_f$. Recall that $a_f\le  b_f$.
Let $v\in \Omega_\rho(\gamma) $. 
We are going to check that $\Re  \widecheck{R}_v(v) \in V_\rho$ and 
$\|\Im \widecheck{R}_v(v)\| \leq \gamma \|\Re \widecheck{R}_v(v)\| $.
On the one hand, by hypothesis (v) on $\overline{f}^N$ and Lemma~\ref{lem:realanalytic} we have that, if $\gamma$ is small, 
$$
\text{dist} \big (\Re \widecheck{R}_v(v), V_\rho^c \big ) \geq \text{dist}\big (\Re v +   \overline{f}^{N}(\Re v,0) ,V_\rho^c \big ) - \Mg \gamma \|\Re v\|^{N}  \geq \frac{a_V}{2} \|\Re v\|^{N}.
$$
On the other hand, if $v\in \Omega_\rho(\gamma)$ with $\gamma\leq 1$, using again  Lemma~\ref{lem:realanalytic}, and that $\|v\| = \| \Re v \| $, we obtain 
\begin{align*}
\|\Im \widecheck{R}_v(v)\| &\leq \| \Im v \| \big ( \| \Id + D\overline{f}^N (\Re v,0)\| + \Mg \gamma \|v\|^{N-1} + \Mg \| v\|^{N} \big ) \\ & \leq 
\gamma \| \Re v\| (1 - (A_f- \Mg \gamma - \Mg \rho) \|v\|^{N-1}).
\end{align*}
Using similar arguments we can see that  $\| \Re \widecheck{R}_v(v)\| \geq \| \Re v \| (1 - (b_f+ \Mg \gamma + \Mg \rho) \|v\|^{N-1})$. Then, to check that 
$\|\Im \widecheck{R}_v(v)\| \leq \gamma \| \Re \widecheck{R}_v(v)\|$, it is  sufficient to check that 
$$
b_f + \Mg (\gamma + \rho) < b < A< A_f - \Mg (\gamma+ \rho) 
$$
which is satisfied if $b<A$ and $\rho,\gamma$ are small enough. This proves that  $\Omega_\rho(\gamma) $ is invariant by $\widecheck{R}_v$.

To prove~\eqref{propcalRlema} in the second item of Lemma~\ref{lem:calR}, we note that there exist $\rho,\gamma$ small enough such that if $v\in \Omega_\rho(\gamma)$,
\begin{equation}\label{appendixpropa}
\|\widecheck{R}_v(v) \| \leq \|v\|- a_f \|v\|^{N} + \Mg \|v\|^{N+1} \leq \|v\| (1- a \|v\|^{N-1}),
\end{equation}
and
\begin{equation}\label{appendixpropb}
\|\widecheck{R}_v(v) \| \geq \|v\| - b_f \|v\|^{N}- \Mg \|v\|^{N+1} \geq \|v\|(1-b\|v\|^{N-1}).
\end{equation}
Analogously,
\begin{equation}\label{appendixpropA}
\|D\widecheck{R}_v(v)\|\leq 1- A \|v\|^{N-1}. 
\end{equation}
Then, since
$$
\|v\|(1-b\|v\|^{N-1}) \leq  \| \widecheck{R}_v(v)\|\leq \|v\|(1-a\|v\|^{N-1}),
$$
taking $a^*  <a(N-1)$, $b^*  >b(N-1)$ and $\rho,\gamma$ small enough, it is clear that
\begin{equation} \label{cotessupinf}
\frac{\|v\|}{\big [ 1 + b^*  \|v\|^{N-1}]^{\frac{1}{N-1}}} \leq \| \widecheck{R}_v(v)\|\leq \frac{\|v\|}{\big [ 1 + a^*  \|v\|^{N-1}\big ]^{\frac{1}{N-1}}}, \qquad v\in \Omega_\rho(\rho,\gamma) .
\end{equation}
Introducing the map
$
\mathcal{R}_{\mathbf{c}} (\xi) = \xi \big [ 1+ \mathbf{c} \xi^{N-1}\big ]^{-\frac{1}{N-1}},
$
with $\mathbf{c}>0$, \eqref{cotessupinf} can be rewritten as
$$
\mathcal{R}_{b^*}  (\|v\|) \leq  \| \widecheck{R}_v(v)\| \leq  \mathcal{R}_{a^* } (\|v\|).
$$
On the other hand, the flow $\varphi(t,w)$ of the differential equation $\dot{w}=-\frac{\mathbf{c}}{N-1} w^N $ is
$$
\varphi(t,w)= \frac{w}{\big [1 + t \mathbf{c} w^{N-1}\big ]^{\frac{1}{N-1}} }.
$$
Clearly, by induction on $k$, $\mathcal{R}^k_{\mathbf{c}} (\|v\|)=\varphi(k,\|v\|)$ for all $k\ge 0$. 
Since $\mathcal{R}_{a^* } $ and $\mathcal{R}_{b^*}  $ are increasing functions and $\Omega_\rho(\gamma)$ is invariant by $\widecheck{R}_v$, using again induction on $k$ we prove~\eqref{propcalRlema}.

In order to prove items (3) and (4) of Lemma~\ref{lem:calR}, we first need some estimates on $D\widecheck{R}_v$ and $ D^2 \widecheck{R}_v$. 

\begin{lemma}\label{Lem:R2}  
Let $a,b$ and $A$ satisfy~\eqref{appendixpropa}, \eqref{appendixpropb} and~\eqref{appendixpropA} with  $A>b$. Let also $1<\ell <A/b$ and $b^*   = \frac{(N-1)A}{\ell}$. Then,  there exist  $\rho,\gamma$ small enough and a constant $\Mg>0$ such that
for all $v \in \Omega_\rho(\gamma) $ and $k\geq 1$
\begin{align}
\|D\widecheck{R}_v^k(v )\| &\leq
\prod_{l=0}^{k-1} \|D\widecheck{R}_v (\widecheck{R}_v^l(v ))\|
\leq \frac{1}{\big [1 + kb^*   \|v\|^{N-1} \big ]^{\frac{\ell}{N-1} }}, \label{boundDRk}
\\ \|D^2 \widecheck{R}_v^k (v )\| & \leq \Mg
\frac{1}{\|v\| \big [ 1 + k b^*   \|v\|^{N-1}\big ]^{\frac{\ell}{N-1}}}. \notag
\end{align}
In addition,
\begin{equation}\label{bound:imRk}
\|\Im \widecheck{R}_v^k (v )\| \leq  \frac{\|\Im v \|}{\big [1 + k b^*  \|v\|^{N-1}\big ]^{\frac{\ell}{N-1}}}.
\end{equation}
\end{lemma}

\begin{proof}
By the chain rule and \eqref{appendixpropA}, if $v\in \Omega_\rho(\gamma)$,   
$$
\|D\widecheck{R}_v^k(v)\| \leq \prod_{l=0}^{k-1} \|D\widecheck{R}_v (\widecheck{R}_v^l(v))\| \leq  \prod_{l=0}^{k-1} (1 - A \| \widecheck{R}_v^l(v)\|^{N-1}).
$$
Now we bound the logarithm of the product. Since $b^*   > b(N-1)$, using property~\eqref{propcalRlema} we obtain
\begin{align*}
\sum_{l=0}^{k-1} \log \big ( 1- A\|\widecheck{R}_v^l(v)\|^{N-1}\big ) & \leq - A \sum_{l=0}^{k-1} \|\widecheck{R}_v^l(v)\|^{N-1}
\leq  -A \|v\|^{N-1} \sum_{l=0}^{k-1} \frac{1}{1+ l b^*  \|v\|^{N-1}} 
\\
&\leq - \frac{A}{b^*  }  \log \big ( 1+ k b^*   \|v\|^{N-1}) .
\end{align*}
Therefore, 
$$
\|D\widecheck{R}_v^k(v)\|
\leq \frac{1}{\big [1 + k b^*   \|v\|^{N-1}\big ]^{A/b^*  }}.
$$
Finally, since $\frac{A}{b^*  } = \frac{\ell}{N-1}$, property~\eqref{boundDRk} is proven. 

Now, we deal with the bound for $\|D^2 \widecheck{R}_v^k(v)\|$. We have that
$$
\|D^2 \widecheck{R}_v^k(v)\| \leq \sum_{m=0}^{k-1} \|D^2 \widecheck{R}_v (\widecheck{R}_v^m(v))\| \|D\widecheck{R}_v^m(v)\|
\prod_{l=0}^{k-1}\|D\widecheck{R}_v(\widecheck{R}_v^l(v))\| \|D\widecheck{R}_v (\widecheck{R}_v^m(v))\|^{-1}.
$$
We recall that $a^*  <a(N-1)$ and $a^*  < b^*  $. Using that $\|D\widecheck{R}_v(\widecheck{R}_v^m(v))\|\geq 1-\Cc\rho ^{N-1}$ for all $m\in \mathbb{N}$, that $\|D^2 \widecheck{R}_v(v)\|\leq \Mg \|v\|^{N-2}$, \eqref{boundDRk} and~\eqref{propcalRlema}:
\begin{align*}
\|D^2 \widecheck{R}_v^k& (v)\| \leq \Mg
\prod_{l=0}^{k-1} \|D\widecheck{R}_v (\widecheck{R}_v^l(v))\|
\sum_{m=0}^{k-1} \|\widecheck{R}_v^m(v)\|^{N-2} \|D\widecheck{R}_v^m(v)\| \\
& \leq \Mg \|v\|^{N-2} \frac{1}{\big [1 + kb^*   \|v\|^{N-1} \big ]^{\frac{\ell}{N-1}}} \sum_{m=0}^{k-1}
\frac{1}{\big [1+ m a^*   \|v\|^{N-1}\big]^{\frac{N-2}{N-1}}}
\frac{1}{\big [1+ m b^*  \|v\|^{N-1}\big]^{\frac{\ell}{N-1}}} \\
& \leq \Mg  \|v\|^{N-2} \frac{1}{\big [1 + kb^*   \|v\|^{N-1} \big ]^{\frac{\ell}{N-1} }}  \sum_{m=0}^{k-1}
\frac{1}{\big [1+ m a^*   \|v\|^{N-1}\big]^{\frac{N-2 + \ell}{N-1}}}.
\end{align*}
Then, since $\ell>1$, the sum above converges when $k\to \infty$ and we conclude that
$$
\|D^2 \widecheck{R}_v^k(v)\| \leq \Mg \frac{1}{\|v\|} \frac{1}{\big [1 + kb^*   \|v\|^{N-1} \big ]^{\frac{\ell}{N-1} }}.
$$

To finish the proof of this lemma we prove~\eqref{bound:imRk}. By Lemma~\ref{lem:realanalytic},
$$
\| \Im \widecheck{R}_v^k(v) \|\leq  \|\Im v \|\int_{0}^1 \| D \widecheck{R}_v^k(\Re v +i s \Im v) \| \, ds.
$$
Then, from the fact $\|\Re v +i s \Im v\|= \max \{ \|\Re v\|, s \|\Im v\|\} = \|\Re v \|=\|v\|$, using~\eqref{boundDRk} for
$\|D\widecheck{R}_v^k(\Re v +i s \Im v)\|$, we obtain the result.
\end{proof}
\begin{remark}
When $n=1$ one can further check that $\Im \widecheck{R}_v^k (v)\cdot \Im v \geq 0$ and that for $a^* <a(N-1)$ and $\ell>N$,
$$
|\Im \widecheck{R}_v^k(v )|\geq  \frac{|\Im v|}{\big [1 + k a^*  \|v\|^{N-1}\big ]^{\frac{\ell}{N-1}}}.
$$
Indeed, when $n=1$, $\widecheck{R}(v)=v-av^N + \mathcal{O}(|v|^{N+1})$.
Then, $\Im \widecheck{R}(v)=\Im v (1- a \mathcal{O}(|v|)^{N-1})$
and it is clear that, 
if $\Im v$ is small,
$\Im R^k(v)$ and $\Im v$ have the same sign.

To prove the lower bound for $\Im \widecheck{R}^k(v)$ we use that, for any $B>a N$, taking $\gamma,\rho$ small enough
$$
|\Im \widecheck{R}(v)|\geq |\Im x | (1- B |v|^{N-1}), \qquad x \in \Omega(\gamma,\rho).
$$
Therefore,
$$
|\Im \widecheck R^k(v)|\geq |\Im v | \prod_{l=0}^{k-1} (1-B|\widecheck R^l(v)|^{N-1}).
$$
As we did in the proof of Lemma~\ref{Lem:R2}, we consider the logarithm of the last product:
$$ 
\sum_{l=0}^{k-1} \log \big (1-B|\widecheck R^l(v)|^{N-1}\big )
 \geq - \frac{B}{a^* }  \log (1+ a^* k) |v|^{N-1}) .
$$   
Take $\wh a^* <a^* $ and $\rho$ small enough such that  
$$ 
|\Im \widecheck R^k(v)| \geq \frac{|\Im v| }{\big [1+ a^*  k\|v\|^{N-1} \big]^{B/\wh a^*}}.
$$
Since the choice of $B, a^* , \wh a^*$ can be done arbitrarily close to $Na,  a(N-1), a^* $ and $B/\wh a^*>N /(N-1) $ the proof is finished.
\end{remark}

Next, we prove property~\eqref{propImRlema} in the third item of Lemma~\ref{lem:calR}. 
Recall that  $\widecheck{R}_\psi(v,\psi)=\omega+\psi + R_\psi(v)$
with  $R_\psi(v)=\OO(\|v\|^P)$. By Lemma~\ref{lem:realanalytic} one has that
$$
\| \Im R_\psi (v) \| \leq  \| \Im v \| \int_{0}^1 \| D R_\psi(\Re v+ is \Im v) \|  \, ds \leq \Mg  \| \Im v \|  \| v \|^{P-1}.
$$
Let $\ell$ be such that $\max\{1,N-P\}<\ell <A/b$. Then, using~\eqref{propcalRlema} and Lemma~\ref{Lem:R2}:
\begin{equation*}
\begin{aligned}  
\sum_{j=0}^\infty \|\Im R_\psi(\widecheck{R}_v^{j} (v))\| & \leq  \Mg \sum_{j=0}^\infty \|\Im \widecheck{R}_v^{j} (v)\| \| \widecheck{R}_v^j (v)\|^{P-1} \\
&\leq \Mg  \|\Im v\| \|v\|^{P-1} \sum_{j=0}^\infty \frac{1}{\big [ 1 + j a^* \|v\|^{N-1}\big ]^{\frac{\ell+ P-1}{N-1}}}
\\
&\leq \Mg \frac{\|\Im v\|}{\|v\|^{N-P}},
\end{aligned}
\end{equation*}
where we have used that 
$a<b$ and $\ell+ P-1>N-1$.

Finally, for item (4) let $(v,\psi)\in \Gamma_\rho(\gamma,\sigma)$. We have already seen that  $\widecheck{R}_v(v)\in \Omega_\rho(\gamma)$.
It remains to prove that $\widecheck{R}_\psi(v,\psi )$  satisfies the condition of the definition of the set 
$\Gamma_\rho(\gamma,\sigma)$. We have
\begin{align*}
\|\Im \widecheck{R}_\psi(v,\psi ) \| + \sum_{l=0}^\infty \|\Im R_\psi(\widecheck{R}_v^{l+1} (v))\|
=&  \| \Im (\psi + R_\psi (v)) \|+ \sum_{l=0}^\infty \|\Im R_\psi(\widecheck{R}_v^{l+1} (v)) \| 
\\  \leq &\|\Im \psi\| + \sum_{l=0}^\infty \|\Im R_\psi(\widecheck{R}_v^{l } (v)) \| 
< \sigma 
\end{align*}
so that $\Gamma_\rho(\gamma,\sigma)$ is invariant by $ \widecheck{R} $.
This finishes the proof of Lemma \ref{lem:calR}.
 
\bibliography{references}

\end{document}